\documentclass[preprint,11pt]{elsarticle}
\journal{}
\usepackage{graphicx}
\usepackage{pifont,latexsym,ifthen,amsthm,rotating,calc,textcase,booktabs}
\usepackage{amsfonts,amssymb,amsbsy,amsmath,extarrows,bm}
\allowdisplaybreaks
\newtheorem{theorem}{Theorem}[section]
\newtheorem{lemma}[theorem]{Lemma}
\newtheorem{corollary}[theorem]{Corollary}

\newtheorem{proposition}[theorem]{Proposition}

\newtheorem{Dbar}{$\bar{\partial}$-Problem}[section]
\numberwithin{equation}{section}
\usepackage[dvipsnames]{xcolor}
\usepackage{natbib}
\newtheorem{RHP}{RH problem}[section]
\usepackage{todonotes}
\usepackage{float}
\usepackage{subfigure}
\usepackage{verbatim}
\usepackage{hyperref}
\hypersetup{
    colorlinks=true, 
    linktoc=all,     
    linkcolor=blue,  
    citecolor=blue
}

\DeclareMathOperator*{\res}{Res}
\DeclareMathOperator*{\im}{Im}
\DeclareMathOperator*{\re}{Re}
\usepackage{tikz}
\begin{document}



\begin{frontmatter}
\title{ The Cauchy problem for the Degasperis-Procesi Equation: Painlev\'e  Asymptotics  in    Transition Zones}


\author[inst2]{Zhaoyu Wang}

\author[inst3]{Xuan Zhou}

\author[inst2]{Engui Fan$^{*,}$  }

\address[inst2]{ School of Mathematical Sciences and Key Laboratory of Mathematics for Nonlinear Science, Fudan University, Shanghai, 200433, P. R. China}
\address[inst3]{College of Mathematics and Systems Science, Shandong University of Science and Technology, Qingdao, 266590,  P. R. China\\
* Corresponding author and e-mail address: faneg@fudan.edu.cn  }





\begin{abstract}
  \baselineskip=15pt
The Degasperis-Procesi (DP) equation
\begin{align}
    &u_t-u_{txx}+3\kappa  u_x+4uu_x=3u_x u_{xx}+uu_{xxx},  \nonumber
  \end{align}
serving as an asymptotic approximation for the unidirectional propagation of shallow
water waves, is an integrable model of the Camassa-Holm type and admits a  $3\times3$ matrix Lax pair.
  In our  previous work,
  we obtained the  long-time asymptotics of
  the solution $u(x,t)$  to the Cauchy problem for the DP equation   in the solitonic  region   $\{(x,t): \xi>3    \}
  \cup \{(x,t):  \xi<-\frac{3}{8}  \}$    and the solitonless  region  $\{(x,t):  -\frac{3}{8}<\xi< 0   \} \cup \{(x,t):  0\leq \xi <3  \}$ where $\xi:=\frac{x}{t}$.
In this paper, we derive
 the leading order approximation to the solution $u(x,t)$ in terms of the  solution for the  Painlev\'{e} \uppercase\expandafter{\romannumeral2} equation
  in two  transition zones $\left|\xi+\frac{3}{8}\right|t^{2/3}<C$ and $\left|\xi -3\right|t^{2/3}<C$ with $C>0$  lying between the solitonic region and  solitonless region.
Our results are established  by performing the $\bar \partial$-generalization of the Deift-Zhou nonlinear steepest descent method and applying a double scaling limit technique to an associated  vector Riemann-Hilbert problem.

\end{abstract}

\begin{keyword}
Degasperis-Procesi equation \sep Riemann-Hilbert problem \sep $\bar{\partial}$-steepest descent analysis \sep
 Painlev\'{e} transcendents\vspace{2mm}

  \end{keyword}
\end{frontmatter}
\tableofcontents

\section{Introduction and Main Results}
In this paper, we investigate the Painlev\'e asymptotics for  the solution to the  Cauchy problem  of   the
 Degasperis-Procesi (DP) equation
\begin{align}
    &u_t-u_{txx}+3\kappa u_x+4uu_x=3u_x u_{xx}+uu_{xxx}, \quad x\in\mathbb{R}, \ t>0, \label{DP}\\
    &u(x,0)=u_{0}(x), \quad x \in \mathbb{R},\label{intva}
\end{align}
where
$\kappa$ is a positive constant  characterizing  the effect of the linear dispersion.
The DP equation (\ref{DP}) refers back to the work of Degasperis and Procesi, who  searched for asymptotically integrable partial differential equations \cite{AM,ADD}.
 Subsequently, it was found that the DP equation arises in the propagation of shallow water waves over a flat bed in the   moderate amplitude regime  \cite{RS,RI1, AD,BD}.

Owing to its   integrable structure and elegant mathematical properties, the DP equation \eqref{DP} has
attracted great research interest, and significant progress has been achieved in the past few years. The DP equation admits not only peakon solitons but also shock peakons \cite{H}.
The global existence and blow-up phenomena were analyzed in \cite{YZ}.  The orbital stability problem of the peaked solitons to the DP equation on the line
was  proved  by constructing a Lyapunov function  \cite{Liu}.
Moreover, the  $L^2\cap L^\infty$  orbital stability of the sum of smooth solitons
in the DP equation   was established in \cite{LLW}.  Furthermore, multi-soliton solutions of the  DP  equation  were
constructed using a   reduction procedure  for
multi-soliton solutions of  the Kadomtsev-Petviashvili hierarchy  \cite{Mu}.
      Algebro-geometric solutions  for the whole  DP hierarchy were   constructed in
\cite{HZF}.
Additionally, the KAM theory  near  an elliptic fixed point for quasi-linear Hamiltonian perturbations of the DP  equation on the circle  was   presented in \cite{FGP}.

The DP  and  Camassa-Holm (CH) equations  are only two integrable numbers
 corresponding to $b=3$ and $b=2$, respectively,  of the following  b-family equation
 \begin{align}
    &u_t-u_{txx}+b\kappa u_x+(b+1)uu_x=bu_x u_{xx}+uu_{xxx},
\end{align}
where $b$ is a constant. Despite there being many similarities between DP and CH equations \cite{RI2},  the spectral analysis of the corresponding Lax pairs
 greatly differ due to the fact that the
 CH equation admits a $2 \times 2$-matrix  spectral problem    \cite{RD}, whereas the DP equation admits a $3 \times 3$-matrix  one \cite{ADD}.
 This difference will  lead  to some new difficulties and challenges on the  inverse scattering transform (IST) and asymptotic analysis of integrable systems with  $3 \times 3$-matrix spectral problem, such as  the Novikov equation \cite{3Nov1,3Nov2}, the Boussinesq equation \cite{3Bou1,3Bou2,3Bou3,3Bou4,3Bou5}, and the Sasa-Satsuma equation \cite{3SS1,3SS2, pa6}.
Thus, although the application of the   IST  to the CH equation has been studied extensively \cite{CH1,CH2,CH3,CH4,CH5,CH6}, the implementation of the IST to DP equation has proved to be more intricate.

An inverse spectral  approach for computing the $n$-peakon solutions of the  DP equation (\ref{DP}) was presented in \cite{isth}, and then
IST   was subsequently developed for  the Cauchy problem  of the DP equation    \cite{ARIJ}.
The soliton solutions of the DP equations were constructed  using the dressing method \cite{Constin}.
Additionally,  an   alternative  Riemann-Hilbert (RH)   approach  was  proposed   to   express  the solution  of the DP equation  in terms of the solution of
 the  RH problem evaluated at a distinguished point in the spectral parameter plane.
This approach can be effectively utilized in studying the long-time behavior of the solution \cite{AD2}.
Further,  the  initial-boundary value problem for the DP equation  (\ref{DP}) on the half-line was  considered in \cite{LDP}  and  an
 explicit formula for the leading order asymptotics of the solution in the similarity
 region was obtained in \cite{MLS}.

Based on the work \cite{AD2},  we considered  the Schwarz initial data $u_0(x)$  that supports   presence of  solitons.
With three critical lines $\xi:=\frac{x}{t}=-\frac{3}{8}, \ \xi=0, \ \xi=3$, we   divide the  upper half $ (x,t)$-plane into  the following   two classes of  distinct  space-time  regions
   denoted by (see Figure \ref{xtzones})
\begin{itemize}
	\item  Solitonic regions  $\mathcal{S}:=\mathcal{S}_1\cup \mathcal{S}_2$, where

$\mathcal{S}_1= \{(x,t):  \xi>3    \}, \ \ \mathcal{S}_2=\{(x,t):  \xi<-\frac{3}{8}  \}$;
	\item  Solitonless  regions  $\mathcal{Z}:=\mathcal{Z}_1\cup \mathcal{Z}_{2}$, where
	
	 $\mathcal{Z}_1:=\{(x,t):  -\frac{3}{8}<\xi< 0  \} $ with 24 phase points without a soliton;
	
	  $\mathcal{Z}_{2}:=\{(x,t): 0\le \xi< 3  \} $  with 12 phase points without a soliton.

\end{itemize}
Furthermore, the long-time  asymptotics  and soliton resolution
 for the DP equation (\ref{DP}) in two  classes  of  regions  $\mathcal{S}$ and     $\mathcal{Z}$  were   obtained \cite{zx1}.
The remaining question is how to describe   the  asymptotics of the solution to  the Cauchy problem \eqref{DP}-\eqref{intva} in the  following   zones
near the three critical lines (see Figure \ref{xtzones})

\begin{itemize}

 \item The  transition zone $\mathcal{T}_1 :=|\xi+\frac{3}{8}|t^{2/3}<C$ (between   $\mathcal{S}_1$  and $\mathcal{Z}_1$);

 \item The transition zone $\mathcal{T}_{2}:=|\xi-3|t^{2/3}<C$ (between    $\mathcal{S}_2$ and $\mathcal{Z}_{2}$);

\item The   zone $\mathcal{T}_{3}:=|\xi|t^{2/3}<C$ with $\xi<0$ (between  $\mathcal{Z}_1$ and $\mathcal{Z}_{2}$),
which is    not    a   transition zone in the sense of  leading-order approximation. The   detailed  proof will be given  in \ref{T0}.

\end{itemize}

It is the aim of the present work to fulfill the asymptotic picture of the DP equation (\ref{DP})  by focusing on the asymptotics   in two  transition zones $\mathcal{T}_1$ and $\mathcal{T}_2$ as $t \to \infty$.
We find that  the leading order approximation to the solution $u(x,t)$  in  the  transition zones   $\mathcal{T}_1$ and $\mathcal{T}_2$  can be described  in terms of the solution of the Painlev\'e equation. Our  main  results    are  stated  as follows.

\begin{figure}[htbp]
				\centering			
				\begin{tikzpicture}
					\node[anchor=south west,inner sep=0] at (0,0) {\includegraphics[width=4.6in]{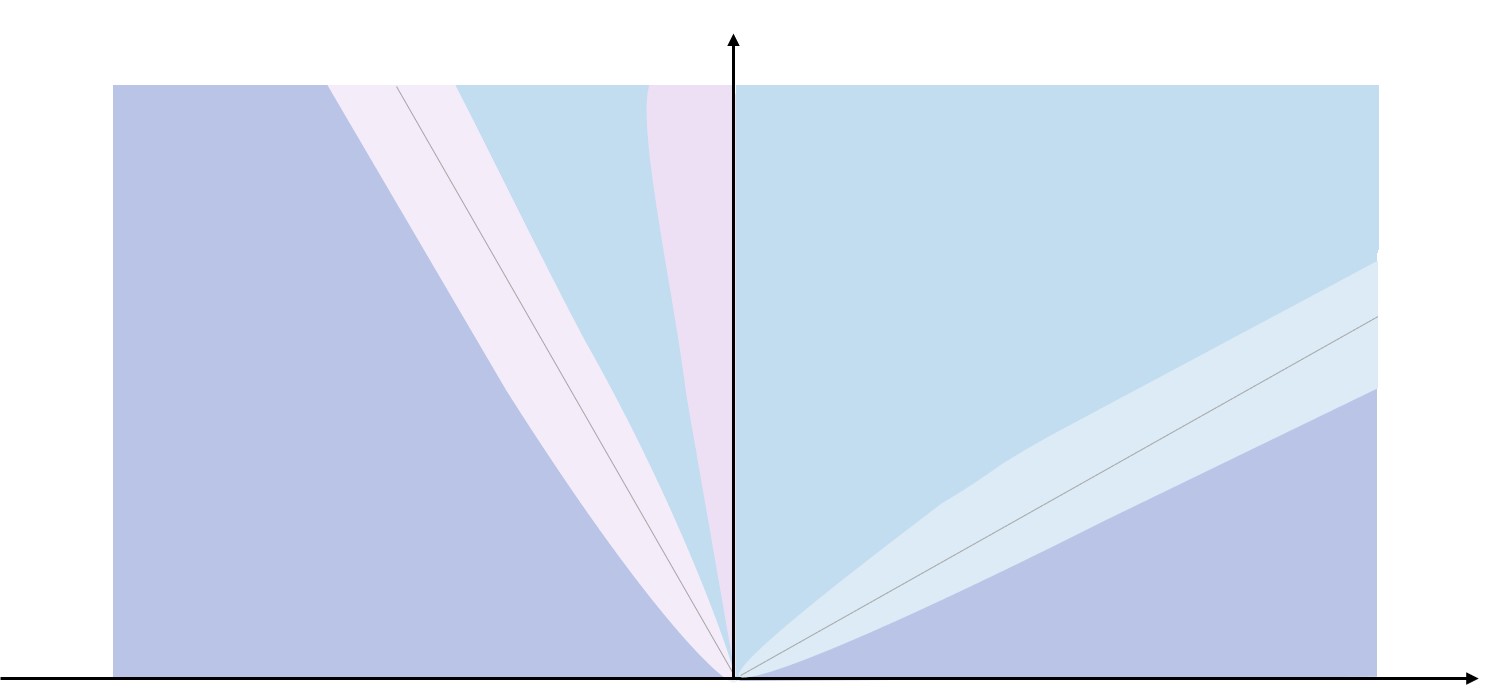}};
					\node    at (5.7,-0.1)  {$0$};
					\node    at (3.3,5.2)  {$\xi=-\frac{3}{8}$};
					\node    at (11.3,3)  {$\xi=3$};
				    \node    at (3.8,3.3)  {$\mathcal{T}_{1}$};
					\node    at (8.9,2)  {$\mathcal{T}_{2}$};
					\node    at (5.34,4) {$\mathcal{T}_{3}$};
				    \node    at (2.5,2)  {$\mathcal{S}_1$};
				    \node    at (9,1)  {$\mathcal{S}_2$};
				    \node    at (4.8,3.8)  {$\mathcal{Z}_1$};
				    \node    at (7.2,3.5)  {$\mathcal{Z}_2$};
					\node    at (11.5,0.35)  {$x$};
					\node    at (5.65,5.57)  {$t$};
				\end{tikzpicture}		
	\caption{\footnotesize   The distinct  space-time asymptotic regions in   $ (x,t)$-plane:
Solitonic regions: $\mathcal{S}_1\cup \mathcal{S}_2$;  Solitonless  regions  $\mathcal{Z}_1 \cup \mathcal{Z}_2$;  Transition zones $\mathcal{T}_{1} \cup \mathcal{T}_{2}$;
Non-transition zone  $ \mathcal{T}_3$.
		  }
\label{xtzones}
\end{figure}
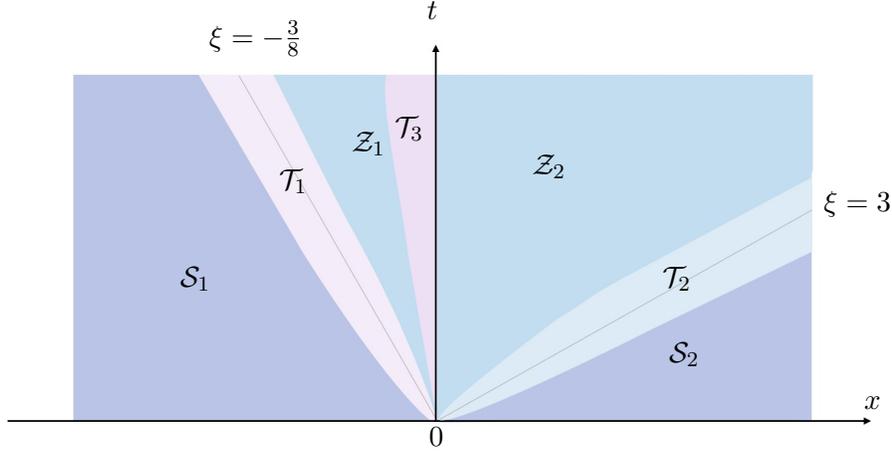



\begin{theorem}\label{th1}
	Let $u_{0}\in\mathcal{S}(\mathbb{R})$ be a function in the Schwartz class,   $r(k)$  and $\{\zeta_n \}_{n=1}^{6N}$  be the associated reflection coefficient and discrete spectrum, respectively.	
Then as $t \to \infty$, the solution $u(x,t)$ of the Cauchy problem \eqref{DP}-\eqref{intva}   satisfies the following asymptotic approximation  formula.
\begin{itemize}
	\item   
   In the first transition zone  $   \mathcal{T}_{1}$,
	\begin{align}\label{ltp1}
	u(x,t)=t^{-1/3} f_1(e^{\frac{\pi}{6}\mathrm{i}};x,t)+\mathcal{O}(t^{-2/3+4\delta_1}),
	\end{align}
	where $\frac{1}{27}<\delta_1<\frac{1}{12}$ and
	\begin{align*}
	&f_1(e^{\frac{\pi}{6}\mathrm{i}};y,t) = \frac{\partial}{\partial t} \left(\sum_{j=1}^3 \left( E_1(e^{\frac{\pi}{6}\mathrm{i}};s)_{j2} -E_1(e^{\frac{\pi}{6}\mathrm{i}};s)_{j1}\right) \right)
	\end{align*}
	with
		\begin{equation}
	s =  \frac{ 2^{2/3}(-7+\sqrt{21})}{3^{2/3} (98-21\sqrt{21})^{1/3}} \left(\frac{x}{t} +\frac{3}{8}\right) t^{\frac{2}{3}}.
	\end{equation}
	Here,  $E_1(k;s)$  is given by \eqref{p1p1} which is expressed in terms of   the unique solution $v(s)$ of the
	Painlev\'{e} \uppercase\expandafter{\romannumeral2}  equation
	\begin{equation}\label{p1pain2}
v_{ss} = 2v^3 +sv,
	\end{equation}
	characterized by
	\begin{equation}
	v(s) \sim \left| r\left(\frac{\sqrt{7}+\sqrt{3}}{2}\right)\right| {\rm Ai}(s), \quad s \to +\infty.
	\end{equation}
Moreover, $E_1(k;s)_{ji}$ represents the element in the $j$-th row and $i$-th column of the matrix $E_1(k;s)$.

	\item 
 In the second transition zone   $ \mathcal{T}_{2}$,
	\begin{align}\label{ltp2}
	&u(x,t)=t^{-1/3}f_2(e^{\frac{\pi}{6}\mathrm{i}};x,t)+\mathcal{O}(t^{-2/3+4\delta_2 }),
	\end{align}
	where $\frac{1}{27}<\delta_2<\frac{1}{12}$ and
	\begin{align*}
	&f_2(e^{\frac{\pi}{6}\mathrm{i}};x,t) = \frac{\partial}{\partial t} \left(\sum_{j=1}^3 \left( E_2(e^{\frac{\pi}{6}\mathrm{i}};s)_{j2} -E_2(e^{\frac{\pi}{6}\mathrm{i}};s)_{j1}\right) \right)
	\end{align*}
	with
	 \begin{equation}
	s = 3^{-\frac{2}{3}}\left(\frac{x}{t}-3\right)t^{\frac{2}{3}}.
	\end{equation}
	Here,  $E_2(k;s)$ is determined by \eqref{p2p1}  which is expressed in terms of the unique solution $v(s)$ of the
	Painlev\'{e} \uppercase\expandafter{\romannumeral2}  equation \eqref{p1pain2}  with
	\begin{equation}
		v(s) \sim -\left| r(1)\right| {\rm Ai}(s), \quad s \to +\infty.
	\end{equation}
	Moreover,  $E_2(k;s)_{ji}$ represents the element in the $j$-th row and $i$-th column of the matrix $E_2(k;s)$.
\end{itemize}
\end{theorem}


Let us mention some previous works on the asymptotics of other integrable systems   in   transition zones. It is worthwhile to see that the Painlev\'e transcendents plays an important role in asymptotic studies of integrable systems.
The appearance of Painlev\'e transcendents in transition zones, as well as the connection between the different zones, was first understood in the case of the Korteweg-de Vries (KdV)  equation \cite{pa1} and the modified KdV  (mKdV) equation \cite{pa2}.  Moreover, this phenomenon also appeared in the study of higher-order Painlev\'e  asymptotics for the mKdV equation \cite{Charlier2020} and its hierarchy \cite{pa4}.
Later, Painlev\'e transcendents  were  discovered in the other integrable systems, like the CH equation \cite{pa5}, Sasa-Satsuma equation \cite{pa6},
 focusing and defocusing nonlinear Schr\"{o}dinger  equations with  nonzero backgrounds \cite{Miller1,WF}, and the modified CH equation \cite{xyz}.  In addition, the Painlev\'e transcendents also appeared in  research on the small dispersion limit \cite{sl1,sll2,sll3,sl4,sl5}.

The rest of this paper is arranged as follows.
In Section \ref{2}, we focus on the IST  to  establish  a  basic matrix RH problem for $M(k)$  and a    vector RH problem for $m(k)$
associated with the Cauchy problem   \eqref{DP}-\eqref{intva}.
After two  preliminary transformations, the vector  RH problem for $m(k)$
 is found to be asymptotically equivalent to a regular RH problem for $m^{(2)}(k)$,  which can be deformed into a solvable model RH problem by  performing  the nonlinear $\bar\partial$-steepest descent approach.
In Sections \ref{7} and \ref{8}, we investigate  Painlev\'{e} asymptotics of the DP equation \eqref{DP} in transition zones  $\mathcal{T}_{1}$ and $\mathcal{T}_{2}$, respectively.
We open $\bar\partial$-lenses to construct a hybrid $\bar\partial$-RH problem, which can be decomposed into a pure RH problem and a pure $\bar\partial$-problem.
The  modified RH problem for Painlev\'e II equation and the corresponding model RH problem for transition zones are introduced in  \ref{appx}, and they  play an important role in
asymptotically  matching   the pure RH problem    with the local parametrix under the  double scaling limit technique.
Finally,  proofs  for two cases  to   Theorem \ref{th1}  are provided at the end of   Sections \ref{7} and \ref{8}, respectively.
In $\mathcal{T}_{3}$ in  \ref{T0},  we show that the leading order asymptotics in  $\mathcal{Z}_1$  matches  that in  $\mathcal{Z}_2$, and thus  there is no    transition region between them.

\section{An RH Characterization of the DP Equation}\label{RHconstruct}\label{2}

The RH problem associated to the Cauchy problem  \eqref{DP}-\eqref{intva} for the DP equation can be constructed using direct and inverse scattering
transform. Here, we provide a brief overview of the corresponding RH problem. For a detailed construction of this RH problem, please refer to \cite{AD2}.

\subsection{A basic RH problem}

It is worth noting that
without loss of generality, one can select   $\kappa=1$  in the DP equation \eqref{DP}.
Then the DP equation \eqref{DP}  admits   the   Lax pair \cite{AD2}
\begin{equation}\label{lax pair}
     \Phi_x=U\Phi, \quad \Phi_t=V\Phi,
\end{equation}
where
\begin{equation}
\begin{split}
& U=\left(\begin{array}{ccc} 0 & 1 & 0 \\ 0 & 0 & 1 \\ z^3q^3 & 1 & 0\end{array}\right), \quad
 V=\left(\begin{array}{ccc}  u_x-\frac{2}{3 } z^{-3} & -u & z^{-3}  \\ u+1 &  z^{-3} & -u \\ u_x-z^3uq^3 & 1 & -u_x+ z^{-3}\end{array}\right),\\
\end{split}
\end{equation}
and $q=(1+u-u_{xx})^{1/3}$.

Let $\lambda_{j}(z),\ j=1,2,3,$ satisfy the algebraic equation
\begin{equation}
\lambda^{3}-\lambda=z^{3},
\end{equation}
so that $\lambda_{j}(z)\sim\omega^{j}z$ as $z \rightarrow \infty$ with  $\omega=\mathrm{e}^{\frac{2\pi}{3}\mathrm{i}}$. Denoting
\begin{align}
& D(x,t) = \left(\begin{array}{ccc} q^{-1}  & 0 & 0 \\
0& 1 & 0\\
0 & 0&  q  \end{array}\right),  \ \
  P(z)= \left(\begin{array}{ccc} 1 & 1 & 1 \\
\lambda_{1}(z) & \lambda_{2}(z) & \lambda_{3}(z) \\
\lambda_{1}^{2}(z) & \lambda_{2}^{2}(z) & \lambda_{3}^{2}(z) \end{array}\right), \nonumber
\end{align}
and making a transformation
\begin{equation}\label{hatphi}
\hat{\Phi}(z)= P^{-1}(z)D^{-1}(x,t)\Phi(z),
\end{equation}
we obtain a new Lax pair
\begin{equation}
\begin{aligned}
&\hat{\Phi}_{x}-q \Lambda(z) \hat{\Phi}=\hat{U} \hat{\Phi}, \\
&\hat{\Phi}_{t}+(u q \Lambda(z)-H(z)) \hat{\Phi}=\hat{V} \hat{\Phi},
\end{aligned}
\end{equation}
where
\begin{align}
& \Lambda(z)=\operatorname{diag}\left\{\lambda_1(z), \lambda_2(z), \lambda_3(z)\right\}, \quad H(z)=\frac{1}{3z^3}I+\Lambda^{-1}(z),\nonumber\\
& \hat{U}(z;x,t)=P^{-1}(z)\left(\begin{array}{ccc} \frac{q_x}{q}&0&0 \\ 0&0&0 \\ 0&\frac{1}{q}-q&-\frac{q_x}{q} \end{array}\right)P(z),\nonumber\\
& \hat{V}(z;x,t)=P^{-1}(z)\left(\left(\begin{array}{ccc} -u\frac{q_x}{q}&0&0 \\ \frac{u+1}{q}-1&0&0 \\ \frac{u_x}{q^2}&\frac{1}{q}-1+uq&u\frac{q_x}{q} \end{array}\right)+\frac{q^2-1}{z^3}\left(\begin{array}{ccc}0&0&1\\ 0&0&0 \\ 0&0&0 \end{array}\right)\right)P(z).\nonumber
\end{align}

Introducing
\begin{equation}
Q=y(x,t)\Lambda(z)+t H(z)
\end{equation}
with $y(x,t)=x-\int_{x}^{\infty}(q(\varsigma,t)-1)\mathrm{d}\varsigma$,  then
the matrix-valued function
\begin{equation}
M(z) =\hat{\Phi} (z) \mathrm{e}^{-Q}
\end{equation}
satisfies the system
\begin{equation}\label{MLie}
\begin{aligned}
&M_{x}-\left[Q_{x}, M\right]=\hat{U}M, \\
&M_{t}-\left[Q_{t}, M\right]=\hat{V}M,
\end{aligned}
\end{equation}
 which leads to  the Volterra  integral equation
\begin{equation}\label{INTM}
M( z)=I+\int_{ \pm \infty}^{x} \mathrm{e}^{Q(x, z)-Q(\varsigma, z)}[\hat{U} M(\varsigma, z)] \mathrm{e}^{-Q(x, z)+Q(\varsigma, z)} \mathrm{d}\varsigma. \end{equation}
Since $q>0$,  the   boundedness and decay  of   the exponential factor $\mathrm{e}^{Q(x, z)-Q(\varsigma, z)}$   is
 determined by the signs of $\operatorname{Re} ( \lambda_{i}(z)- \lambda_{j}(z)), 1\leqslant i \neq j \leqslant 3.$

For convenience, we introduce a new spectral parameter $k$ such that
\begin{equation}
z(k)=\frac{1}{\sqrt{3}} k\left(1+\frac{1}{k^{6}}\right)^{1/3},
\end{equation}
  then $z(k)\sim\frac{1}{\sqrt{3}}k$ as $k\rightarrow\infty$, and
\begin{equation}
\lambda_{j}(k)=\frac{1}{\sqrt{3}}\left(\omega^{j} k+\frac{1}{\omega^{j} k}\right).
\end{equation}
The contour $\Sigma=\left\{k\mid\operatorname{Re}\lambda_{i}(k)=\operatorname{Re}\lambda_{j}(k)\right.$ for some $\left.i\neq j\right\}$ consists of six rays
\begin{equation}
l_{j}=\mathbb{R}_{+}\mathrm{e}^{\frac{\pi}{3}\mathrm{i}(j-1)}, \quad j=1,\cdots,6,
\end{equation}
which divides  the $k$-plane into six sectors
\begin{equation}
D_{j}=\left\{k\mid\frac{\pi}{3}(j-1)<\arg k<\frac{\pi}{3}j\right\}, \quad j=1,\cdots,6.
\end{equation}

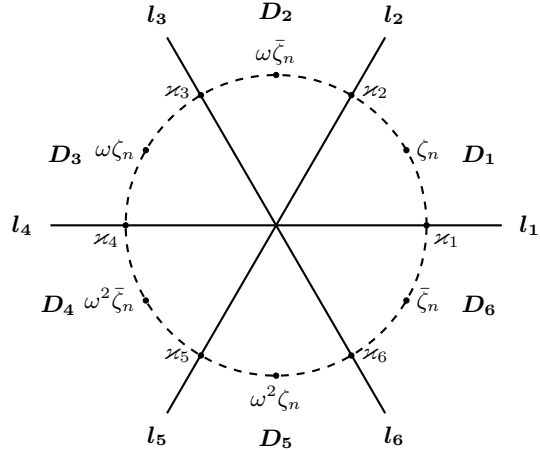
\begin{figure}[H]
\begin{center}
\begin{tikzpicture}[scale=0.5]
\draw [thick](-6,0)--(6,0);
\draw [thick](-2.9,-5)--(2.9,5);
\draw [thick](-2.9,5)--(2.9,-5);
\draw [thick, dashed] (0,0)circle(4cm);
\node  [right]  at (6.2,0) {\footnotesize$\bm{l_1}$};
\node  [left]  at (-6.2,0) {\footnotesize$\bm{l_4}$};
\node  [right]  at (2.6,5.6) {\footnotesize$\bm{l_2}$};
\node  [left]  at (-2.6,-5.6) {\footnotesize$\bm{l_5}$};
\node  [left]  at (-2.6,5.6) {\footnotesize$\bm{l_3}$};
\node  [right]  at (2.6,-5.6) {\footnotesize$\bm{l_6}$};
\filldraw[black] (4,0) circle [radius=0.07];
\filldraw[black] (2,3.464) circle [radius=0.07];
\filldraw[black] (-2,3.464) circle [radius=0.07];
\filldraw[black] (-4,0) circle [radius=0.07];
\filldraw[black] (-2,-3.464) circle [radius=0.07];
\filldraw[black] (2,-3.464) circle [radius=0.07];

\node  [right]  at (3.9,-0.4) {\footnotesize$\varkappa_1$};
\node  [right] at (2,3.55) {\footnotesize$\varkappa_2$};
\node  [left]  at (-2,3.55) {\footnotesize$\varkappa_3$};
\node  [left]  at (-3.9,-0.4) {\footnotesize$\varkappa_4$};
\node  [left]  at (-2,-3.464) {\footnotesize$\varkappa_5$};
\node  [right]  at (2,-3.464) {\footnotesize$\varkappa_6$};

\node  [below]  at (5.4,2.4) {\footnotesize$\bm{D_1}$};
\node  [below]  at (0,6.2) {\footnotesize$\bm{D_2}$};
\node  [below]  at (-5.6,2.4) {\footnotesize$\bm{D_3}$};
\node  [below]  at (-5.8,-1.6) {\footnotesize$\bm{D_4}$};
\node  [above]  at (0,-6.2) {\footnotesize$\bm{D_5}$};
\node  [below]  at (5.4,-1.6) {\footnotesize$\bm{D_6}$};

\filldraw[black] (3.464,2) circle [radius=0.07];
\filldraw[black] (0,4) circle [radius=0.07];
\filldraw[black] (-3.464,2) circle [radius=0.07];
\filldraw[black] (-3.464,-2) circle [radius=0.07];
\filldraw[black] (0,-4) circle [radius=0.07];
\filldraw[black] (3.464,-2) circle [radius=0.07];

\node  [right]  at (3.464,2) {\footnotesize$\zeta_n$};
\node  [above]  at (0,4) {\footnotesize$\omega\bar{\zeta}_n$};
\node  [left]  at (-3.464,2) {\footnotesize$\omega\zeta_n$};
\node  [left]  at (-3.464,-2) {\footnotesize$\omega^2\bar{\zeta}_n$};
\node  [below]  at (0,-4)       {\footnotesize$\omega^2\zeta_n$};
\node  [right]  at (3.464,-2) {\footnotesize$\bar{\zeta}_n$};

\end {tikzpicture}
\end{center}
\caption{The critical rays $l_j, \ j=1,\cdots,6$, analytical domains $D_j,\ j=1,\cdots,6$,  poles $ \omega^l\zeta_n, \omega^l\bar{\zeta}_n, \ n=1,\cdots,N,\ l=0,1,2,$  and singularity  points $\varkappa_j,\ j=1,\cdots,6$, in the $k$-plane.}
\label{analzone&spectrumsdis}
\end{figure}

In order to obtain an analytic matrix-valued solution from \eqref{INTM} in $\mathbb{C}\setminus\Sigma$, the initial points
of integration $\infty_{il}$ are specified for each matrix entry $(i,j)$, $1\leqslant i,j\leqslant3$ as follows
\begin{equation}
\infty_{ij}= \begin{cases}+\infty, & \text {if } {\rm Re}\lambda_{i}(k) \geqslant \operatorname{Re}\lambda_{j}(k), \\ -\infty, & \text { if } \operatorname{Re}\lambda_{i}(k)<{\rm Re} \lambda_{j}(k).\end{cases}
\end{equation}
We consider the system of Fredholm integral equations, for $1\leqslant i,j\leqslant3$,
\begin{equation}
M_{ij}(k;x,t)=I_{ij}+\int_{\infty_{ij}}^{x} \mathrm{e}^{-\lambda_{i}(k) \int_{x}^{\varsigma} q(\zeta, t) \mathrm{d}\zeta}\left[(\hat{U} M)_{ij}(\varsigma, t,k)\right] \mathrm{e}^{\lambda_{j}(k) \int_{x}^{\varsigma} q(\zeta,t) \mathrm{d}\zeta} \mathrm{d}\varsigma.
\end{equation}
Moreover, it can be shown that $M(k):=( M_{ij}(k;x,t))_{3\times 3}$ satisfies the following symmetry relations \cite{AD2}.
\begin{proposition} \label{msym}
\begin{align}
M(k)=\Gamma_1\overline{M(\bar{k})}\Gamma_1=\Gamma_2\overline{M(\omega^2\bar{k})}\Gamma_2=\Gamma_3\overline{M(\omega\bar{k})}
\Gamma_3=\Gamma_4M(\omega k)\Gamma_4^{-1}=\overline{M(\bar{k}^{-1})}, \label{S1}
\end{align}
where
\begin{equation}\label{Gamma}
    \Gamma_1=\left(\begin{array}{ccc} 0 & 1 & 0 \\ 1 & 0 & 0 \\ 0 & 0 & 1\end{array}\right),\
    \Gamma_2=\left(\begin{array}{ccc} 0 & 0 & 1 \\ 0 & 1 & 0 \\ 1 & 0 & 0\end{array}\right), \
    \Gamma_3=\left(\begin{array}{ccc} 1 & 0 & 0 \\ 0 & 0 & 1 \\ 0 & 1 & 0 \end{array}\right), \
    \Gamma_4=\left(\begin{array}{ccc} 0 & 0 & 1 \\ 1 & 0 & 0 \\ 0 & 1 & 0 \\ \end{array}\right).
\end{equation}
\end{proposition}
The limiting values of $M(k)$ satisfies the jump relation
\begin{equation}\label{jump0}
M_+(k)=M_-(k)e^{Q}V_0(k)e^{-Q}, \quad k\in\Sigma,
\end{equation}
where $V_0(k)$ is determined by the initial value $u_0$. Take $k\in\mathbb{R}^{\pm}$ as an example, $V_0(k)$ has  a special matrix structure
\begin{equation*}
V_0(k)=\left(\begin{array}{ccc} 1 & \bar{r}_{\pm}(k) & 0 \\ -r_{\pm}(k) & 1-|r_{\pm}(k)|^2 & 0 \\ 0 & 0 & 1\end{array}\right),
\end{equation*}
$r_{\pm}(k)$ are  scalar functions with $r_{\pm}(k)=\mathcal{O}(k^{-1})$ as $k\to\infty$, which together with the symmetry $r_{\pm}(k)=\overline{r_{\pm}(\bar{k}^{-1})}$ leads to $\lim\limits_{k\to 0}r_{\pm}(k)=0$. Naturally, we define the reflection coefficient as
\begin{align*}
	r(k)=\left\{ \begin{array}{ll}
		r_\pm(k),   & k\in \mathbb{R}^\pm,\\
		0  , & k=0.
	\end{array}\right.
\end{align*}
Moreover, a standard approach \cite{pa1} gives that as the initial data $u_0(x) \in \mathcal{S}(\mathbb{R})$, $r(k) \in\mathcal{S}(\mathbb{R})$.
According to \cite{AD2}, $M(k)$ has at most a finite number of simple poles lying in $\{k\in \mathbb{C}: |k|=1\}$. Note that there are two types of poles in $D_1\cap\{k\in \mathbb{C}: |k|=1\}$, we denote them as $k_j
 \ (j=1,\cdots,N_1)$ and $  k_l^A \ (l=1,\cdots,N_1^A)$ respectively. Let $N=N_1+N_1^A$ and define indicator sets as
\begin{equation}\label{index set}
\widetilde{\mathcal{N}}=\{1,\cdots,N_1\}, \quad \widetilde{\mathcal{N}}^A=\{N_1+1,\cdots,N\}, \quad \mathcal{N}=\widetilde{\mathcal{N}}\cup\widetilde{\mathcal{N}}^A.
\end{equation}
 Furthermore, denote
 \begin{align*}
	\left\{\begin{aligned}
		&\zeta_j=k_j, \quad \quad \ \ j\in\widetilde{\mathcal{N}},\\
		&\zeta_{l+N_1}=k_l^A, \quad l\in\widetilde{\mathcal{N}}^A,
	\end{aligned}\right.
\end{align*}
then for $\zeta_n, \ n\in\mathcal{N}$, it is easy to know that $\omega\bar{\zeta}_n, \ \omega\zeta_n, \ \omega^2\bar{\zeta}_n, \ \omega^2\zeta_n, \ \bar{\zeta}_n$ are also poles of $M(k)$ according to the symmetries \eqref{S1}. Correspondingly, denote
\begin{align*}
&\zeta_{n+N}=\omega\bar{\zeta}_n, \ \zeta_{n+2N}=\omega\zeta_n, \ \zeta_{n+3N}=\omega^2\bar{\zeta}_n, \ \zeta_{n+4N}=\omega^2\zeta_n, \ \zeta_{n+5N}=\bar{\zeta}_n.
\end{align*}
To sum up, the discrete spectrum can be defined as
\begin{equation}
\mathcal{K}=\{\zeta_n\}_{n=1}^{6N},
\end{equation}
whose distribution on the $k$-plane is shown in Figure \ref{analzone&spectrumsdis}.

Now, to establish the associated RH problem, we consider the jump relation \eqref{jump0}. For $i,j=1,2,3$, define
\begin{equation}
\theta_{ij}(k):=\theta_{ij}(k;\hat{\xi})=-\mathrm{i}\left[\hat{\xi}\left(\lambda_{i}(k)-\lambda_{j}(k)\right)+\left(\frac{1}{\lambda_{i}(k)}-\frac{1}{\lambda_{j}(k)}\right)\right],
\end{equation}
where $\hat{\xi}:=\frac{y}{t}$ and in Subsection \ref{proof1} we will prove
$$\hat{\xi} \sim\frac{x}{t}:=\xi,\  \ {\rm as}\ \  t \to \infty. $$
Specifically,
\begin{equation}\label{oritheta12}
\theta_{12}(k)=\left(k-\frac{1}{k}\right)\left[\hat{\xi}-\frac{3}{k^{2}-1+k^{-2}}\right],
\end{equation}
and
\begin{equation}
\theta_{13}(k)=-\theta_{12}(\omega^2k), \quad \theta_{23}(k)=\theta_{12}(\omega k).
\end{equation}
Therefore, we have the following RH problem.
\begin{RHP}\label{rhpM}
Find a $3\times 3$ matrix-valued function $M(k):= M(k;y,t)$ such that
\begin{itemize}
\item Analyticity: $M(k)$ is meromorphic in $\mathbb{C}\setminus\Sigma$, where $\Sigma=\mathbb{R}\cup\omega\mathbb{R}\cup\omega^2\mathbb{R}$.
\item Jump relation: $M_{+}(k)=M_{-}(k)V(k)$, $k\in \Sigma$, where
\begin{align}\label{jumpp1}
   V(k)=\left\{
   \begin{aligned}
    &\left(\begin{array}{ccc} 1 & \bar{r}(k) e^{\mathrm{i}t\theta_{12}(k)} & 0 \\ -r(k)e^{-\mathrm{i}t\theta_{12}(k)} & 1-|r(k)|^2 & 0 \\ 0 & 0 & 1 \end{array}\right):= V_0(k), \quad k\in\mathbb{R},\\
    &\Gamma_4^2V_0(\omega^2k)\Gamma_4^{-2}, \quad k\in\omega\mathbb{R},\\
    &\Gamma_4V_0(\omega k)\Gamma_4^{-1}, \quad k\in\omega^2 \mathbb{R}.\\
   \end{aligned}
        \right.
\end{align}
\item Asymptotic behaviors:
\begin{equation}
M(k)=I+\mathcal{O}(k^{-1}), \quad  k\rightarrow\infty.
\end{equation}
\item Singularities: $M(k)$ has singularity at $\varkappa_{\nu}, \ \nu=1,\cdots,6$ with
\begin{align}\label{resp1}
   M(k)=\left\{
   \begin{aligned}
    &M_{\pm 1}(k)+\mathcal{O}(1), \quad k\rightarrow\pm1,\\
    &\Gamma_4^2M_{\pm 1}(\omega^2k)\Gamma_4^{-2}+\mathcal{O}(1), \quad k\rightarrow\pm\omega,\\
    &\Gamma_4M_{\pm 1}(\omega k)\Gamma_4^{-1}+\mathcal{O}(1), \quad k\rightarrow\pm\omega^2,\\
   \end{aligned}
        \right.
\end{align}
where
\begin{equation*}
M_{\pm 1}(k)=\frac{1}{k\mp 1}\left(\begin{array}{ccc} \alpha_{\pm} & \alpha_{\pm}  & \beta_{\pm} \\ -\alpha_{\pm} & -\alpha_{\pm}  & -\beta_{\pm} \\ 0 & 0 & 0 \end{array}\right),
\end{equation*}
and $\alpha_{\pm}=-\bar{\alpha}_{\pm}, \ \beta_{\pm}=-\bar{\beta}_{\pm}$.
\item Residue conditions: For $\zeta_n\in D_1\cap\mathcal{K}$,
\begin{equation}\label{resM11}
\begin{aligned}
&\res\limits_{k=\zeta_n} M(k)=\lim_{k\rightarrow \zeta_n}M(k)B_n,\\
&\res\limits_{k=\omega \bar{\zeta}_n}M(k)=\lim_{k\rightarrow \omega\bar{\zeta}_n}M(k)\Gamma_3(\omega \bar{B}_n)\Gamma_3:=\lim_{k\rightarrow \omega\bar{\zeta}_n}M(k)B_{n+N},\\
&\res\limits_{k=\omega \zeta_n}M(k)=\lim_{k\rightarrow \omega \zeta_n}M(k)\Gamma_4^2(\omega B_n)\Gamma_4^{-2}:=\lim_{k\rightarrow \omega \zeta_n}M(k)B_{n+2N},\\
&\res\limits_{k=\omega^2\bar{\zeta}_n}M(k)=\lim_{k\rightarrow\omega^2\bar{\zeta}_n}M(k)\Gamma_2(\omega^2 \bar{B}_n)\Gamma_2:=\lim_{k\rightarrow \omega^2\bar{\zeta}_n}M(k)B_{n+3N},\\
&\res\limits_{k=\omega^2\zeta_n}M(k)=\lim_{k\rightarrow\omega^2\zeta_n}M(k)\Gamma_4(\omega^2 B_n)\Gamma_4^{-1}:=\lim_{k\rightarrow \omega^2\zeta_n}M(k)B_{n+4N},\\
&\res\limits_{k=\bar{\zeta}_n} M(k)=\lim_{k\rightarrow \bar{\zeta}_n}M(k)\Gamma_1\bar{B}_n\Gamma_1:=\lim_{k\rightarrow \bar{\zeta}_n}M(k)B_{n+5N},
\end{aligned}
\end{equation}
where
\begin{equation}\label{res22}
   B_n=\left\{
   \begin{aligned}
    &\left(\begin{array}{ccc} 0 & -c_ne^{\mathrm{i}t\theta_{12}(\zeta_n)} & 0 \\ 0 & 0 & 0 \\ 0 & 0 & 0 \end{array}\right), \quad n=1,\cdots,N_1,\\
    &\left(\begin{array}{ccc} 0 & 0 & 0 \\ 0 & 0 & -c_ne^{\mathrm{i}t\theta_{23}(\zeta_n)} \\ 0 & 0 & 0 \end{array}\right), \quad n=N_1+1,\cdots,N.\\
   \end{aligned}
        \right.
\end{equation}
\end{itemize}
\end{RHP}

Particularly, introduce a vector-valued function
\begin{equation}\label{mtov}
m(k):=\begin{pmatrix} m_1(k)& m_2(k)& m_3(k) \end{pmatrix}= \begin{pmatrix} 1& 1& 1 \end{pmatrix}M(k),
\end{equation}
which can be viewed as a transformation from the $3\times 3$ matrix RH problem to the $1\times 3$ vector RH problem,
which suppresses the singularities at $\varkappa_j, \ j=1,\cdots,6$, and leads to the following vector-valued RH problem.
\begin{RHP}\label{rhpvm}
Find a  row vector-valued function $m(k):= m(k;y,t)$ such that
\begin{itemize}
\item  $m(k)$ is meromorphic in $\mathbb{C}\setminus\Sigma$.
\item For  $k\in \Sigma$, we have $m_{+}(k)=m_{-}(k)V(k)$, where $V(k)$ is given by \eqref{jumpp1}.
\item  As $k \to \infty$ in $\mathbb{C}\setminus \Sigma$, $m(k)=\begin{pmatrix}1& 1& 1\end{pmatrix}+\mathcal{O}(k^{-1})$.
\item $m(k)$ has the same form of residue conditions as $M(k)$ in RH problem \ref{rhpM}.
\end{itemize}
\end{RHP}
It follows that the solution  for the DP equation \eqref{DP} can be expressed in the following parametric form:
\begin{equation}\label{rescon}
\begin{aligned}
&u(y,t)=\frac{\partial}{\partial t}\log{\frac{m_2}{m_1}}(e^{\frac{\pi}{6}\mathrm{i}};y,t),\\
&x(y,t)=y+\log{\frac{m_2}{m_1}}(e^{\frac{\pi}{6}\mathrm{i}};y,t).
\end{aligned}
\end{equation}

\subsection{A regular RH problem}\label{conju}

In order to perform the long-time analysis via the $\bar{\partial}$-steepest descent method, we need to construct a regular RH problem by the following  two essential operations:
\begin{itemize}
	\item[(i)] Decompose the jump matrix $V(k)$ into appropriate upper/lower triangular factorizations so that the oscillating factor $e^{\pm2\mathrm{i}\theta_{12}(k)}$ are decaying in the corresponding sectors respectively;
	\item[(ii)]  Interpolate the poles by trading them for jumps along small closed loops enclosing
	each pole \cite{CJ}.
\end{itemize}
The operation (i) is aided by two well known factorizations of the jump matrix
\begin{equation}\label{facjum1}
	V(k)=b(k)^{-\dag}b(k)=B(k)T_0(k)B(k)^{-\dag}, \quad k\in\mathbb{R},
\end{equation}
where
\begin{align*}
	&b(k)^{-\dag}=\left(\begin{array}{ccc}  1 & 0 & 0 \\  -r(k)e^{-\mathrm{i}t\theta_{12}(k)} & 1 & 0 \\ 0 & 0 & 1 \end{array}\right), \quad b(k)=\left(\begin{array}{ccc}  1 & \bar{r}(k)e^{\mathrm{i}t\theta_{12}(k)} & 0 \\  0 & 1 & 0 \\ 0 & 0 &1 \end{array}\right),\\
	&B(k)^{-\dag}=\left(\begin{array}{ccc}  1 & 0 & 0 \\ -\frac{r(k)}{1-|r(k)|^2}e^{-\mathrm{i}t\theta_{12}(k)} & 1 & 0 \\ 0 & 0 & 1 \end{array}\right), \quad B(k)=\left(\begin{array}{ccc}  1 & \frac{\bar{r}(k)}{1-|r(k)|^2}e^{\mathrm{i}t\theta_{12}(k)} & 0 \\  0 & 1 & 0 \\ 0 & 0 & 1 \end{array}\right),\\
	&T_0(k)=\left(\begin{array}{ccc}  \frac{1}{1-|r(k)|^2} & 0 & 0 \\ 0 & 1-|r(k)|^2 & 0 \\ 0 & 0 & 1 \end{array}\right).
\end{align*}

To remove the diagonal matrix in the middle of the second factorization,  on $\mathbb{R}$, we  define
\begin{align}
&I(\hat{\xi})=\left\{
\begin{aligned}
	&\mathbb{R}, \quad  \hat{\xi} \in \mathcal{T}_1,\\
	&\emptyset, \quad \hat{\xi} \in \mathcal{T}_2,
\end{aligned}
\right.
\end{align}
then on $\omega \mathbb{R}, \ \omega^2 \mathbb{R}$,
\begin{align}
&\omega I(\hat{\xi})=\left\{\omega k: k\in I(\hat{\xi})\right\}, \quad \omega^2 I(\hat{\xi})=\left\{\omega^2 k: k\in I(\hat{\xi})\right\}.
\end{align}
Further, we introduce a scalar RH problem, which satisfies
\begin{RHP}\label{scalarrhp}
	Find a function $\delta(k):=\delta(k;\hat{\xi})$ satisfying the following properties:
	\begin{itemize}
		\item  $\delta(k)$ is analytical in $\mathbb{C}\setminus \mathbb{R}$.
		\item For $k \in \mathbb{R}$, we have
		\begin{equation}
		\begin{split}
		& \delta_{+}(k)=\delta_{-}(k)(1-|r(k)|^2), \quad k\in I(\hat{\xi}),\\
		& \delta_{+}(k)=\delta_{-}(k), \quad k\in\mathbb{R}\setminus I(\hat{\xi}).
		\end{split}
		\end{equation}
		\item As $k \to \infty$ in $\mathbb{C}\setminus \mathbb{R}$, $\delta(k)\rightarrow 1$.
	\end{itemize}
\end{RHP}
Utilizing the Plemelj's formula,  RH problem \ref{scalarrhp} admits a unique  solution
\begin{equation}
	\delta(k)=\exp\left(-\mathrm{i}\int_{I(\hat{\xi})}\frac{\nu(s)}{s-k}ds\right),
\end{equation}
with $\nu(k)=-\frac{1}{2\pi}\log{(1-|r(k)|^2)}$.

Now, we focus on operation (ii), our method for dealing with the poles in the RH problem follows the ideas in \cite{CJ}. We observe that on the unit circle the phase function appearing in the residue conditions \eqref{resM11} satisfies
\begin{equation}\label{cri}
	\operatorname{Im}\theta_{12}(\zeta_n)=2\sin{\phi_n}\left(\hat{\xi}-\frac{3}{4\cos^2{\phi_n}-3}\right)
\end{equation}
with $\zeta_n=e^{\mathrm{i}\phi_n}$.

Denote
the critical line
$$\operatorname{Re}k=L(\hat{\xi}):=\frac{\sqrt{3}}{2}\sqrt{1+1/\hat{\xi}},$$
then define
\begin{align}
&\Delta_{1}=\left\{j\in\widetilde{\mathcal{N}}: {\rm Re}{\zeta_j}<L(\hat{\xi})\right\}, \ \Delta_{2}=\left\{l\in\widetilde{\mathcal{N}}^A: {\rm Re}{\zeta_l}<L(\hat{\xi})\right\},
 \ \Delta=\Delta_1\cup\Delta_2,\nonumber\\
&\nabla_{1}=\left\{j\in\widetilde{\mathcal{N}}: {\rm Re}{\zeta_j}>L(\hat{\xi})\right\},
\ \nabla_{2}=\left\{l\in\widetilde{\mathcal{N}}^A: {\rm Re}{\zeta_l}>L(\hat{\xi})\right\},\ \nabla=\nabla_1\cup\nabla_2.\nonumber
\end{align}

To carry out operation (ii), define
\begin{equation}\label{T}
	\begin{aligned}
		& T_1(k)=\frac{H(\omega^2k)}{H(k)}, \quad T_2(k)=\frac{H(k)}{H(\omega k)},  \quad T_3(k)=\frac{H(\omega k)}{H(\omega^2k)},
	\end{aligned}
\end{equation}
where
\begin{equation}
	H(k)=\underset{j\in\Delta_1}\prod\frac{k-\zeta_j}{k-\bar{\zeta}_j}\underset{l\in\Delta_2}\prod\frac{k-\omega\zeta_l}{k-\omega^2\bar{\zeta}_l}\delta(k;\hat{\xi})^{-1},
\end{equation}
and
\begin{align}
	\varrho=\frac{1}{4}\min&\left\lbrace \min_{n\in \mathcal{N}} |\operatorname{Im}\zeta_n| ,\min_{n\in \mathcal{N},\arg k=\frac{\pi }{3}\mathrm{i}}|\zeta_n-k|,\min_{n\in
		\mathcal{N}\setminus\Lambda,\operatorname{Im}\theta_{12}(k)=0}|\zeta_n-k|,\right.\nonumber\\
	&\left. \min_{n\in\mathcal{N}}|\zeta_n-e^{\frac{\pi}{6}\mathrm{i}}|,\min_{n\neq m\in \mathcal{N}}|\zeta_n-\zeta_m|\right\rbrace .
\end{align}
Then the small disks $\mathbb{D}_n:=\mathbb{D}(\zeta_n,\varrho)$ are pairwise disjoint, also disjoint with critical lines and the contours. Moreover, $e^{\frac{\pi}{6}\mathrm{i}}\notin \mathbb{D}_n$.

Let
\begin{equation}\label{defT}
	T(k)=\text{diag}\{T_1(k),T_2(k),T_3(k)\},
\end{equation}
and for $n=1,\cdots,6N$, define
\begin{align}\label{G}
	G(k)=\left\{
	\begin{aligned}
		&I-\frac{B_n}{k-\zeta_n}, \quad k\in\mathbb{D}_n, \ n-k_0N\in\nabla, \ k_0\in\{0,\cdots,5\},\\
		&\begin{pmatrix} 1& 0& 0 \\-\frac{k-\zeta_n}{C_{n}e^{\mathrm{i}t\theta_{12}(\zeta_{n})}}& 1 & 0 \\ 0&	0 & 1 \end{pmatrix}, \quad k\in\mathbb{D}_n,n\in\Delta_1 \text{ or }n-2N\in\Delta_2,\\
		&\begin{pmatrix} 1&	0 & 0 \\ 0& 1 & 0\\ 0& -\frac{k-\zeta_n}{C_{n}e^{\mathrm{i}t\theta_{13}(\zeta_{n})}} & 1 \end{pmatrix}, \quad k\in\mathbb{D}_n,n-N\in\Delta_1 \text{ or }n-5N\in\Delta_2,\\
		&\begin{pmatrix} 1&	0 & -\frac{k-\zeta_n}{C_{n}e^{-\mathrm{i}t\theta_{13}(\zeta_{n})}} \\ 0& 1 & 0\\ 0&	0 & 1\end{pmatrix}, \quad k\in\mathbb{D}_n,n-2N\in\Delta_1 \text{ or }n-4N\in\Delta_2,\\
		&\begin{pmatrix} 1&	0 &  0\\ 0& 1 &  -\frac{k-\zeta_n}{C_{n}e^{-\mathrm{i}t\theta_{23}(\zeta_{n})}}\\ 0&	0 & 1\end{pmatrix}, \quad k\in\mathbb{D}_n,n-3N\in\Delta_1 \text{ or }n-N\in\Delta_2,\\
		&\begin{pmatrix} 1&	0 & 0 \\ 0& 1 & 0 \\ 0 &-\frac{k-\zeta_n}{C_{n}e^{\mathrm{i}t\theta_{23}(\zeta_{n})}} & 1 \end{pmatrix}, \quad k\in\mathbb{D}_n,n-4N\in\Delta_1 \text{ or }n\in\Delta_2,\\
		&\begin{pmatrix} 1&	-\frac{k-\zeta_n}{C_{n}e^{-\mathrm{i}t\theta_{12}(\zeta_{n})}} & 0 \\ 0& 1 & 0\\ 0&	0 & 1 \end{pmatrix}, \quad k\in\mathbb{D}_n,n-5N\in\Delta_1\text{ or } n-3N\in\Delta_2,\\
		&I, \quad elsewhere.
	\end{aligned}
	\right.
\end{align}
Denote
\begin{equation}
	\Sigma^{(1)}=\Sigma\cup\Sigma^{\mathrm{C}}, \quad \Sigma^{\mathrm{C}}= \mathop{\cup}\limits_{n=1}^{6N}\partial\mathbb{D}_n.
\end{equation}
 Here, $\mathbb{R}$ is oriented left-to-right and the directions on $\omega\mathbb{R}$ and $ \omega^2\mathbb{R}$ are determined by rotations of $\mathbb{R}$. Moreover,
  the disk boundaries are oriented counterclockwise in $\mathrm{D}_{2\nu-1}$ and clockwise in $\mathrm{D}_{2\nu}, \ \nu=1,2,3$.

We make the transformation
\begin{equation}\label{defM1}
	m^{(1)}(k)=m(k)G(k)T(k),
\end{equation}
which satisfies the following RH problem.

\begin{RHP}\label{rhp1}
	Find a $1\times 3$  vector-valued function $m^{(1)}(k):=m^{(1)}(k;y,t)$ such that
	\begin{itemize}
		\item $m^{(1)}(k)$ is holomorphic  in $\mathbb{C}\setminus\Sigma^{(1)}$.
		\item For $k \in \Sigma^{(1)}$, we have
		\begin{equation}
			m_+^{(1)}(k)=m_-^{(1)}(k)V^{(1)}(k),
		\end{equation}
		where
		\begin{align}\label{rhp1jump}
			V^{(1)}(k)=
			& \left\{
			\begin{aligned}
				&T^{-1}(k) \Gamma^j_4 b^{-\dag}(\omega^j k)b(\omega^j k) \Gamma^{-j}_4T(k), \quad k\in  \omega^j \mathbb{R}\setminus  \omega^j I(\hat{\xi}),\ j=0,1,2,\\
				  &T_{-}^{-1}(k) \Gamma^j_4 B(\omega^j k) T_0(\omega^jk)B^{-\dag}(\omega^j k) \Gamma^{-j}_4 T_{+}(k), \quad k\in \omega^j I(\hat{\xi}),\ j=0,1,2,\\
				&T^{-1}(k)G(k)T(k), \quad k\in\partial\mathbb{D}_n\cap\Big(\mathop{\cup}\limits_{\nu=1,2,3} \mathrm{D}_{2\nu-1}\Big),\\
				&T^{-1}(k)G^{-1}(k)T(k), \quad k\in\partial\mathbb{D}_n\cap\Big(\mathop{\cup}\limits_{\nu=1,2,3} \mathrm{D}_{2\nu}\Big).
			\end{aligned}
			\right.
		\end{align}
		\item As $k \to \infty$ in $\mathbb{C}\setminus \Sigma^{(1)}$, $	m^{(1)}(k)=\begin{pmatrix} 1& 1& 1 \end{pmatrix} +\mathcal{O}(k^{-1}).$
	\end{itemize}
\end{RHP}

By \eqref{rhp1jump}, it is readily seen that $V^{(1)}(z) \to I$ as $t \to \infty$ for $z \in \Sigma^{\mathrm{C}}$ exponentially fast. Thus, RH problem
\ref{rhp1} is asymptotically equivalent to the following RH problem. 

\begin{RHP}\label{rhp2}
Find a $1\times 3$  vector-valued function $m^{(2)}(k):=m^{(2)}(k;y,t)$ such that
\begin{itemize}
\item $m^{(2)}(k)$ is  holomorphic  in $\mathbb{C}\setminus \Sigma$.
\item For $k \in \Sigma$, we have
\begin{equation}
m_+^{(2)}(k)=m_-^{(2)}(k)V^{(2)}(k),
\end{equation}
where for $j=0,1,2,$
\begin{align}\label{rhp2jump}
V^{(2)}(k)=
& \left\{
\begin{aligned}
&T^{-1}(k) \Gamma^j_4 b^{-\dag}(\omega^j k)b(\omega^j k) \Gamma^{-j}_4T(k), \quad k\in  \omega^j \mathbb{R}\setminus  \omega^j I(\hat{\xi}),\\
&T_{-}^{-1}(k) \Gamma^j_4 B(\omega^j k) T_0(\omega^jk)B^{-\dag}(\omega^j k) \Gamma^{-j}_4 T_{+}(k), \quad k\in \omega^j I(\hat{\xi}).
\end{aligned}
\right.
\end{align}
\item As $k \to \infty$  in $\mathbb{C}\setminus \Sigma$, $	m^{(2)}(k)=\begin{pmatrix} 1& 1& 1 \end{pmatrix}+\mathcal{O}(k^{-1}).$
\end{itemize}
\end{RHP}

It can be shown that
\begin{equation}\label{defm2}
	m^{(1)}(k) = 	m^{(2)}(k) (I +\mathcal{O}(e^{-ct}) ),
\end{equation}
where $c>0$ is a constant.
Next, we will perform the asymptotic analysis in different transition zones based on RH Problem  \ref{rhp2} and finally  complete  the proof of Theorem \ref{th1}.

\section{Painlev\'e Asymptotics  in Transition Zone $\mathcal{T}_{1}$}\label{7}


In this section,  we  study the Painlev\'e asymptotics in  the transition  zone   $\mathcal{T}_{1}$  between $\mathcal{S}_1$ (See Figure \ref{figure0}) and $\mathcal{Z}_1$ (See Figure \ref{figurea}),
whose critical case corresponds to  Figure \ref{figureb}.
We only provide a detailed analysis in
 $\mathcal{T}_1^R=\mathcal{T}_1\cap\{(y,t):  \hat{\xi}\geq -\frac{3}{8}\}$
in this section, as the discussion for the other half zone is similar.

\begin{figure}[htbp]
\centering
\subfigure[$\hat{\xi}<-\frac{3}{8}$]{\label{figure0}
	\begin{minipage}[t]{0.32\linewidth}
	\centering
	\includegraphics[width=1.52in]{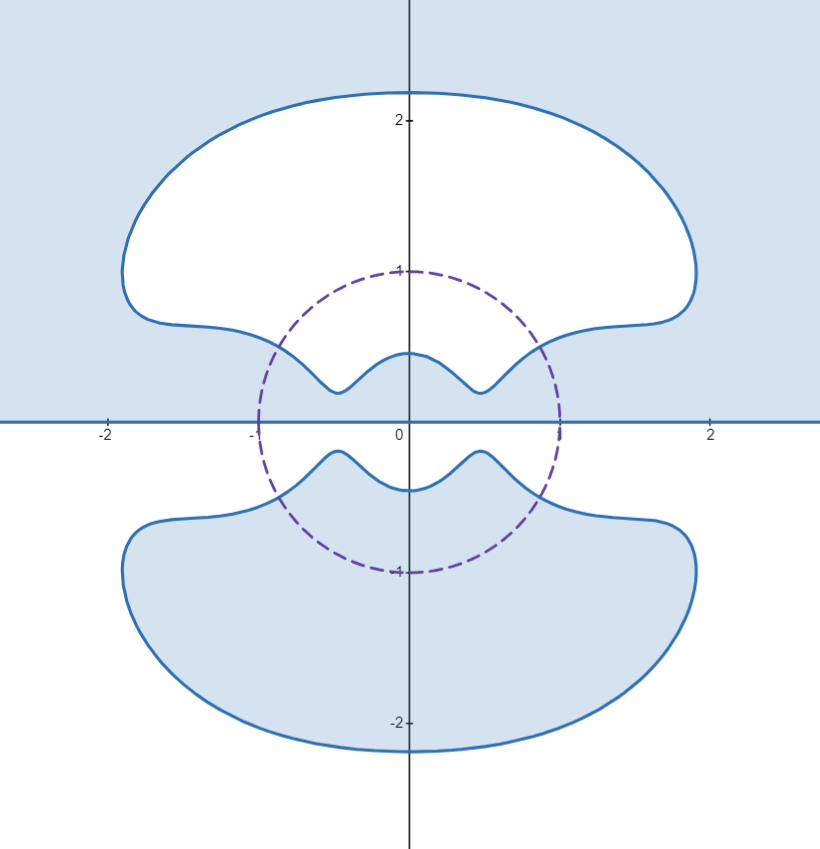}
	\end{minipage}%
}%
\subfigure[$\hat{\xi}=\frac{3}{8}$]{\label{figurea}
	\begin{minipage}[t]{0.32\linewidth}
	\centering
	\includegraphics[width=1.52in]{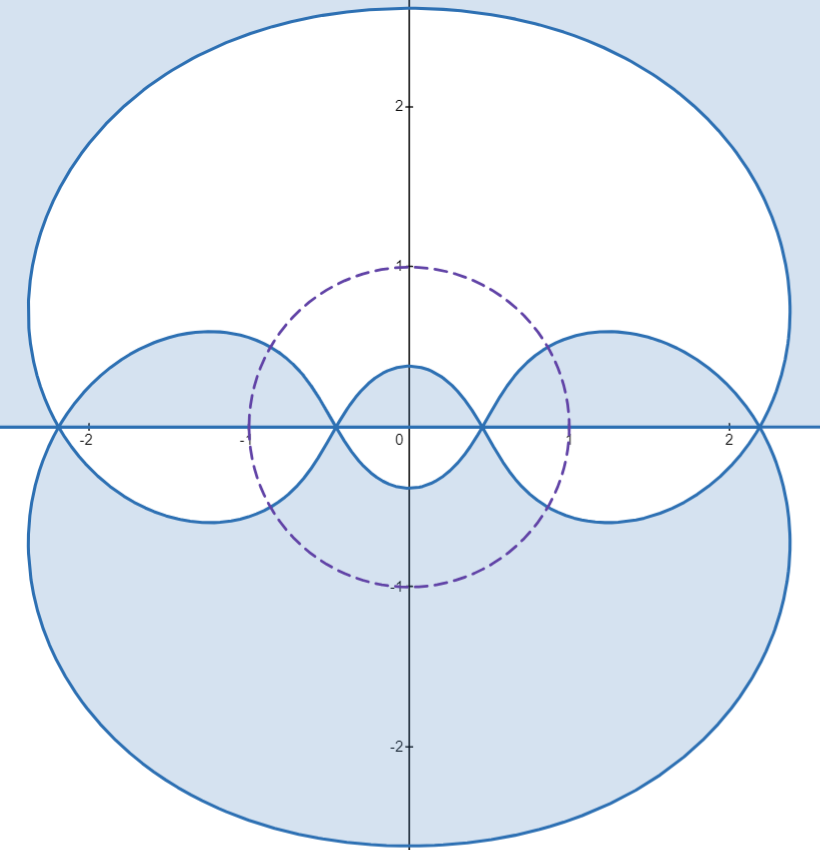}
	\end{minipage}%
}%
\subfigure[$-\frac{3}{8}<\hat{\xi}<0$]{\label{figureb}
	\begin{minipage}[t]{0.32\linewidth}
	\centering
	\includegraphics[width=1.52in]{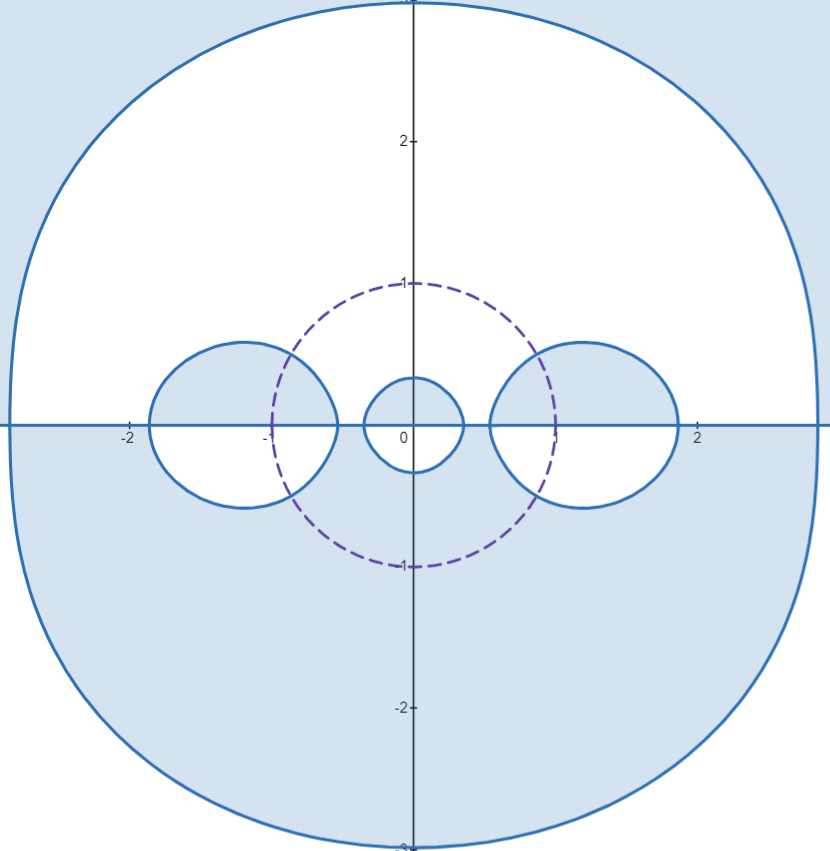}
	\end{minipage}
}%

\caption{\footnotesize Signature table of ${\rm Im}\theta_{12}(k)$ with different $\hat{\xi}$:
	$\textbf{(a)}$ $\hat{\xi}<-\frac{3}{8}$,
	$\textbf{(b)}$ $\hat{\xi}=-\frac{3}{8}$,
	$\textbf{(c)}$ $-\frac{3}{8}<\hat{\xi}<0$.
The blue and white sectors are
 ${\rm Im}\theta_{12}(k)<0$ and   ${\rm Im}\theta_{12}(k)>0$, respectively.
	Moreover, the  purple dotted line  stands for the unit circle.}
\label{figtheta}
\end{figure}

In the region $\mathcal{T}_1^R$,  there are 24 saddle points on three contours $\omega^l \mathbb{R}$, $l=0,1,2$, among them 8 saddle points are on $\mathbb{R}$, which are
\begin{align}
&k_1 = -k_8 = \frac{\sqrt{2}}{4} \left( \sqrt{s_1 +s_2 } + \sqrt{s_1+8 +s_2}\right),\label{pha01}\\
& k_2 = -k_7 = \frac{\sqrt{2}}{4} \left( \sqrt{s_1 -s_2} + \sqrt{s_1+8 -s_2}\right),\label{pha02}\\
&k_3 = -k_6= \frac{1}{k_2},  \quad k_4 = -k_5 = \frac{1}{k_1}, \label{pha03}
\end{align}
where
\begin{equation}\label{phas1s2}
s_1 = -\frac{3}{\hat{\xi}} -2,\quad  s_2 =- \frac{\sqrt{3}\sqrt{3+8\hat{\xi}}}{\hat{\xi}}.
\end{equation}
Moreover, in $\mathcal{T}_{1}^R$, as $t\to\infty$, we have
 $\hat{\xi} \to -\frac{3}{8}$, $s_1 \to 6$, $s_2 \to 0$, and then
\begin{align}
&k_1, k_2 \to k_a = \frac{\sqrt{7}+\sqrt{3}}{2}, \quad k_3, k_4 \to k_b= \frac{\sqrt{7}-\sqrt{3}}{2},\label{hjfd1}\\
&k_5, k_6 \to k_c =-\frac{\sqrt{7}-\sqrt{3}}{2}, \quad k_7, k_8 \to k_d =- \frac{\sqrt{7}+\sqrt{3}}{2},\label{hjfd2}
\end{align}
from which we obtain  the corresponding limit points $\omega^l k_j$, $j=a,b,c,d,\ l=1,2$  on  other two   contours $\omega^l \mathbb{R}, l=1,2$.

By \eqref{rhp2jump}, the jump matrix  $V^{(2)}(k)$ reads
\begin{equation}
	V^{(2)}(k)= \left(\begin{array}{ccc} 1 &   - \bar{d}(k)  e^{\mathrm{i}t\theta_{12}(k) } &0 \\ 0&1& 0\\0&0&1 \end{array}  \right)
	\left(\begin{array}{ccc} 1 & 0&0\\
		d(k) e^{-\mathrm{i}t\theta_{12}(k) } &1&0 \\ 0 & 0&1\end{array}  \right),\quad k \in \mathbb{R},
\end{equation}
where
\begin{equation}\label{p1d1}
	d(k) := -\frac{r(k)}{1-|r(k)|^2}\left(\frac{T_1}{T_2}\right)_+ (k).
\end{equation}
The factorization of  $V^{(2)}(k)$ on $\omega \mathbb{R}$ and $\omega^2 \mathbb{R}$  can be given by the symmetries.

\subsection{Hybrid $\bar{\partial}$-RH problem } \label{subsec31}

Define some intervals
\begin{align*}
& I_{1} = (k_1,\infty),\, I_{2}=((k_2+k_3)/2,k_2),\, I_3 =(k_3, (k_2+k_3)/2 ),\, I_4 =(k_0,k_4),\\
&I_5=( k_5,k_0),\, I_6 = (( k_6+ k_7)/2, k_6),\, I_7 = ( k_7,( k_6+ k_7)/2 ),\, I_8=(-\infty, k_8),
\end{align*}
where $k_0:=0 $ and $\omega I_j=\{k\in I_j:\omega k \}, \ \omega^2 I_j=\{k \in I_j: \omega^2 k \},\ j =1,\cdots, 8$. Further denote
\begin{align*}
I = \mathop{\cup}\limits_{j=1}^8 I_j, \quad  \omega I = \{\omega k:k\in I\},\quad \omega^2 I = \{\omega^2 k: k\in I \}.
\end{align*}
The signature table of $\rm{Im} \theta_{12}(k)$ illustrated in  Figure \ref{figureb}, implies opening $\bar\partial$-lenses  around the intervals $ \omega^lI, l=0,1,2 $ as follows.

First, we open  the interval $ I$   as   depicted    in  Figure \ref{pfv4}  with  a sufficiently small  fixed  angle $0<\varphi<\pi/4$
  such that  the following conditions hold
\begin{itemize}
\item  all opened sectors $\Omega_j \cup \Omega_j^*$  and their boundaries  $\Sigma_{j}\cup \Sigma_{j}^* $  fall within their  decaying regions;

\item  each $\Omega_j \cup \Omega_j^*$  does't   intersect any  critical line ${\rm Im} \theta_{12}(k)=0$;

\item  each $\Omega_j \cup \Omega_j^*$ does't   intersect  any small disk $\mathbb{D}_n$ and $\mathbb{D}_n^*$, $n=1,\cdots,N$.

\end{itemize}
Also, let $\Sigma_{ij}$ be the boundary between $\Omega_i$ and $\Omega_j$,   and  denote
$$\Omega = \mathop{\cup}\limits_{j=0}^8\Omega_{j}, \ \ \Sigma=\mathop{\cup}\limits_{j=0}^8\Sigma_j,
\ \ \tilde\Sigma= \Sigma_{23}\mathop{\cup}\Sigma_{04}\mathop{\cup}\Sigma_{05}\mathop{\cup}\Sigma_{67}.$$


	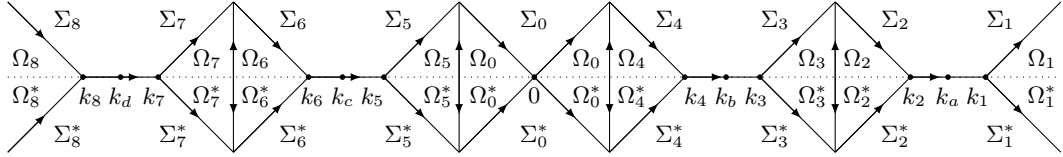
\begin{figure}[http]
\vspace{2mm}
	\centering
	\begin{tikzpicture}[scale=1]
		\draw[dotted](-7,0)--(7,0);
		
		\draw [  ] (-2,0)--(-1,1);
	    \draw [  ](2,0)-- (1,1);
	    \draw [  ] (0,0)--(-1,1);
	    \draw [  ](0,0)-- (1,1);
	    \draw[-latex,  ] (-2,0)--(-1.45,0.55);
	    \draw[-latex,  ] (-1, 1)--(-0.35, 0.35);
	     \draw[-latex,  ] (0,0)--(0.55,0.55);	
		\draw[-latex,  ] (1, 1)--(1.65, 0.35);	    	
	   	\draw [  ](-2,0) -- (-1,-1);
		\draw [  ](2,0)-- (1,-1);
		 \draw [  ] (0,0)--(-1,-1);
		 \draw [  ](0,0)-- (1,-1);
        \draw[-latex,  ](-2,0)--(-1.45,-0.55);
        \draw[-latex,  ](-1,-1)--(-0.35,-0.35);
        \draw[-latex,  ](0,0)--(0.55,-0.55);
        \draw[-latex,  ](1,-1)--(1.65,-0.35);

		\draw [  ](-3,0) -- (-4,1);
		\draw [  ](-5,0) -- (-4,1);		
		\draw[-latex,  ] (-5,0)--(-4.45,0.55);
		\draw[-latex,  ] (-4,1)--(-3.35,0.35);
		\draw [  ](-3,0) -- (-4,-1);
		\draw [  ](-5,0) -- (-4,-1);		
		\draw[-latex,  ] (-5,0)--(-4.45,-0.55);
		\draw[-latex,  ] (-4,-1)--(-3.35,-0.35);

		\draw [  ](-6,0) -- (-7,1);
		\draw [  ](6,0) -- (7,1);		
		\draw [-latex,   ](-7,1) -- (-6.5,0.5);
        \draw [-latex,   ](6,0) -- (6.5,0.5);

		\draw [  ](-6,0) -- (-7,-1);
     	\draw [  ](6,0) -- (7,-1);		
    	\draw [-latex,   ](-7,-1) -- (-6.5,-0.5);
     	\draw [-latex,   ](6,0) -- (6.5,-0.5);
			
 	\draw [  ](3,0) -- (4,1);
	\draw [  ](5,0) -- (4,1);		
	\draw[-latex,  ] (4,1)--(4.65,0.35);
	\draw[-latex,  ] (3,0)--(3.55,0.55);
	\draw [  ](3,0) -- (4,-1);
	\draw [  ](5,0) -- (4,-1);		
	\draw[-latex,  ] (4,-1)--(4.65,-0.35);
	\draw[-latex,  ] (3,-0)--(3.55,-0.55);

	\draw[  ](1,1)--(1,0);
	\draw[  ](1,-1)--(1,0);
	\draw[-latex,  ](1,0)--(1,0.5);
	\draw[-latex,  ](1,0)--(1,-0.5);

	\draw[  ](-1,1)--(-1,0);
	\draw[  ](-1,-1)--(-1,0);
	\draw[-latex,  ](-1,0)--(-1,0.5);
	\draw[-latex,  ](-1,0)--(-1,-0.5);
	
		\draw[  ](4,1)--(4,0);
		\draw[  ](4,-1)--(4,0);
		\draw[-latex,  ](4,0)--(4,0.5);
		\draw[-latex,  ](4,0)--(4,-0.5);
		
		\draw[  ](-4,1)--(-4,0);
        \draw[  ](-4,-1)--(-4,0);
        \draw[-latex,  ](-4,0)--(-4,0.5);
        \draw[-latex,  ](-4,0)--(-4,-0.5);

		 \draw[  ](-3,0)--(-2,0);
	   	 \draw[-latex,  ](-3,0)--(-2.3,0);
	     \draw[  ](-6,0)--(-5,0);
		\draw[-latex,  ](-6,0)--(-5.25,0);
	     \draw[  ](3,0)--(2,0);
      	\draw[-latex,  ](2,0)--(2.5,0);
   	    \draw[  ](6,0)--(5,0);
     	\draw[-latex,  ](5,0)--(5.4,0);
   		
		\filldraw [](0,0) circle [radius=0.03];
 	\node  [below]  at (0,0) {\footnotesize$0$};
		\filldraw [](3,0) circle [radius=0.03];
		\filldraw [](5,0) circle [radius=0.03];
		\filldraw [](2,0) circle [radius=0.03];	
		\filldraw [](6,0) circle [radius=0.03];
		
		\filldraw [](-3,0) circle [radius=0.03];
		\filldraw [](-5,0) circle [radius=0.03];
		\filldraw [](-2,0) circle [radius=0.03];	
		\filldraw [](-6,0) circle [radius=0.03];
		
		\node  [below]  at (2.15,0) {\footnotesize$k_{4}$};
		\node  [below]  at (2.95,0) {\footnotesize$k_{3}$};	
		\node  [below]  at (5.05,0) {\footnotesize$k_{2}$};
		\node  [below]  at (5.9,0) {\footnotesize$k_{1}$};	
		
		\filldraw [](5.5,0) circle [radius=0.03];
		\node  [below]  at (5.5,0) {\footnotesize$k_{a}$};	
		\filldraw [](2.55,0) circle [radius=0.03];
		\node  [below]  at (2.55,0) {\footnotesize$k_{b}$};		
		\filldraw [](-5.5,0) circle [radius=0.03];
		\node  [below]  at (-5.5,0) {\footnotesize$k_{d}$};	
		\filldraw [](-2.55,0) circle [radius=0.03];
		\node  [below]  at (-2.55,0) {\footnotesize$k_{c}$};		
		
		\node  [below]  at (-2.15,0) {\footnotesize$k_{5}$};
		\node  [below]  at (-2.95,0) {\footnotesize$k_{6}$};	
		\node  [below]  at (-5.05,0) {\footnotesize$k_{7}$};
		\node  [below]  at (-5.9,0) {\footnotesize$k_{8}$};		
		
		\node  [above]  at (0,0.5) {\footnotesize$\Sigma_{0}$};
		\node  [below]  at (0,-0.5) {\footnotesize$\Sigma_{0}^*$};	
		
		\node  [above]  at (1.8,0.5) {\footnotesize$\Sigma_{4}$};
		\node  [below]  at (1.8,-0.5) {\footnotesize$\Sigma_{4}^*$};
		\node  [above]  at (3.2,0.5) {\footnotesize$\Sigma_{3}$};
		\node  [below]  at (3.2,-0.5) {\footnotesize$\Sigma_{3}^*$};
		\node  [above]  at (4.8,0.5) {\footnotesize$\Sigma_{2}$};
		\node  [below]  at (4.8,-0.5) {\footnotesize$\Sigma_{2}^*$};	
		\node  [above]  at (6.2,0.5) {\footnotesize$\Sigma_{1}$};
		\node  [below]  at (6.2,-0.5) {\footnotesize$\Sigma_{1}^*$};	
		
		\node  [above]  at (-1.8,0.5) {\footnotesize$\Sigma_{5}$};
		\node  [below]  at (-1.8,-0.5) {\footnotesize$\Sigma_{5}^*$};
		\node  [above]  at (-3.2,0.5) {\footnotesize$\Sigma_{6}$};
		\node  [below]  at (-3.2,-0.5) {\footnotesize$\Sigma_{6}^*$};
		\node  [above]  at (-4.8,0.5) {\footnotesize$\Sigma_{7}$};
		\node  [below]  at (-4.8,-0.5) {\footnotesize$\Sigma_{7}^*$};	
		\node  [above]  at (-6.2,0.5) {\footnotesize$\Sigma_{8}$};
		\node  [below]  at (-6.2,-0.5) {\footnotesize$\Sigma_{8}^*$};

	 \node  [above]  at (0.7,-0.02) {\footnotesize$\Omega_{0}$};
	\node  [below]  at (0.7,0.02) {\footnotesize$\Omega_{0}^*$};
	
		 \node  [above]  at (-0.67,-0.02) {\footnotesize$\Omega_{0}$};
	\node  [below]  at (-0.67,0.02) {\footnotesize$\Omega_{0}^*$};
		
		\node  [above]  at (1.3,-0.02) {\footnotesize$\Omega_{4}$};
		\node  [below]  at (1.3,0.02) {\footnotesize$\Omega_{4}^*$};
		\node  [above]  at (3.7,-0.02) {\footnotesize$\Omega_{3}$};
		\node  [below]  at (3.7,0.02) {\footnotesize$\Omega_{3}^*$};
		\node  [above]  at (4.3,-0.02) {\footnotesize$\Omega_{2}$};
		\node  [below]  at (4.3,0.02) {\footnotesize$\Omega_{2}^*$};
		\node  [above]  at (6.75,-0.02) {\footnotesize$\Omega_{1}$};
		\node  [below]  at (6.75,0.02) {\footnotesize$\Omega_{1}^*$};
		
		\node  [above]  at (-1.28,-0.02) {\footnotesize$\Omega_{5}$};
		\node  [below]  at (-1.28,0.02) {\footnotesize$\Omega_{5}^*$};
		\node  [above]  at (-3.7,-0.02) {\footnotesize$\Omega_{6}$};
		\node  [below]  at (-3.7,0.02) {\footnotesize$\Omega_{6}^*$};
		\node  [above]  at (-4.35,-0.02) {\footnotesize$\Omega_{7}$};
		\node  [below]  at (-4.35,0.02) {\footnotesize$\Omega_{7}^*$};
		\node  [above]  at (-6.75,-0.02) {\footnotesize$\Omega_{8}$};
		\node  [below]  at (-6.75,0.02) {\footnotesize$\Omega_{8}^*$};
	
	\end{tikzpicture}
	\caption{\footnotesize The  contour obtained after opening a part  contour  $I$ of     $\mathbb{R}$.}
	\label{pfv4}
\vspace{2mm}
\end{figure}

Then, we  open   other  intervals  $ \omega  I $  and  $ \omega^2  I$  by rotating  $ I $ according to  the symmetries,
corresponding opened sectors $\omega^l(\Omega_j \cup \Omega_j^*), l=1,2$,  boundaries  $\omega^l(\Sigma_{j}\cup \Sigma_{j}^*), l=1,2 $ and  $\omega^l(\tilde \Sigma_{j}\cup \tilde \Sigma_{j}^*), l=1,2$.
Thus, the whole contour  $ \Sigma^{(3)}$  is given by
\begin{align*}
&\Sigma^{(3)} =   \mathop{\cup}\limits_{l=0}^2 \omega^l  \left(  (\mathbb{R}\setminus  I ) \cup (\Sigma \cup \Sigma^*) \cup (\tilde \Sigma\cup \tilde \Sigma^*) \right).
\end{align*}
See Figure  \ref{pfv3}.

\begin{figure}[http]
	\centering
	\begin{tikzpicture}[scale=0.7]
		\draw[dotted](-7,0)--(7,0);
\node  [right]  at (7.5,0) {\footnotesize$\mathbb{R}$};

\draw [  ] (-2,0)--(-1.5,0.866);
\draw [  ](2,0)-- (1.5,0.866);
\draw[-latex,  ] (-2,0)--(-1.6,0.6928);
\draw[-latex,  ] (1.5,0.866)--(1.85,0.2598);			
\draw [  ] (0,0)--(-1.5,0.866);
\draw [  ] (0,0)--(1.5,0.866);
\draw[-latex,  ] (-1.5,0.866)--(-0.65,0.3753);
\draw[-latex,  ] (0, 0)--(1, 0.5773);

\draw [  ](-2,0) -- (-1.5,-0.866);
\draw [  ](2,0)-- (1.5,-0.866);
\draw[-latex,  ](-2,0)--(-1.6,-0.6928);
\draw[-latex,  ](1.5,-0.866)--(1.85,-0.2598);	
\draw [  ](0,0) -- (-1.5,-0.866);
\draw [  ](0,0)-- (1.5,-0.866);
\draw[-latex,  ](-1.5,-0.866)--(-0.65,-0.3753);
\draw[-latex,  ](0,0)--(1,-0.5773);	
	
	       \draw[  ](-1.5,0.866)--(-1.5,0);
	\draw[  ](-1.5,-0.866)--(-1.5,0);
	\draw[-latex,  ](-1.5,0)--(-1.5,0.5);
	\draw[-latex,  ](-1.5,0)--(-1.5,-0.5);

	\draw[  ](1.5,0.866)--(1.5,0);
	\draw[  ](1.5,-0.866)--(1.5,0);
	\draw[-latex,  ](1.5,0)--(1.5,0.5);
	\draw[-latex,  ](1.5,0)--(1.5,-0.5);

	\draw [  ](-3,0) -- (-4,1);
	\draw [  ](-5,0) -- (-4,1);		
	\draw[-latex,  ] (-5,0)--(-4.45,0.55);
	\draw[-latex,  ] (-4,1)--(-3.35,0.35);
	\draw [  ](-3,0) -- (-4,-1);
	\draw [  ](-5,0) -- (-4,-1);		
	\draw[-latex,  ] (-5,0)--(-4.45,-0.55);
	\draw[-latex,  ] (-4,-1)--(-3.35,-0.35);

	\draw [  ](-6,0) -- (-7,1);
	\draw [  ](6,0) -- (7,1);		
	\draw [-latex,   ](-7,1) -- (-6.5,0.5);
	\draw [-latex,   ](6,0) -- (6.5,0.5);

	\draw [  ](-6,0) -- (-7,-1);
	\draw [  ](6,0) -- (7,-1);		
	\draw [-latex,   ](-7,-1) -- (-6.5,-0.5);
	\draw [-latex,   ](6,0) -- (6.5,-0.5);
	
	\draw [  ](3,0) -- (4,1);
	\draw [  ](5,0) -- (4,1);		
	\draw[-latex,  ] (4,1)--(4.65,0.35);
	\draw[-latex,  ] (3,0)--(3.55,0.55);
	\draw [  ](3,0) -- (4,-1);
	\draw [  ](5,0) -- (4,-1);		
	\draw[-latex,  ] (4,-1)--(4.65,-0.35);
	\draw[-latex,  ] (3,-0)--(3.55,-0.55);

	\draw[  ](4,1)--(4,0);
	\draw[  ](4,-1)--(4,0);
	\draw[-latex,  ](4,0)--(4,0.5);
	\draw[-latex,  ](4,0)--(4,-0.5);
	
	\draw[  ](-4,1)--(-4,0);
	\draw[  ](-4,-1)--(-4,0);
	\draw[-latex,  ](-4,0)--(-4,0.5);
	\draw[-latex,  ](-4,0)--(-4,-0.5);

	\draw[  ](-3,0)--(-2,0);
	\draw[-latex,  ](-3,0)--(-2.3,0);
	\draw[  ](-6,0)--(-5,0);
	\draw[-latex,  ](-6,0)--(-5.25,0);
	\draw[  ](3,0)--(2,0);
	\draw[-latex,  ](2,0)--(2.5,0);
	\draw[  ](6,0)--(5,0);
	\draw[-latex,  ](5,0)--(5.4,0);

\draw[dotted,rotate=60](-7,0)--(7,0);
\node    at (4.1,6.7) {\footnotesize$\omega\mathbb{R}$};	

\draw [  ,rotate=60] (-2,0)--(-1.5,0.866);
\draw [  ,rotate=60](2,0)-- (1.5,0.866);
\draw[-latex,  ,rotate=60] (-2,0)--(-1.6,0.6928);
\draw[-latex,  ,rotate=60] (1.5,0.866)--(1.85,0.2598);			
\draw [  ,rotate=60] (0,0)--(-1.5,0.866);
\draw [  ,rotate=60] (0,0)--(1.5,0.866);
\draw[-latex,  ,rotate=60] (-1.5,0.866)--(-0.65,0.3753);
\draw[-latex,  ,rotate=60] (0, 0)--(1, 0.5773);

\draw [  ,rotate=60](-2,0) -- (-1.5,-0.866);
\draw [  ,rotate=60](2,0)-- (1.5,-0.866);
\draw[-latex,  ,rotate=60](-2,0)--(-1.6,-0.6928);
\draw[-latex,  ,rotate=60](1.5,-0.866)--(1.85,-0.2598);	
\draw [  ,rotate=60](0,0) -- (-1.5,-0.866);
\draw [  ,rotate=60](0,0)-- (1.5,-0.866);
\draw[-latex,  ,rotate=60](-1.5,-0.866)--(-0.65,-0.3753);
\draw[-latex,  ,rotate=60](0,0)--(1,-0.5773);	

\draw[  ,rotate=60](-1.5,0.866)--(-1.5,0);
\draw[  ,rotate=60](-1.5,-0.866)--(-1.5,0);
\draw[-latex,  ,rotate=60](-1.5,0)--(-1.5,0.5);
\draw[-latex,  ,rotate=60](-1.5,0)--(-1.5,-0.5);

\draw[  ,rotate=60](1.5,0.866)--(1.5,0);
\draw[  ,rotate=60](1.5,-0.866)--(1.5,0);
\draw[-latex,  ,rotate=60](1.5,0)--(1.5,0.5);
\draw[-latex,  ,rotate=60](1.5,0)--(1.5,-0.5);

\draw [  ,rotate=60](-3,0) -- (-4,1);
\draw [  ,rotate=60](-5,0) -- (-4,1);		
\draw[-latex,  ,rotate=60] (-5,0)--(-4.45,0.55);
\draw[-latex,  ,rotate=60] (-4,1)--(-3.35,0.35);
\draw [  ,rotate=60](-3,0) -- (-4,-1);
\draw [  ,rotate=60](-5,0) -- (-4,-1);		
\draw[-latex,  ,rotate=60] (-5,0)--(-4.45,-0.55);
\draw[-latex,  ,rotate=60] (-4,-1)--(-3.35,-0.35);

\draw [  ,rotate=60](-6,0) -- (-7,1);
\draw [  ,rotate=60](6,0) -- (7,1);		
\draw [-latex,   ,rotate=60](-7,1) -- (-6.5,0.5);
\draw [-latex,   ,rotate=60](6,0) -- (6.5,0.5);

\draw [  ,rotate=60](-6,0) -- (-7,-1);
\draw [  ,rotate=60](6,0) -- (7,-1);		
\draw [-latex,   ,rotate=60](-7,-1) -- (-6.5,-0.5);
\draw [-latex,   ,rotate=60](6,0) -- (6.5,-0.5);

\draw [  ,rotate=60](3,0) -- (4,1);
\draw [  ,rotate=60](5,0) -- (4,1);		
\draw[-latex,  ,rotate=60] (4,1)--(4.65,0.35);
\draw[-latex,  ,rotate=60] (3,0)--(3.55,0.55);
\draw [  ,rotate=60](3,0) -- (4,-1);
\draw [  ,rotate=60](5,0) -- (4,-1);		
\draw[-latex,  ,rotate=60] (4,-1)--(4.65,-0.35);
\draw[-latex,  ,rotate=60] (3,-0)--(3.55,-0.55);

\draw[  ,rotate=60](4,1)--(4,0);
\draw[  ,rotate=60](4,-1)--(4,0);
\draw[-latex,  ,rotate=60](4,0)--(4,0.5);
\draw[-latex,  ,rotate=60](4,0)--(4,-0.5);

\draw[  ,rotate=60](-4,1)--(-4,0);
\draw[  ,rotate=60](-4,-1)--(-4,0);
\draw[-latex,  ,rotate=60](-4,0)--(-4,0.5);
\draw[-latex,  ,rotate=60](-4,0)--(-4,-0.5);

\draw[  ,rotate=60](-3,0)--(-2,0);
\draw[-latex,  ,rotate=60](-3,0)--(-2.3,0);
\draw[  ,rotate=60](-6,0)--(-5,0);
\draw[-latex,  ,rotate=60](-6,0)--(-5.25,0);
\draw[  ,rotate=60](3,0)--(2,0);
\draw[-latex,  ,rotate=60](2,0)--(2.5,0);
\draw[  ,rotate=60](6,0)--(5,0);
\draw[-latex,  ,rotate=60](5,0)--(5.4,0);

\draw[dotted,rotate=120](-7,0)--(7,0);
\node    at (-4.1,6.7) {\footnotesize$\omega^2\mathbb{R}$};

\draw [  ,rotate=120] (-2,0)--(-1.5,0.866);
\draw [  ,rotate=120](2,0)-- (1.5,0.866);
\draw[-latex,  ,rotate=120] (-2,0)--(-1.6,0.6928);
\draw[-latex,  ,rotate=120] (1.5,0.866)--(1.85,0.2598);			
\draw [  ,rotate=120] (0,0)--(-1.5,0.866);
\draw [  ,rotate=120] (0,0)--(1.5,0.866);
\draw[-latex,  ,rotate=120] (-1.5,0.866)--(-0.65,0.3753);

\draw [  ,rotate=120](-2,0) -- (-1.5,-0.866);
\draw [  ,rotate=120](2,0)-- (1.5,-0.866);
\draw[-latex,  ,rotate=120](-2,0)--(-1.6,-0.6928);
\draw[-latex,  ,rotate=120](1.5,-0.866)--(1.85,-0.2598);	
\draw [  ,rotate=120](0,0) -- (-1.5,-0.866);
\draw [  ,rotate=120](0,0)-- (1.5,-0.866);
\draw[-latex,  ,rotate=120](0,0)--(1,-0.5773);	

\draw[  ,rotate=120](-1.5,0.866)--(-1.5,0);
\draw[  ,rotate=120](-1.5,-0.866)--(-1.5,0);
\draw[-latex,  ,rotate=120](-1.5,0)--(-1.5,0.5);
\draw[-latex,  ,rotate=120](-1.5,0)--(-1.5,-0.5);

\draw[  ,rotate=120](1.5,0.866)--(1.5,0);
\draw[  ,rotate=120](1.5,-0.866)--(1.5,0);
\draw[-latex,  ,rotate=120](1.5,0)--(1.5,0.5);
\draw[-latex,  ,rotate=120](1.5,0)--(1.5,-0.5);

\draw [  ,rotate=120](-3,0) -- (-4,1);
\draw [  ,rotate=120](-5,0) -- (-4,1);		
\draw[-latex,  ,rotate=120] (-5,0)--(-4.45,0.55);
\draw[-latex,  ,rotate=120] (-4,1)--(-3.35,0.35);
\draw [  ,rotate=120](-3,0) -- (-4,-1);
\draw [  ,rotate=120](-5,0) -- (-4,-1);		
\draw[-latex,  ,rotate=120] (-5,0)--(-4.45,-0.55);
\draw[-latex,  ,rotate=120] (-4,-1)--(-3.35,-0.35);

\draw [  ,rotate=120](-6,0) -- (-7,1);
\draw [  ,rotate=120](6,0) -- (7,1);		
\draw [-latex,   ,rotate=120](-7,1) -- (-6.5,0.5);
\draw [-latex,   ,rotate=120](6,0) -- (6.5,0.5);

\draw [  ,rotate=120](-6,0) -- (-7,-1);
\draw [  ,rotate=120](6,0) -- (7,-1);		
\draw [-latex,   ,rotate=120](-7,-1) -- (-6.5,-0.5);
\draw [-latex,   ,rotate=120](6,0) -- (6.5,-0.5);

\draw [  ,rotate=120](3,0) -- (4,1);
\draw [  ,rotate=120](5,0) -- (4,1);		
\draw[-latex,  ,rotate=120] (4,1)--(4.65,0.35);
\draw[-latex,  ,rotate=120] (3,0)--(3.55,0.55);
\draw [  ,rotate=120](3,0) -- (4,-1);
\draw [  ,rotate=120](5,0) -- (4,-1);		
\draw[-latex,  ,rotate=120] (4,-1)--(4.65,-0.35);
\draw[-latex,  ,rotate=120] (3,-0)--(3.55,-0.55);

\draw[  ,rotate=120](4,1)--(4,0);
\draw[  ,rotate=120](4,-1)--(4,0);
\draw[-latex,  ,rotate=120](4,0)--(4,0.5);
\draw[-latex,  ,rotate=120](4,0)--(4,-0.5);

\draw[  ,rotate=120](-4,1)--(-4,0);
\draw[  ,rotate=120](-4,-1)--(-4,0);
\draw[-latex,  ,rotate=120](-4,0)--(-4,0.5);
\draw[-latex,  ,rotate=120](-4,0)--(-4,-0.5);

\draw[  ,rotate=120](-3,0)--(-2,0);
\draw[-latex,  ,rotate=120](-3,0)--(-2.3,0);
\draw[  ,rotate=120](-6,0)--(-5,0);
\draw[-latex,  ,rotate=120](-6,0)--(-5.25,0);
\draw[  ,rotate=120](3,0)--(2,0);
\draw[-latex,  ,rotate=120](2,0)--(2.5,0);
\draw[  ,rotate=120](6,0)--(5,0);
\draw[-latex,  ,rotate=120](5,0)--(5.4,0);

	\node  [below]  at (2.15,0) {\footnotesize$k_{4}$};
\node  [below]  at (2.95,0) {\footnotesize$k_{3}$};	
\node  [below]  at (5.05,0) {\footnotesize$k_{2}$};
\node  [below]  at (5.9,0) {\footnotesize$k_{1}$};

\node  [below]  at (-2.15,0) {\footnotesize$k_{5}$};
\node  [below]  at (-2.95,0) {\footnotesize$k_{6}$};	
\node  [below]  at (-5.05,0) {\footnotesize$k_{7}$};
\node  [below]  at (-5.9,0) {\footnotesize$k_{8}$};		
\end{tikzpicture}
\caption{\footnotesize   The whole jump contour $ \Sigma^{(3)}$ for the RH problem \ref{p1rhpM2}, which is obtained by  opening three  contours  $\omega^l I, \ l=0, 1, 2$.}
\label{pfv3}
\end{figure}
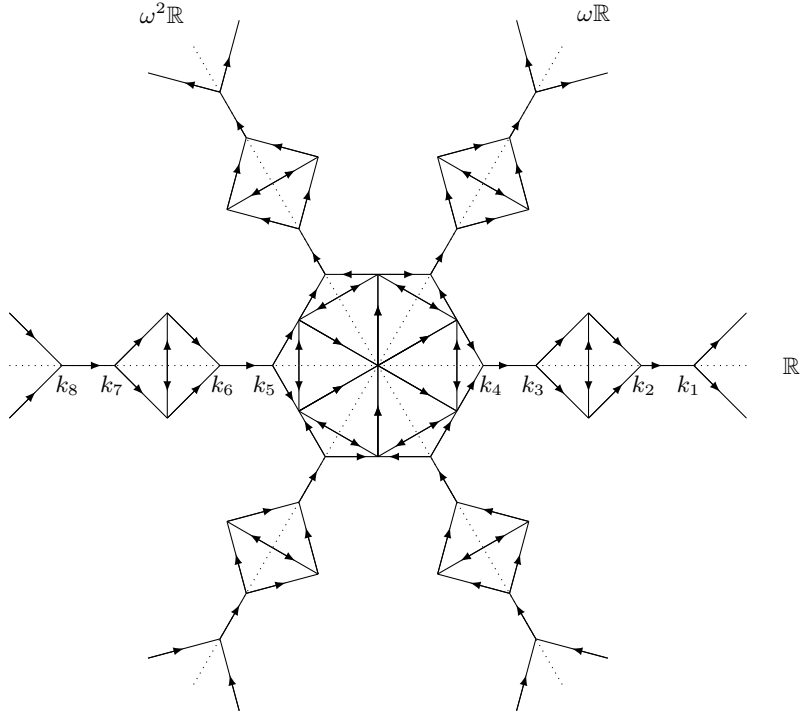

The  continuous extension functions $R_j(k),\ j =0, \cdots,8$   on $\Omega $ are defined as follows. Others
on $\Omega^*$ and $\omega^l(\Omega_j \cup \Omega_j^*), \ l=1,2$  can be obtained by symmetries.
\begin{lemma}
There exist functions $R_{j}(k): \ \bar{\Omega}_{j}\rightarrow\mathbb{C}, \ j=0,\cdots,8,$ continuous on $\bar{\Omega}_{j}$, with continuous first partial derivative
on $\Omega_j$, and boundary values
  \begin{align}
  &R_{j}(k)=\Bigg\{\begin{array}{ll}
    d(k) , \ k\in I_{j},\\
    d(k_{j}), \ k\in\Sigma_{j},
    \end{array}\label{R212}
  \end{align}
and
for $k\in\Omega_{j}, \ j=0,\cdots,8$, admit the estimates
\begin{align}
  |\bar{\partial}R_{j}(k)|& \leq c_1  +c_2 |\re k-k_j|^{-1/2}, \label{p1est0}\\
|\bar{\partial}R_{j}(k)|& \leq c_3    |\re k-k_j|^{1/2},\label{p1est1}\\
|\bar{\partial}R_{j}(k)|&\leq c_4. \label{p1est2}
\end{align}
Setting $R: \Omega \to \mathbb{C}$ by $R(k)|_{k\in\Omega_j} =R_{j}(k)$ and $R(k)|_{k\in\Omega_{j}^*} =R_{j}^*(k)$, the extension can preserve the symmetry  $R(k)=\overline{R(\bar{k}^{-1})}$.
\end{lemma}
\begin{proof}
	Without loss of generality, we give the brief proof for  $R_{1}(k)$ and other cases for $R_j(k),\ j =0, 2,\cdots,8$, can be given similarly.

We define the function
\begin{equation}\label{r11}
 R_{1}(k) := (d(\re k)-d(k_1)) \cos \left( \frac{ \pi \arg(k-k_1) \chi\left(\arg(k-k_1)\right)}{2 \varphi}\right) +d (k_1),
\end{equation}
where $\chi \in C_{0}^\infty (\mathbb{R})$ is a  cut-off function satisfying
\begin{equation}\label{chi}
\chi(k) := \begin{cases}
0, \quad k <  \frac{\varphi}{3},\\
1, \quad k > \frac{2\varphi}{3}.
\end{cases}
\end{equation}
Then   $R_{1}(k) $ satisfies the boundary condition (\ref{R212}) for $j=1$.
Next, we show that  $R_{1}(k) $ admits the estimates (\ref{p1est0})-(\ref{p1est2}).

Note that for $k=k_1+ l e^{\mathrm{i} \varphi_0}=u+   \mathrm{i} v  \in \Omega_{1}, \ \varphi_0 < \varphi$, we have
$$\bar\partial = \frac{e^{\mathrm{i} \varphi_0}}{2} (\partial_l + \mathrm{i} l^{-1} \partial_{\varphi_0}),$$
which acts   on  $ R_{1}(k)$ in \eqref{r11} yields
\begin{equation}
 \bar \partial R_{1}(k) = \frac{1}{2} d'(u) \cos \left( \frac{\pi \varphi_0 \chi(\varphi_0)}{2 \varphi}\right) - \frac{\mathrm{i} e^{\mathrm{i} \varphi_0} }{2 l}(d(u)-d(k_1))  \frac{\pi \chi'(\varphi_0)}{2\varphi} \sin \left(\frac{\pi \varphi_0\chi(\varphi_0)}{2\varphi} \right). \label{ieow}
\end{equation}

By H\"{o}lder's inequality,   we   have
\begin{equation}\label{r21}
|d(u)-d(k_1)|  =\bigg| \int_{k_1}^u d'(s)ds  \bigg|  \le ||d'||_{L^2(\mathbb{R})} |u-k_1|^{1/2},
\end{equation}
which together with (\ref{ieow})  leads to \eqref{p1est0}.

Noting $d'(k_1)=0$,  and using \eqref{r21} along with H\"{o}lder's inequality once again,  we obtain
\begin{align}
&|d(u)-d(k_1)|  \le ||d'' ||_{L^2(\mathbb{R})} |u-k_1|^{3/2},\\
& |d'(u)|   =\bigg| \int_{k_1}^u d''(s)ds  \bigg|  \le ||d''||_{L^2(\mathbb{R})} |u-k_1|^{1/2}.
\end{align}
Substituting this into  (\ref{ieow})  gives \eqref{p1est1}.

Moreover, direct calculations show that
\begin{equation}
|d(u)-d(k_1)|  =\bigg| \int_{k_1}^u d'(s)ds  \bigg| \le ||d' ||_{L^\infty(\mathbb{R})} |u-k_1|,
\end{equation}
which together with (\ref{ieow}) gives \eqref{p1est2}.

\end{proof}

Using $R_{j}(k), \ j=0, \cdots,8$, we define a matrix function
\begin{align}
\mathcal{R}^{(2)}(k)=\left\{
\begin{aligned}
    &\left(\begin{array}{ccc}1 & 0 & 0 \\ - R_{j}(k)e^{-\mathrm{i}t\theta_{12}(k)} & 1 & 0 \\ 0 & 0 & 1 \end{array}  \right), \quad k\in\Omega_j, \\
    &\left(\begin{array}{ccc}1 &  -R_j^*(k)e^{\mathrm{i}t\theta_{12}(k)} & 0 \\ 0 & 1 & 0 \\ 0 & 0 & 1 \end{array}  \right), \quad k\in\Omega_j^*,\\
    &\left(\begin{array}{ccc} 1 & 0 & -R_{j}(\omega k)e^{\mathrm{i}t\theta_{13}(k)} \\ 0 & 1 & 0 \\ 0 & 0 & 1 \end{array}  \right), \ k\in \omega\Omega_{j},\\
    &\left(\begin{array}{ccc} 1 & 0 & 0 \\ 0 & 1 & 0 \\-R_j^*(\omega k)e^{-\mathrm{i}t\theta_{13}(k)} & 0 & 1\end{array}  \right), \ k\in \omega\Omega_{j}^{*}, \\
    &\left(\begin{array}{ccc} 1 & 0 & 0 \\ 0 & 1 & 0 \\ 0 & -R_{j}(\omega^2 k)e^{-\mathrm{i}t\theta_{23}(k)} & 1 \end{array}  \right), \ k\in \omega^2\Omega_{j}, \\
    &\left(\begin{array}{ccc} 1 & 0 & 0 \\ 0 & 1 & -R_j^*(\omega^2 k)e^{\mathrm{i}t\theta_{23}(k)} \\ 0 & 0 & 1\end{array}  \right), \quad k\in \omega^2\Omega_{j}^{*},  \\
    &I, \quad elsewhere.
\end{aligned}
\right.
\end{align}
Then, we make a transformation
\begin{equation}\label{p1m1m2}
m^{(3)}(k) = m^{(2)}(k) \mathcal{R}^{(2)}(k),
\end{equation}
which satisfies the following RH problem.

\begin{RHP}\label{p1rhpM2}
Find a  row vector-valued function  $m^{(3)}(k):= m^{(3)}(k;y,t)$ such that
\begin{itemize}
\item $m^{(3)}(k)$ is continuous in
$\mathbb{C}\setminus \Sigma^{(3)}$.

\item $m^{(3)}(k)$ has continuous boundary values $m^{(3)}_\pm(k)$ on $\Sigma^{(3)}$ and
\begin{equation}
	m^{(3)}_+(k)=m^{(3)}_-(k)V^{(3)}(k),\hspace{0.5cm}k \in \Sigma^{(3)},
\end{equation}
where
\begin{align}
&V^{(3)}(k)=\left\{
\begin{aligned}
    & V^{(2)}(k), \quad k \in \mathop{\cup}\limits_{l=0}^{2}    \omega^l (\mathbb{R} \setminus I),\\
      &\mathcal{R}^{(2)}(k)|_{k\in\Omega_{il}}^{-1}\mathcal{R}^{(2)}(k)|_{k\in\Omega_{jl}}, \quad k\in \mathop{\cup}\limits_{l=0}^{2}    \omega^l\tilde{\Sigma}, \\
  &\mathcal{R}^{(2)}(k)|_{k\in\Omega_{jl}}^{-1}\mathcal{R}^{(2)}(k)|_{k\in\Omega_{il}}, \quad k \in   \mathop{\cup}\limits_{l=0}^{2}    \omega^l\tilde{\Sigma}^*,\\
    & \mathcal{R}^{(2)}(k')^{-1}, \quad k\in  \mathop{\cup}\limits_{l=0}^{2}    \omega^l\Sigma,\\
    & \mathcal{R}^{(2)}(k'), \quad k\in\mathop{\cup}\limits_{l=0}^{2}    \omega^l\Sigma^*.\label{p1v2}
\end{aligned}
\right.
\end{align}
\item
	$m^{(3)}(k) =\begin{pmatrix} 1& 1& 1 \end{pmatrix}+\mathcal{O}(k^{-1}),\hspace{0.5cm}k \rightarrow \infty.$
\item For $k\in\mathbb{C}$, we have
\begin{align}
	\bar{\partial}m^{(3)}(k)=m^{(3)}(k)\bar{\partial}\mathcal{R}^{(2)}(k),
\end{align}
where
\begin{align} \label{dbarr2}
\bar{\partial} \mathcal{R}^{(2)}(k)=\left\{
\begin{aligned}
    &\left(\begin{array}{ccc}0 & 0 & 0 \\  -\bar{\partial}R_{j}(k)e^{-\mathrm{i}t\theta_{12}(k)} & 0 & 0 \\ 0 & 0 & 0 \end{array}  \right), \quad k\in\Omega_{j}, \\
    &\left(\begin{array}{ccc}0 & - \bar{\partial}R_{j}^*(k)e^{\mathrm{i}t\theta_{12}(k)} & 0 \\ 0 & 0 & 0 \\ 0 & 0 & 0 \end{array}  \right), \quad k\in\Omega_{j}^*,\\
    &\left(\begin{array}{ccc} 0 & 0 & -\bar{\partial}R_{j}(\omega k)e^{\mathrm{i}t\theta_{13}(k)} \\ 0 & 0 & 0 \\ 0 & 0 & 0 \end{array}  \right), \ k\in \omega\Omega_{j},\\
    &\left(\begin{array}{ccc} 0 & 0 & 0 \\ 0 & 0 & 0 \\ -\bar{\partial}R_{j}^*(\omega k)e^{-\mathrm{i}t\theta_{13}(k)} & 0 & 0\end{array}  \right), \ k\in \omega\Omega_{j}^*, \\
    &\left(\begin{array}{ccc} 0 & 0 & 0 \\ 0 & 0 & 0 \\ 0 & -\bar{\partial}R_{j}(\omega^2 k)e^{-\mathrm{i}t\theta_{23}(k)} & 0 \end{array}  \right), \ k\in\omega^2\Omega_{j}, \\
    &\left(\begin{array}{ccc} 0 & 0 & 0 \\ 0 & 0 & -\bar{\partial}R_{j}^*(\omega^2 k)e^{\mathrm{i}t\theta_{23}(k)} \\ 0 & 0 & 0\end{array}  \right), \quad k\in \omega^2\Omega_{j}^*, \\
    &0, \quad elsewhere.
\end{aligned}
\right.
\end{align}

\end{itemize}	
\end{RHP}

We decompose the hybrid $\bar{\partial}$-RH problem \ref{p1rhpM2} as follows:
\begin{equation}\label{p1m3m2}
  m^{(3)}(k)= m^{(4)}(k)M^{rhp}(k),
\end{equation}
where $m^{(4)}(k)$ is the solution of a pure $\bar{\partial}$-problem that will be solved in Subsection \ref{p1pdbar}, and $M^{rhp}(k)$ satisfies the following pure RH problem.
\begin{RHP}\label{p1purerhp}
Find a  matrix-valued function  $M^{rhp}(k):= M^{rhp}(k;y,t)$ such that
\begin{itemize}
\item $M^{rhp}(k)$ is analytic in
$\mathbb{C}\setminus \Sigma^{(3)}$.
\item $M^{rhp}(k)$ has continuous boundary values $M^{rhp}_\pm(k)$ on $\Sigma^{(3)}$ and
\begin{equation}
	M^{rhp}_+(k)=M^{rhp}_-(k)V^{(3)}(k),\hspace{0.5cm}k \in \Sigma^{(3)},
\end{equation}
where $V^{(3)}(k)$ is defined by \eqref{p1v2}.
\item
	$M^{rhp}(k) =I+\mathcal{O}(k^{-1}),\hspace{0.5cm}k \rightarrow \infty.$
\end{itemize}	
\end{RHP}

\subsection{Asymptotic analysis on a pure  RH problem} \label{p1prh}


Define   small disks   
$$\mathrm{U}_{jl} := \{k \in \mathbb{C}: | k- \omega^l k_j| \le c_0  \}, \, j\in \{a,b,c,d\},\, l=0, 1,2$$
 around    critical points   $ \omega^l k_j$, where
\begin{equation}\label{p1c0}
c_0:= \min \left\{\frac{\sqrt{7}-\sqrt{3}}{4},  2|k_1-k_a| t^{\delta_1}, 2|k_3-k_c|t^{\delta_1} \right\},
\end{equation}
with  $\frac{1}{27}<\delta_1<\frac{1}{12}$.

In the transition zone $\mathcal{T}_1^R$,  it follows from (\ref{hjfd1})-(\ref{hjfd2}) that,  for   $t$  large enough,  the saddle points $w^l k_j,\ j=1,\cdots, 8$  fall  in
$\mathrm{U} :=\mathop{\cup}\limits_{j \in \{a,b,c,d\}} (\mathrm{U}_{j0} \cup \mathrm{U}_{j1} \cup \mathrm{U}_{j2})$.
Moreover,   we have
	\begin{align*}
	&|k_1 - k_a|, |k_2 -  k_a|, |k_3 - k_b|,|k_4 -  k_b| \le \sqrt{2C} t^{-1/3}, \\
	&| k_5 -  k_c|, | k_6 -  k_c|, | k_7 -  k_d|,| k_8 - k_d| \le \sqrt{2C} t^{-1/3},
\end{align*}
which  reveals that
\begin{align}
c_0 \lesssim t^{\delta_1-1/3} \to 0, \ t \to \infty. \label{cc1}
\end{align}

Now we construct the solution $M^{rhp}(k)$ as follows:
\begin{align}\label{p1pdesM2RHP}
M^{rhp}(k)=\left\{
\begin{aligned}
    &E(k), \quad k\notin \mathrm{U},\\
    &E(k)M^{loc}(k), \quad k\in \mathrm{U},
\end{aligned}
\right.
\end{align}
where  $M^{loc}(k)$ is the solution of a local model, and the error function $E(k)$ is the solution of a small-norm RH problem.
Using  (\ref{p1v2}) and  \eqref{p1c0},  also recalling the definition \eqref{oritheta12} of $\theta_{12}(k)$, we have
\begin{equation}
 \| V^{(3)}(k)-I\|_{L^\infty(\Sigma^{(3)} \setminus \mathrm{U})} =\mathcal{O}(e^{-ct}),
\end{equation}
where $c$ is a positive constant.
This estimate implies the necessity of constructing a local model  within $\mathrm{U}$.

\subsubsection{Local models}\label{p1lomod}

Denote the local jump contour $\Sigma^{loc} := \Sigma^{(3)}\cap \mathrm{U}$.
Then  the solution $M^{loc}(k)$ is approximated by the sum of the separate local models in the neighborhood
 of $\mathrm{U}_{jl}, \ j\in \{a,b,c,d\}, l=0,1,2$.
On each contour  $\Sigma_{jl} := \Sigma^{(3)} \cap \mathrm{U}_{jl},\ l=0,1,2$, we define the local models  $M_{jl}(k), \ j\in \{a,b,c,d\},\, l=0,1,2$. Taking RH problem for $M_{j0}(k)$, $j=a,b,c,d$,  as an example, the other cases can be given similarly.

\begin{RHP}\label{p1mlolj}
Find a $3\times 3$ matrix-valued function $M_{j0}(k):=M_{j0}(k;y,t)$ such that
\begin{itemize}
\item $M_{j0}(k)$ is analytic in $\mathbb{C} \setminus \Sigma_{j0}$.
\item For $k \in \Sigma_{j0}$, $M_{j0,+}(k)=M_{j0,-}(k)V_{j0}(k)$
where $V_{j0}(k) = V^{(3)}(k)|_{k \in \Sigma_j }$.
\item As  $k \to \infty$  in $\mathbb{C} \setminus \Sigma_j$,
$M_{j0}(k)=I+\mathcal{O}(k^{-1}).$
\end{itemize}
\end{RHP}

  Next, we show that   each  local model $M_{j0}(k)$ can asymptotically  match   the model RH problem for  $M^{L}(\hat k)$ in \ref{appx},  which is equivalent to the  Painlev\'e model. For this purpose, we
introduce the following localized scaling variables.
\begin{itemize}
\item For $k$ close to $k_a$,
\begin{equation}\label{ttheta12}
t\theta_{12}(k) = t\theta_{12}(k_a) - \frac{8}{3}\hat{k}^3 -2 s \hat{k} +\mathcal{O}(\hat{k}^4 t^{-\frac{1}{3}}),
\end{equation}
where
\begin{align}
 & t\theta_{12}(k_a)=\frac{\sqrt{3}}{4}(4y- 3t),\ \hat{k}= c_a t^{ \frac{1}{3}} (k-k_a),\label{scal1}\\
  &s = \frac{ 2^{2/3}(-7+\sqrt{21})}{3^{2/3} (98-21\sqrt{21})^{1/3}} \left( \hat{\xi} +\frac{3}{8}\right) t^{\frac{2}{3}},
\end{align}
with
\begin{equation}\label{p1ca}
c_a =\frac{3^{2/3} }{2^{8/3}}(98-21\sqrt{21})^{1/3}.
\end{equation}

\item For $k$ close to $k_b$,
\begin{equation}\label{p1akb}
t\theta_{12}(k) = t\theta_{12}(k_b) - \frac{8}{3}\check{k}^3 -2 s \check{k} +\mathcal{O}(\check{k}^4 t^{-\frac{1}{3}}),
\end{equation}
where
\begin{equation}\label{p1akb1}
  t\theta(k_{b}) = -t\theta(k_{a}), \quad \check{k}= c_b t^{ \frac{1}{3}} (k-k_b),
\end{equation}
and $s$ is defined as  (\ref{scal1})
with
\begin{equation}\label{p1cb}
c_b= \frac{3^{2/3} }{2^{8/3}}(98+21\sqrt{21})^{1/3}.
\end{equation}
\item For  $k$ close to $k_c$, following the symmetry $k_c=-k_b$,
 \begin{equation}\label{p1akc1}
t\theta_{12}(k) = t\theta_{12}(k_c) - \frac{8}{3}\tilde{k}^3 -2 s \tilde{k} +\mathcal{O}(\tilde{k}^4 t^{-\frac{1}{3}}),
\end{equation}
where $s$ is defined as  (\ref{scal1}) and
\begin{equation}\label{p1akc2}
 \tilde{k}= c_c t^{ \frac{1}{3}} (k-k_c),\quad \theta_{12}(k_c) =- \theta_{12}(k_b), \quad c_c= c_b.
\end{equation}

\item For  $k$ close to $k_d$, following the symmetry $k_d=-k_a$,
 \begin{equation}\label{p1akd1}
t\theta_{12}(k) = t\theta_{12}(k_d) - \frac{8}{3}\breve{k}^3 -2 s\breve{k} +\mathcal{O}(\breve{k}^4 t^{-\frac{1}{3}}),
\end{equation}
where $s$ is defined as  (\ref{scal1}) and
\begin{equation}\label{p1akd2}
  \breve{k}= c_d t^{ \frac{1}{3}} (k-k_d),\quad \theta_{12}(k_d) =- \theta_{12}(k_a), \quad c_d = c_a.
\end{equation}

\end{itemize}

 As an illustrative example,  we use the  local model RH problem  \ref{p1mlolj}  for $M_{a0}(k)$  to match  the model RH problem in \ref{appx}. Other local models can be constructed in a similar manner.


\subparagraph{Step I: Scaling.}
Under   a  new scaled variable $\hat{k}  = c_a t^{\frac{1}{3}}(k-k_a),$
 the contour $\Sigma_{a0}$ is changed into
 a contour $\hat{\Sigma}_{a}$ in the $\hat{k}$-plane
\begin{equation*}
  \hat{\Sigma}_{a} := \left( \mathop{\cup}\limits_{j=1}^2 (\hat{\Sigma}_{j} \cup \hat{\Sigma}_{j}^*) \right) \cup (\hat{k}_1, \hat{k}_2),
\end{equation*}
where
\begin{align*}
&\hat{k}_j = c_a t^{\frac{1}{3}}(k_j-k_a),\ j=1,2; \ \ \hat{\Sigma}_{1} = \{\hat{k}: \hat{k}-\hat{k}_1 = l e^{\mathrm{i}  \varphi  }, 0\le l \le c_0  c_a t^{\frac{1}{3}}\},\\
   & \hat{\Sigma}_{2} = \{\hat{k}: \hat{k}-\hat{k}_2 = l e^{\mathrm{i} (\pi-\varphi) }, 0\le l \le c_0 c_a t^{\frac{1}{3}}\}.
\end{align*}

Further,   RH problem  \ref{p1mlolj}   becomes   the following RH problem in the $\hat{k}$-plane.

\begin{RHP}\label{p1mloc2}
Find a $3\times 3$ matrix-valued function $ {M}_{a0}(\hat{k}):=  {M}_{a0}(\hat{k}; y,t)$ such that
\begin{itemize}
\item  $ {M}_{a0}(\hat{k})$ is analytic in $\mathbb{C} \setminus \hat{\Sigma}_{a}$.
\item  For  $\hat{k}\in \hat{\Sigma}_{a}$, $ {M}_{a0,+}(\hat{k})= {M}_{a0,-}(\hat{k}) {V}_{a}(\hat{k})$, where
\begin{equation}
 {V}_a(\hat{k}) = \begin{cases}
            \left(\begin{array}{ccc} 1&0&0 \\ d(k_j)e^{-\mathrm{i} t\theta_{12}(k_a)} e^{-\mathrm{i}t\theta_{12} (c_a^{-1} t^{-\frac{1}{3}}\hat{k} +k_a) }&1&0 \\ 0&0&1\end{array}  \right), \, \hat{k} \in \hat{\Sigma}_{j}, \, j=1,2,\\
             \left(\begin{array}{ccc} 1& -\bar{d}(k_j) e^{\mathrm{i} t\theta_{12}(k_a)} e^{\mathrm{i}t\theta_{12} (c_a^{-1} t^{-\frac{1}{3}}\hat{k} +k_a) }&0 \\ 0&1&0 \\ 0&0&1\end{array}  \right), \,  \hat{k} \in \hat{\Sigma}_{j}^*, \, j=1,2,\\
             V^{(3)}(c_a^{-1} t^{-\frac{1}{3}}\hat{k} +k_a), \,  \hat{k} \in (\hat{k}_1, \hat{k}_2).
           \end{cases}
\end{equation}
\item  $ {M}_{a0}(\hat{k})=I+\mathcal{O}(\hat{k}^{-1}), \quad  \hat{k}\rightarrow\infty.$
\end{itemize}
\end{RHP}

\subparagraph{Step \uppercase\expandafter{\romannumeral2}: Matching   the model RH problem.}

 According to (\ref{ttheta12}),  we show the following proposition.
\begin{proposition} \label{p1mathch}
As $t \to \infty$,
\begin{equation}\label{mp1}
  {M}_{a0}(\hat{k}) =    \mathcal{A}  M^\mathrm{L}(\hat{k})  \mathcal{A}^{-1} + \mathcal{O}(t^{-\frac{1}{3}+4 \delta_1}),
\end{equation}
where $M^{\mathrm{L}}(\hat{k})$ is the solution of RH problem \ref{modelp2} with  $c_1=\mathrm{i}|d(k_a)|,$
and
\begin{equation}\label{hA}
\mathcal{A} =   \left(\begin{array}{ccc}
e^{-\mathrm{i} \left(\frac{\varphi_a}{2} -\frac{\pi}{4} \right)} & 0 &0 \\
0 & e^{\mathrm{i} \left(\frac{\varphi_a}{2} -\frac{\pi}{4} \right)} & 0 \\
0 & 0 & 1
\end{array} \right)
\end{equation}
with $\varphi_a = \arg d(k_a)-t\theta_{12}(k_a)$.
\end{proposition}

\begin{proof}
Let
 $\hat{M}(\hat{k}) = \mathcal{A}^{-1} {M}_{a0}(\hat{k}) \mathcal{A}$, which satisfies the jump relation
 $$\hat{M}_+(\hat{k})=\hat{M}_-(\hat{k}) \hat{V}(\hat{k}), $$
 where   the jump matrix is  $ \hat{V}(\hat{k})=  \mathcal{A}^{-1}  {V}_a(\hat{k}) \mathcal{A}$.
 Next, we show that $\hat{M}(\hat{k})$ can be approximated by $M^\mathrm{L}(\hat{k}) $  of  the model RH problem \ref{modelp2}.
It is enough to estimate $\hat{V}(\hat{k}) - V^{\mathrm{L}}(\hat{k})$.

 For $\hat{k} \in  [\hat k_2, \hat k_1]$,  noticing that  $|e^{\mathrm{i}(\frac{8\hat{k}^3}{3} +2 s \hat{k})}|=|e^{-\mathrm{i}t\theta_{12}(k)}| = 1$,  we get
 \begin{align}
& \left| \hat{V}(\hat{k}) - V^{\mathrm{L}}(\hat{k}) \right| \le
  \left| \hat{d}(\hat{k})e^{-\mathrm{i}t\theta_{12}}-d(k_a)e^{\mathrm{i}(\frac{8\hat{k}^3}{3} + 2s \hat{k})} \right| \lesssim \left| \hat{d}(\hat{k}) - d(k_a) \right|\nonumber\\
  & \le  \Vert d'(k_a)\Vert_{L^\infty(\Sigma^{\mathrm{L}}_{5})} \left|   c_a^{-1} t^{-\frac{1}{3}}\hat{k} \right| \lesssim  t^{-\frac{1}{3}}\hat{k}. \label{pier}
 \end{align}
Further by  (\ref{cc1}) and (\ref{scal1}), we have
\begin{equation}
|\hat{k}| =|c_a t^{ \frac{1}{3}} (k-k_a)| \lesssim t^{\delta_1}, \label{ooero}
\end{equation}
 which together with (\ref{pier})   gives the estimate
\begin{equation}
   \left| \hat{V}(\hat{k}) - V^{\mathrm{L}} (\hat{k})\right|\lesssim t^{-\frac{1}{3}+\delta_1}.\label{gg1}
\end{equation}
For $\hat{k} \in  \hat\Sigma_{1}$,    since   $ {\rm{Re}} \left( \mathrm{i}(\frac{8\hat{k}^3}{3} +2 s \hat{k})\right)<0  $,
 by using  (\ref{ooero}),  we have
\begin{align*}
	\left|e^{-\mathrm{i}t\theta_{12}} -e^{\mathrm{i}(\frac{8\hat{k}^3}{3} + 2s \hat{k})} \right| \le  \left| e^{\mathcal{O}(\hat{k}^4t^{-1/3})} -1 \right| \lesssim  t^{-\frac{1}{3}+4\delta_1},
	\end{align*}
by which we  obtain that
\begin{equation}
 \left|\hat{V}(\hat{k}) - V^{\mathrm{L}}(\hat{k}) \right| \le \left| |d(k_1)|e^{-\mathrm{i}t\theta_{12}}-|d(k_a)|e^{\mathrm{i}(\frac{8\hat{k}^3}{3} +2 s \hat{k})} \right| \lesssim t^{-\frac{1}{3}+4\delta_1}. \label{gg2}
\end{equation}
 In a similar way, we have the estimate
\begin{equation*}
 \left|\hat{V}(\hat{k}) - V^{\mathrm{L}}(\hat{k}) \right|   \lesssim t^{-\frac{1}{3}+4\delta_1}, \  \hat k\in \hat\Sigma_2\cup \hat\Sigma_1^*\cup\hat\Sigma_2^*,
\end{equation*}
 which together with (\ref{gg1}) and (\ref{gg2})  gives
\begin{equation*}
 \left|\hat{V}(\hat{k}) - V^{\mathrm{L}}(\hat{k}) \right|   \lesssim t^{-\frac{1}{3}+4\delta_1}, \  \hat k\in \hat\Sigma_a.
\end{equation*}
By using the small-norm theorem, we obtain the   relation  (\ref{mp1}).

\end{proof}

As a corollary of Proposition \ref{p1mathch}, we have the following result.

\begin{corollary} As $\hat{k} \to \infty$,
\begin{equation}
  {M}_{a0}(\hat{k}) =  I + \frac{ {M}_{a0}^{(1)} (s)}{\hat{k}} + \mathcal{O}(\hat{k}^{-2}),
 \end{equation}
 where
 \begin{equation}\label{p1ma1}
 {M}_{a0}^{(1)} (s) = \frac{\mathrm{i}}{2}    \left(\begin{array}{ccc} -\int_s^\infty v^2(\varsigma) \mathrm{d} \varsigma & v(s)e^{-\mathrm{i}\varphi_a}&0 \\ -v(s)e^{\mathrm{i}\varphi_a} &\int_s^\infty v^2(\varsigma) \mathrm{d}\varsigma&0 \\ 0&0&0\end{array}  \right) + \mathcal{O}(t^{-\frac{1}{3}+4 \delta_1}),
 \end{equation}
 with $v(s)$ be the unique solution of Painlev\'{e} \uppercase\expandafter{\romannumeral2} equation \eqref{p23}, fixed by the boundary condition
 \begin{equation}
 v(s) \sim - |r(k_a)| \mathrm{Ai}(s), \quad s \to +\infty.
 \end{equation}
\end{corollary}

In a similar way to $ {M}_{a0} (\hat{k})$,  with the help of \eqref{p1akb}-\eqref{p1cb},  we obtain
\begin{equation}
  {M}_{b0}(\check{k}) =  I + \frac{ {M}_{b0}^{(1)}(s)}{\check{k}} + \mathcal{O}(\check{k}^{-2}), \quad \check{k} \to \infty
 \end{equation}
where
\begin{equation}\label{p1mb1}
 {M}_{b0}^{(1)}(s)= \frac{\mathrm{i}}{2}    \left(\begin{array}{ccc} -\int_s^\infty v^2(\varsigma) \mathrm{d} \varsigma & v(s)e^{-\mathrm{i}\varphi_b}&0 \\ -v(s)e^{\mathrm{i}\varphi_b} &\int_s^\infty v^2(\varsigma) \mathrm{d}\varsigma&0 \\ 0&0&0\end{array}  \right) + \mathcal{O}(t^{-\frac{1}{3}+4 \delta_1}),
 \end{equation}
with  the argument $\varphi_b =  \arg d(k_b)-t \theta(k_b)$.
With the help of \eqref{p1akc1}-\eqref{p1akc2}, we have
\begin{equation}
 {M}_{c0}(\tilde{k}) =  I + \frac{ {M}_{c0}^{(1)}(s)}{\tilde{k}} + \mathcal{O}(\tilde{k}^{-2}), \quad \tilde{k} \to \infty
\end{equation}
where
\begin{equation}\label{p1mc1}
 {M}_{c0}^{(1)}(s)= \frac{\mathrm{i}}{2}    \left(\begin{array}{ccc} -\int_s^\infty v^2(\varsigma) \mathrm{d} \varsigma & v(s)e^{-\mathrm{i}\varphi_c}&0 \\ -v(s)e^{\mathrm{i}\varphi_c} &\int_s^\infty v^2(\varsigma) \mathrm{d}\varsigma&0 \\ 0&0&0\end{array}  \right) + \mathcal{O}(t^{-\frac{1}{3}+4 \delta_1}),
\end{equation}
with  the argument $\varphi_c =  \arg d(k_c)-t \theta(k_c)$.
With the help of \eqref{p1akd1}-\eqref{p1akd2}, we have
\begin{equation}
  {M}_{d0} ( \breve{k}) =  I + \frac{ {M}_{d0}^{(1)} (s)}{ \breve{k}} + \mathcal{O}( \breve{k}^{-2}), \quad  \breve{k} \to \infty
\end{equation}
where
\begin{equation}\label{p1mc3}
 {M}_{d0}^{(1)}(s)= \frac{\mathrm{i}}{2}    \left(\begin{array}{ccc} -\int_s^\infty v^2(\varsigma) \mathrm{d} \varsigma & v(s)e^{-\mathrm{i}\varphi_d}&0 \\ -v(s)e^{\mathrm{i}\varphi_d} &\int_s^\infty v^2(\varsigma) \mathrm{d}\varsigma&0 \\ 0&0&0\end{array}  \right) + \mathcal{O}(t^{-\frac{1}{3}+4 \delta_1}),
\end{equation}
with  the argument $\varphi_d =  \arg d(k_d)-t \theta(k_d)$.

Following the above steps, we also can obtain other solutions $M_{jl}(k), \ j \in \{ a,b,c,d\},\ l=1,2$.
Finally, the solution $M^{loc}(k)$ for the local model is  reconstructed as follows.
\begin{proposition}\label{promloc}As $t \to \infty$,
\begin{equation}\label{p1mlo}
  M^{loc}(k) = I + t^{-1/3} M^{loc}_1(k,s)+\mathcal{O}(t^{-2/3+4 \delta_1}),
\end{equation}
where
\begin{equation}
 M^{loc}_1(k,s)=	   \sum\limits_{j=a,b,c,d}  \left( \frac{M_{j }^{(1)}(s)}{c_j(k-   k_j)} +     \frac{\omega \Gamma_3  \overline{M_{j }^{(1)}(s)}  \Gamma_3 }{c_j(k- \omega  k_j)} +
 \frac{   \omega ^2 \Gamma_2  \overline{M_{j }^{(1)}(s)}  \Gamma_2}{c_j(k- \omega^2 k_j)}  \right),
\end{equation}
in which  $c_j$ are given by \eqref{p1ca} and \eqref{p1cb} respectively,  while for $j=a,b,c,d$,
 \begin{equation}\label{p1mj01}
  M_{j }^{(1)}(s) = \frac{\mathrm{i}}{2}    \left(\begin{array}{ccc} -\int_s^\infty v^2(\varsigma) \mathrm{d} \varsigma & v(s)e^{-\mathrm{i}\varphi_j}&0 \\ -v(s)e^{\mathrm{i}\varphi_j} &\int_s^\infty v^2(\varsigma) \mathrm{d}\varsigma&0 \\ 0&0&0\end{array}  \right).
  \end{equation}

\end{proposition}
\subsubsection{Small-norm RH problem} \label{p1sne}

Denote the contour
\begin{equation}
\Sigma^{E} :=\Big(\Sigma^{(3)}\setminus \mathrm{U} \Big)\cup\partial \mathrm{U}.
\end{equation}
See Figure \ref{p2fige}. Then, the error function $E(k)$ defined by
\eqref{p1pdesM2RHP} satisfies the following
RH problem.

\begin{figure}[http]
	\centering
	\begin{tikzpicture}[scale=0.7]
	\draw[dotted](-7,0)--(7,0);
	\node  [right]  at (7.5,0) {\footnotesize$\mathbb{R}$};	
\node    at (4.1,6.7) {\footnotesize$\omega^2\mathbb{R}$};
\node    at (-4.1,6.7) {\footnotesize$\omega^2\mathbb{R}$};

	\draw [  ] (-2,0)--(-1.5,0.866);
	\draw [  ](2,0)-- (1.5,0.866);
	\draw[-latex,  ] (-2,0)--(-1.6,0.6928);
	\draw[-latex,  ] (1.5,0.866)--(1.85,0.2598);			
	\draw [  ] (0,0)--(-1.5,0.866);
	\draw [  ] (0,0)--(1.5,0.866);
	\draw[-latex,  ] (-1.5,0.866)--(-0.65,0.3753);
	\draw[-latex,  ] (0, 0)--(1, 0.5773);

	\draw [  ](-2,0) -- (-1.5,-0.866);
	\draw [  ](2,0)-- (1.5,-0.866);
	\draw[-latex,  ](-2,0)--(-1.6,-0.6928);
	\draw[-latex,  ](1.5,-0.866)--(1.85,-0.2598);	
	\draw [  ](0,0) -- (-1.5,-0.866);
	\draw [  ](0,0)-- (1.5,-0.866);
	\draw[-latex,  ](-1.5,-0.866)--(-0.65,-0.3753);
	\draw[-latex,  ](0,0)--(1,-0.5773);	
	
	\draw[  ](-1.5,0.866)--(-1.5,0);
	\draw[  ](-1.5,-0.866)--(-1.5,0);
	\draw[-latex,  ](-1.5,0)--(-1.5,0.5);
	\draw[-latex,  ](-1.5,0)--(-1.5,-0.5);

	\draw[  ](1.5,0.866)--(1.5,0);
	\draw[  ](1.5,-0.866)--(1.5,0);
	\draw[-latex,  ](1.5,0)--(1.5,0.5);
	\draw[-latex,  ](1.5,0)--(1.5,-0.5);

	\draw [  ](-3,0) -- (-4,1);
	\draw [  ](-5,0) -- (-4,1);		
	\draw[-latex,  ] (-5,0)--(-4.45,0.55);
	\draw[-latex,  ] (-4,1)--(-3.35,0.35);
	\draw [  ](-3,0) -- (-4,-1);
	\draw [  ](-5,0) -- (-4,-1);		
	\draw[-latex,  ] (-5,0)--(-4.45,-0.55);
	\draw[-latex,  ] (-4,-1)--(-3.35,-0.35);

	\draw [  ](-6,0) -- (-7,1);
	\draw [  ](6,0) -- (7,1);		
	\draw [-latex,   ](-7,1) -- (-6.5,0.5);
	\draw [-latex,   ](6,0) -- (6.5,0.5);

	\draw [  ](-6,0) -- (-7,-1);
	\draw [  ](6,0) -- (7,-1);		
	\draw [-latex,   ](-7,-1) -- (-6.5,-0.5);
	\draw [-latex,   ](6,0) -- (6.5,-0.5);
	
	\draw [  ](3,0) -- (4,1);
	\draw [  ](5,0) -- (4,1);		
	\draw[-latex,  ] (4,1)--(4.65,0.35);
	\draw[-latex,  ] (3,0)--(3.55,0.55);
	\draw [  ](3,0) -- (4,-1);
	\draw [  ](5,0) -- (4,-1);		
	\draw[-latex,  ] (4,-1)--(4.65,-0.35);
	\draw[-latex,  ] (3,-0)--(3.55,-0.55);

	\draw[  ](4,1)--(4,0);
	\draw[  ](4,-1)--(4,0);
	\draw[-latex,  ](4,0)--(4,0.5);
	\draw[-latex,  ](4,0)--(4,-0.5);
	
	\draw[  ](-4,1)--(-4,0);
	\draw[  ](-4,-1)--(-4,0);
	\draw[-latex,  ](-4,0)--(-4,0.5);
	\draw[-latex,  ](-4,0)--(-4,-0.5);

	\draw[  ](-3,0)--(-2,0);
	\draw[-latex,  ](-3,0)--(-2.3,0);
	\draw[  ](-6,0)--(-5,0);
	\draw[-latex,  ](-6,0)--(-5.25,0);
	\draw[  ](3,0)--(2,0);
	\draw[-latex,  ](2,0)--(2.5,0);
	\draw[  ](6,0)--(5,0);
	\draw[-latex,  ](5,0)--(5.4,0);

	\draw[dotted,rotate=60](-7,0)--(7,0);
	
	\draw [  ,rotate=60] (-2,0)--(-1.5,0.866);
	\draw [  ,rotate=60](2,0)-- (1.5,0.866);
	\draw[-latex,  ,rotate=60] (-2,0)--(-1.6,0.6928);
	\draw[-latex,  ,rotate=60] (1.5,0.866)--(1.85,0.2598);			
	\draw [  ,rotate=60] (0,0)--(-1.5,0.866);
	\draw [  ,rotate=60] (0,0)--(1.5,0.866);
	\draw[-latex,  ,rotate=60] (-1.5,0.866)--(-0.65,0.3753);
	\draw[-latex,  ,rotate=60] (0, 0)--(1, 0.5773);

	\draw [  ,rotate=60](-2,0) -- (-1.5,-0.866);
	\draw [  ,rotate=60](2,0)-- (1.5,-0.866);
	\draw[-latex,  ,rotate=60](-2,0)--(-1.6,-0.6928);
	\draw[-latex,  ,rotate=60](1.5,-0.866)--(1.85,-0.2598);	
	\draw [  ,rotate=60](0,0) -- (-1.5,-0.866);
	\draw [  ,rotate=60](0,0)-- (1.5,-0.866);
	\draw[-latex,  ,rotate=60](-1.5,-0.866)--(-0.65,-0.3753);
	\draw[-latex,  ,rotate=60](0,0)--(1,-0.5773);	
	
	\draw[  ,rotate=60](-1.5,0.866)--(-1.5,0);
	\draw[  ,rotate=60](-1.5,-0.866)--(-1.5,0);
	\draw[-latex,  ,rotate=60](-1.5,0)--(-1.5,0.5);
	\draw[-latex,  ,rotate=60](-1.5,0)--(-1.5,-0.5);

	\draw[  ,rotate=60](1.5,0.866)--(1.5,0);
	\draw[  ,rotate=60](1.5,-0.866)--(1.5,0);
	\draw[-latex,  ,rotate=60](1.5,0)--(1.5,0.5);
	\draw[-latex,  ,rotate=60](1.5,0)--(1.5,-0.5);

	\draw [  ,rotate=60](-3,0) -- (-4,1);
	\draw [  ,rotate=60](-5,0) -- (-4,1);		
	\draw[-latex,  ,rotate=60] (-5,0)--(-4.45,0.55);
	\draw[-latex,  ,rotate=60] (-4,1)--(-3.35,0.35);
	\draw [  ,rotate=60](-3,0) -- (-4,-1);
	\draw [  ,rotate=60](-5,0) -- (-4,-1);		
	\draw[-latex,  ,rotate=60] (-5,0)--(-4.45,-0.55);
	\draw[-latex,  ,rotate=60] (-4,-1)--(-3.35,-0.35);

	\draw [  ,rotate=60](-6,0) -- (-7,1);
	\draw [  ,rotate=60](6,0) -- (7,1);		
	\draw [-latex,   ,rotate=60](-7,1) -- (-6.5,0.5);
	\draw [-latex,   ,rotate=60](6,0) -- (6.5,0.5);

	\draw [  ,rotate=60](-6,0) -- (-7,-1);
	\draw [  ,rotate=60](6,0) -- (7,-1);		
	\draw [-latex,   ,rotate=60](-7,-1) -- (-6.5,-0.5);
	\draw [-latex,   ,rotate=60](6,0) -- (6.5,-0.5);
	
	\draw [  ,rotate=60](3,0) -- (4,1);
	\draw [  ,rotate=60](5,0) -- (4,1);		
	\draw[-latex,  ,rotate=60] (4,1)--(4.65,0.35);
	\draw[-latex,  ,rotate=60] (3,0)--(3.55,0.55);
	\draw [  ,rotate=60](3,0) -- (4,-1);
	\draw [  ,rotate=60](5,0) -- (4,-1);		
	\draw[-latex,  ,rotate=60] (4,-1)--(4.65,-0.35);
	\draw[-latex,  ,rotate=60] (3,-0)--(3.55,-0.55);

	\draw[  ,rotate=60](4,1)--(4,0);
	\draw[  ,rotate=60](4,-1)--(4,0);
	\draw[-latex,  ,rotate=60](4,0)--(4,0.5);
	\draw[-latex,  ,rotate=60](4,0)--(4,-0.5);
	
	\draw[  ,rotate=60](-4,1)--(-4,0);
	\draw[  ,rotate=60](-4,-1)--(-4,0);
	\draw[-latex,  ,rotate=60](-4,0)--(-4,0.5);
	\draw[-latex,  ,rotate=60](-4,0)--(-4,-0.5);

	\draw[  ,rotate=60](-3,0)--(-2,0);
	\draw[-latex,  ,rotate=60](-3,0)--(-2.3,0);
	\draw[  ,rotate=60](-6,0)--(-5,0);
	\draw[-latex,  ,rotate=60](-6,0)--(-5.25,0);
	\draw[  ,rotate=60](3,0)--(2,0);
	\draw[-latex,  ,rotate=60](2,0)--(2.5,0);
	\draw[  ,rotate=60](6,0)--(5,0);
	\draw[-latex,  ,rotate=60](5,0)--(5.4,0);

	\draw[dotted,rotate=120](-7,0)--(7,0);
	
	\draw [  ,rotate=120] (-2,0)--(-1.5,0.866);
	\draw [  ,rotate=120](2,0)-- (1.5,0.866);
	\draw[-latex,  ,rotate=120] (-2,0)--(-1.6,0.6928);
	\draw[-latex,  ,rotate=120] (1.5,0.866)--(1.85,0.2598);			
	\draw [  ,rotate=120] (0,0)--(-1.5,0.866);
	\draw [  ,rotate=120] (0,0)--(1.5,0.866);
	\draw[-latex,  ,rotate=120] (-1.5,0.866)--(-0.65,0.3753);

	\draw [  ,rotate=120](-2,0) -- (-1.5,-0.866);
	\draw [  ,rotate=120](2,0)-- (1.5,-0.866);
	\draw[-latex,  ,rotate=120](-2,0)--(-1.6,-0.6928);
	\draw[-latex,  ,rotate=120](1.5,-0.866)--(1.85,-0.2598);	
	\draw [  ,rotate=120](0,0) -- (-1.5,-0.866);
	\draw [  ,rotate=120](0,0)-- (1.5,-0.866);
	\draw[-latex,  ,rotate=120](0,0)--(1,-0.5773);	
	
	\draw[  ,rotate=120](-1.5,0.866)--(-1.5,0);
	\draw[  ,rotate=120](-1.5,-0.866)--(-1.5,0);
	\draw[-latex,  ,rotate=120](-1.5,0)--(-1.5,0.5);
	\draw[-latex,  ,rotate=120](-1.5,0)--(-1.5,-0.5);

	\draw[  ,rotate=120](1.5,0.866)--(1.5,0);
	\draw[  ,rotate=120](1.5,-0.866)--(1.5,0);
	\draw[-latex,  ,rotate=120](1.5,0)--(1.5,0.5);
	\draw[-latex,  ,rotate=120](1.5,0)--(1.5,-0.5);

	\draw [  ,rotate=120](-3,0) -- (-4,1);
	\draw [  ,rotate=120](-5,0) -- (-4,1);		
	\draw[-latex,  ,rotate=120] (-5,0)--(-4.45,0.55);
	\draw[-latex,  ,rotate=120] (-4,1)--(-3.35,0.35);
	\draw [  ,rotate=120](-3,0) -- (-4,-1);
	\draw [  ,rotate=120](-5,0) -- (-4,-1);		
	\draw[-latex,  ,rotate=120] (-5,0)--(-4.45,-0.55);
	\draw[-latex,  ,rotate=120] (-4,-1)--(-3.35,-0.35);

	\draw [  ,rotate=120](-6,0) -- (-7,1);
	\draw [  ,rotate=120](6,0) -- (7,1);		
	\draw [-latex,   ,rotate=120](-7,1) -- (-6.5,0.5);
	\draw [-latex,   ,rotate=120](6,0) -- (6.5,0.5);

	\draw [  ,rotate=120](-6,0) -- (-7,-1);
	\draw [  ,rotate=120](6,0) -- (7,-1);		
	\draw [-latex,   ,rotate=120](-7,-1) -- (-6.5,-0.5);
	\draw [-latex,   ,rotate=120](6,0) -- (6.5,-0.5);
	
	\draw [  ,rotate=120](3,0) -- (4,1);
	\draw [  ,rotate=120](5,0) -- (4,1);		
	\draw[-latex,  ,rotate=120] (4,1)--(4.65,0.35);
	\draw[-latex,  ,rotate=120] (3,0)--(3.55,0.55);
	\draw [  ,rotate=120](3,0) -- (4,-1);
	\draw [  ,rotate=120](5,0) -- (4,-1);		
	\draw[-latex,  ,rotate=120] (4,-1)--(4.65,-0.35);
	\draw[-latex,  ,rotate=120] (3,-0)--(3.55,-0.55);

	\draw[  ,rotate=120](4,1)--(4,0);
	\draw[  ,rotate=120](4,-1)--(4,0);
	\draw[-latex,  ,rotate=120](4,0)--(4,0.5);
	\draw[-latex,  ,rotate=120](4,0)--(4,-0.5);
	
	\draw[  ,rotate=120](-4,1)--(-4,0);
	\draw[  ,rotate=120](-4,-1)--(-4,0);
	\draw[-latex,  ,rotate=120](-4,0)--(-4,0.5);
	\draw[-latex,  ,rotate=120](-4,0)--(-4,-0.5);

	\draw[  ,rotate=120](-3,0)--(-2,0);
	\draw[-latex,  ,rotate=120](-3,0)--(-2.3,0);
	\draw[  ,rotate=120](-6,0)--(-5,0);
	\draw[-latex,  ,rotate=120](-6,0)--(-5.25,0);
	\draw[  ,rotate=120](3,0)--(2,0);
	\draw[-latex,  ,rotate=120](2,0)--(2.5,0);
	\draw[  ,rotate=120](6,0)--(5,0);
	\draw[-latex,  ,rotate=120](5,0)--(5.4,0);

				\filldraw[fill=white, draw=PineGreen, thick] (2.5,0) circle (0.6cm);		\filldraw[fill=white, draw=PineGreen, thick] (-2.5,0) circle (0.6cm);		\draw[-latex,PineGreen ] (2.4,0.6)--(2.6,0.6);		
	\draw[-latex,PineGreen ] (-2.6,0.6)--(-2.4,0.6);
	\filldraw[fill=white, draw=PineGreen, thick] (5.5,0) circle (0.6cm);		\filldraw[fill=white, draw=PineGreen, thick] (-5.5,0) circle (0.6cm);		\draw[-latex,PineGreen ] (5.4,0.6)--(5.6,0.6);		
	\draw[-latex,PineGreen ] (-5.6,0.6)--(-5.4,0.6);

					\filldraw[fill=white, draw=PineGreen, thick,rotate=60] (2.5,0) circle (0.6cm);	
						\filldraw[fill=white, draw=PineGreen, thick,rotate=60] (-2.5,0) circle (0.6cm);	
							\draw[-latex,PineGreen ,rotate=60] (2.4,0.6)--(2.6,0.6);		
	\draw[-latex,PineGreen ,rotate=60] (-2.6,0.6)--(-2.4,0.6);
	\filldraw[fill=white, draw=PineGreen, thick,rotate=60] (5.5,0) circle (0.6cm);		\filldraw[fill=white, draw=PineGreen, thick,rotate=60] (-5.5,0) circle (0.6cm);		\draw[-latex,PineGreen ,rotate=60] (5.4,0.6)--(5.6,0.6);		
	\draw[-latex,PineGreen ,rotate=60] (-5.6,0.6)--(-5.4,0.6);
	
				\filldraw[fill=white, draw=PineGreen, thick,rotate=120] (2.5,0) circle (0.6cm);	
					\filldraw[fill=white, draw=PineGreen, thick,rotate=120] (-2.5,0) circle (0.6cm);	
						\draw[-latex,PineGreen ,rotate=120] (2.4,0.6)--(2.6,0.6);		
	\draw[-latex,PineGreen ,rotate=120] (-2.6,0.6)--(-2.4,0.6);
	\filldraw[fill=white, draw=PineGreen, thick,rotate=120] (5.5,0) circle (0.6cm);		\filldraw[fill=white, draw=PineGreen, thick,rotate=120] (-5.5,0) circle (0.6cm);		\draw[-latex,PineGreen ,rotate=120] (5.4,0.6)--(5.6,0.6);		
	\draw[-latex,PineGreen ,rotate=120] (-5.6,0.6)--(-5.4,0.6);

	\end{tikzpicture}
	\caption{\footnotesize  The jump contour $ \Sigma^{E}$ of the RH problem \ref{p1E} for the error function  $E(k)$.}
	\label{p2fige}
\end{figure}
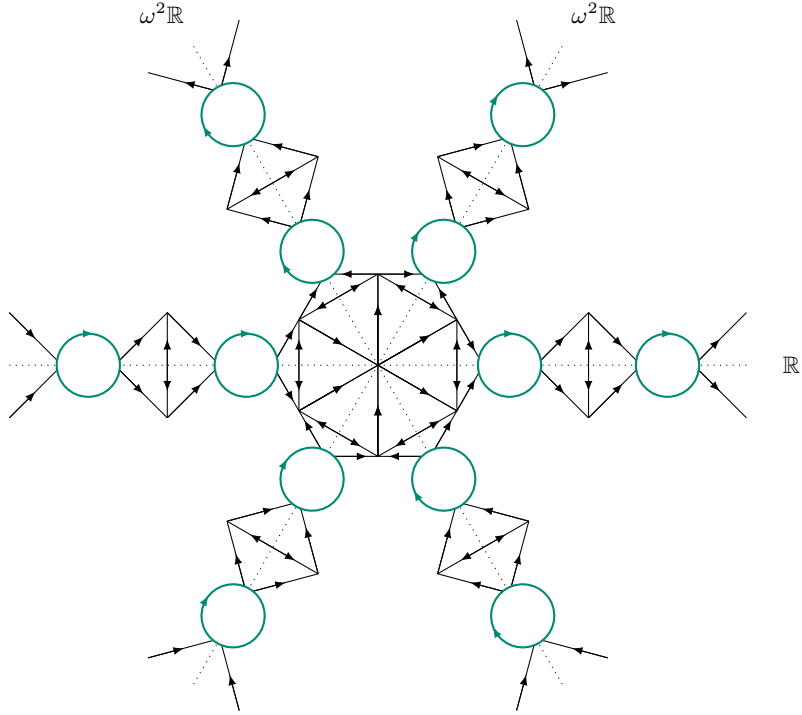

\begin{RHP}\label{p1E}
Find a  matrix-valued function $E(k):=E(k;y,t)$ such that
\begin{itemize}
\item $E(k)$ is analytic in $\mathbb{C}\setminus \Sigma^{E}$.
\item $E_+(k)=E_-(k)V^{E}(k)$ with
\begin{align}\label{pVE}
V^{E}(k)=\left\{
\begin{aligned}
   &V^{(3)}(k), \quad k\in\Sigma^{(3)}\setminus \mathrm{U},\\
   &M^{loc}(k), \quad k\in\partial\mathrm{U}.
\end{aligned}
\right.
\end{align}
\item $E(k)=I+\mathcal{O}(k^{-1}), \quad k\rightarrow \infty $.
\end{itemize}
\end{RHP}

It is readily seen that for $1\le p\le \infty$,
\begin{equation}
  \| V^{E}(k) -I \|_{L^p(\Sigma^{E})} = \begin{cases}
  \mathcal{O}(e^{-ct^{3\delta_1}}), \quad k \in \Sigma^{E} \setminus \mathrm{U},\\
  \mathcal{O}(t^{-\kappa_p}), \quad k \in \partial  \mathrm{U},
  \end{cases}
\end{equation}
where $c$ is a positive constant, and $\kappa_p = \frac{p-1}{p} \delta_1 +\frac{1}{3p}$.

It follows from the small-norm RH problems arguments  \cite{RN10} that there exists a unique solution to RH problem \ref{p1E} for large
$t$. In fact, according to Beals-Coifman's theorem \cite{BC1984}, the solution of the RH problem \ref{p1E} can be expressed as
\begin{equation}\label{p1e2}
E(k)=I+\frac{1}{2\pi \mathrm{i}}\int_{\Sigma^{E}}\dfrac{\left( I+\varpi(\varsigma)\right) (V^{E}(\varsigma)-I)}{\varsigma-k}d\varsigma,
\end{equation}
where $\varpi \in L^2(\Sigma^{E})$ is the unique solution of the following equation
\begin{equation} \label{p1en}
(1-C_{E}) \varpi = C_{E}I,
\end{equation}
with $C_{E}$ being the Cauchy projection operator on $\Sigma^{E}$. Further, we have the following  estimate
\begin{equation}
\| C_{E} \|_{L^2(\Sigma^{E})}  \le \|C_-\|_{L^2(\Sigma^{E})} \|V^{E} -I \|_{L^\infty(\Sigma^{E})} \lesssim t^{-\delta_1}. \nonumber
\end{equation}
Therefore,  the equation \eqref{p1en}  admits  a unique solution
\begin{equation}\label{p1var}
\varpi = (1-C_{E})^{-1} (C_{E} I),
\end{equation}
which  can be rewritten as
\begin{equation}
\varpi =  \sum_{j=1}^4 C_{E}^j I +(1-C_E)^{-1} (C_E^5 I),\nonumber
\end{equation}
where for $j=1,\cdots,4$, we have
\begin{align}
\| C_{E}^j I\|_{L^2(\Sigma^{E})} \lesssim t^{-\frac{1}{6}-j\delta_1+\frac{1}{2}\delta_1}, \quad
\| \varpi -  \sum_{j=1}^4 C_{E}^j I\|_{L^2(\Sigma^{E})} \lesssim t^{-\frac{1}{6}- \frac{9}{2}\delta_1 }.
\end{align}

For later use, we   evaluate the value  of  $E(k)$ at $k = e^{\frac{\pi}{6} \mathrm{i}}$.

\begin{proposition}\label{p1este1}
As $t\to\infty$,
\begin{equation}\label{p1e2e}
E(e^{\frac{\pi}{6}\mathrm{i}}) =  I + t^{-1/3} E_1(e^{\frac{\pi}{6}\mathrm{i}})
+ \mathcal{O}(t^{-2/3+4\delta_1}),
\end{equation}
in which    $E_1(e^{\frac{\pi}{6}\mathrm{i}})$ is the value of the following function at $k=e^{\frac{\pi}{6}\mathrm{i}}$,
\begin{align}
E_1(k)=    \sum\limits_{j=a,b,c,d}  \left( \frac{M_{j }^{(1)}(s)}{c_j(k-   k_j)} +     \frac{\omega \Gamma_3  \overline{M_{j }^{(1)}(s)}  \Gamma_3 }{c_j(k- \omega  k_j)} +
 \frac{   \omega ^2 \Gamma_2  \overline{M_{j }^{(1)}(s)}  \Gamma_2}{c_j(k- \omega^2 k_j)}  \right) \label{p1p1}
 \end{align}
with $c_j$ and $M_{j}^{(1)}(s)$, $j=a,b,c,d  $ being given by  \eqref{p1ca}, \eqref{p1cb} and   \eqref{p1mj01}   respectively.

\end{proposition}
\begin{proof}
 From \eqref{p1e2} and Proposition \ref{promloc}, we have
 \begin{align*}
E(e^{\frac{\pi}{6}\mathrm{i}})
&= I + \frac{1}{2 \pi \mathrm{i}} \oint_{\partial \mathrm{U}} \frac{M^{loc}(\varsigma)-I}{\varsigma - e^{\frac{\pi}{6}\mathrm{i}}} d\varsigma + \mathcal{O}(t^{-1/3-5\delta_1}),
 \end{align*}
which  together with (\ref{p1mlo})  gives \eqref{p1e2e}.
\end{proof}

\subsection{Asymptotic analysis on a pure  $\bar\partial$-problem} \label{p1pdbar}

Based on the results  of the pure RH problem \ref{p1purerhp},  we   consider  a  transformation
\begin{equation}\label{p1m2m3}
	m^{(4)}(k) = m^{(3)}(k) M^{rhp}(k)^{-1},
\end{equation}
which satisfies the following pure $\bar{\partial}$-problem.
\begin{Dbar}\label{p1dbarproblem}
 Find a row vector-valued function  $m^{(4)}(k):=m^{(4)}(k;y,t)$ such that
\begin{itemize}
\item $m^{(4)}(k)$ is continuous  in $\mathbb{C}$.
\item As $k \to \infty$ in $\mathbb{C}$, $m^{(4)}(k)=\begin{pmatrix} 1& 1& 1 \end{pmatrix}+\mathcal{O}(k^{-1})$.
\item $m^{(4)}(k)$ satisfies the $\bar\partial$-equation
\begin{equation}\label{p1Dbar1}
\bar{\partial}m^{(4)}(k)=m^{(4)}(k)W^{(4)}(k),\ \ k\in \mathbb{C}
\end{equation}
with
\begin{equation}\label{p1w4}
W^{(4)}(k)=M^{rhp}(k)\bar{\partial}\mathcal{R}^{(2)}(k)M^{rhp}(k)^{-1},
\end{equation}
where $\bar{\partial}\mathcal{R}^{(2)}(k)$ is defined in \eqref{dbarr2}.
\end{itemize}
\end{Dbar}

The $\bar \partial$-problem \ref{p1dbarproblem} is equivalent to the following integral equation
\begin{equation}
	m^{(4)}(k)  = \begin{pmatrix} 1&1&1 \end{pmatrix} +\frac{1}{\pi} \iint_{\mathbb{C}} \frac{m^{(4)}(\varsigma) W^{(4)}(\varsigma)}{\varsigma-k} dA(\varsigma),
\end{equation}
which  can be rewritten as
\begin{equation}
(I-S)m^{(4)}(k)=\begin{pmatrix} 1& 1& 1 \end{pmatrix},
\end{equation}
where $S$ is left Cauchy-Green integral  operator defined by
\begin{equation}\label{p1m3os2}
S[f](k)=\frac{1}{\pi}\underset{\mathbb{C}}\iint\dfrac{f(\varsigma)W^{(4)} (\varsigma)}{\varsigma-k}dA(\varsigma).
\end{equation}

In order to    prove  the existence of  the operator $(I-S)^{-1}$,    we need the following  estimates on   $\im \theta_{12}(k)$.
\begin{lemma}\label{p2im2}
	In $\mathcal{T}_{1}$, denote $k=|k|e^{\mathrm{i}\phi_0}$.
	Then the following estimates hold:
	\begin{itemize}
		\item Near the phase points $k_j,\ j=1,8$:
		\begin{align}
			& \im \theta_{12}(k)\leq \begin{cases}
				-c_j|\re k -k_j|^2 |\im k|,\quad k\in \Omega_{j} \cap \{ |k|\leq 2 \}, \label{omega1}\\
				-c_j|\im k|, \quad k \in \Omega_{j} \cap \{ |k|>2 \},
			\end{cases}\\
			& \im \theta_{12}(k)\geq  \begin{cases}
				c_j|\re  k -k_j|^2 |\im k|,\quad k\in  {\Omega}_{j}^* \cap \{ |k|\leq 2 \}, \\
				c_j|\im k|, \quad k\in {\Omega}_{j}^* \cap \{ |k|>2 \},
			\end{cases}
		\end{align}
		where $c_j=c_j(k_j, \phi_0,\xi)$ is a constant.
		\item Near the phase points $k_j, \ j=2,\cdots,7$:
		\begin{align}
			&	\im \theta_{12}(k)\leq -c_j|\re k -k_j|^2|\im k| ,\quad k\in \Omega_{j},	\\
			& 	\im \theta_{12}(k)\geq c_j|\re k-k_j|^2|\im k| ,\quad k\in  {\Omega}_{j}^*,
		\end{align}
		where $c_j=c_j(k_j, \phi_0,\xi)$ is a constant.
	\end{itemize}

\end{lemma}

\begin{proof}
	The proof  is similar with  Lemma 3.6 in \cite{WF}. Here we omit it.
	\end{proof}

Further with  above   Lemma \ref{p2im2},   we obtain the following estimate.
\begin{lemma}\label{lp2ests}
	The norm of the operator $S$ satisfies
	\begin{equation}\label{p2ests}
		\| S\|_{L^\infty \to L^\infty} \lesssim t^{-1/3}, \quad t \to \infty,
	\end{equation}
	which implies that $(I-S)^{-1} $ exists for large   $t$.
\end{lemma}
\begin{proof}
	We only estimate the operator $S$ defined by (\ref{p1m3os2}) on $\Omega_{1}$.
	Setting
	$$\varsigma = u+k_1 + v\mathrm{i} = |\varsigma|e^{\mathrm{i}w}, \quad k = x+ y\mathrm{i}, \quad u,v,x,y\in \mathbb{R},$$
	further using \eqref{p1est0}-\eqref{p1est2}, \eqref{p1w4} and \eqref{p1m3os2}, it is readily seen that
	\begin{equation*}
		\| S\|_{L^\infty \to L^\infty}   \le  c (I_1+I_2+I_3),
	\end{equation*}
	where
	\begin{align*}
		&I_1  =    \iint_{\Omega_{1} \cap \{|k| \le 2\} }\frac{e^{-c_1tu^2v}}{|\varsigma-k|} dA(\varsigma), \ \  I_2 =  \iint_{\Omega_{1} \cap \{|k| > 2\}}  \frac{e^{-c_1tv}}{|\varsigma-k|}dA(\varsigma),\\
		&I_3  =    \iint_{\Omega_{1} \cap \{|k| > 2\} } \frac{|u|^{-1/2} e^{-c_1tv} }{|\varsigma-k|}dA(\varsigma).
	\end{align*}

 Next, we  estimate the integrals $I_i,\ i=1,2,3$, respectively.
  Through calculations, we have the following basic inequalities
	\begin{equation*}
		\| |\varsigma-k|^{-1} \|_{L^q_u(v,\infty)} \lesssim |v-y|^{-1+1/q}, \quad \| e^{-c_1tu^2v}\|_{L^p_u(v,\infty)} \lesssim (tv)^{-1/(2p)},\ p, q>1.
	\end{equation*}
	Further, using H\"{o}lder's inequality,  we obtain
	\begin{align*}
		I_1  &= \int_{0}^{2 \sin w} \int_{v}^{2\cos w-k_1} \frac{e^{-c_1tu^2v}}{|\varsigma-k|} \mathrm{d}u \mathrm{d}v 	\lesssim t^{-\frac{1}{4}}\int_{0}^{2 \sin w}  |v-y|^{-1/2} v^{-\frac{1}{4}} e^{-c_1 tv^3} \mathrm{d}v   \lesssim t^{-1/3},\\
		I_2  &= \int_{2 \sin w}^\infty \int_{2\cos w -k_1}^{\infty} \frac{e^{-c_1tv}}{|\varsigma-k|} \mathrm{d}u \mathrm{d}v 	\lesssim \int_{2 \sin w}^\infty |v-y|^{-1/2} e^{-c_1 tv} \mathrm{d}v   \lesssim t^{-1/2},\\
 	I_3  &\lesssim  \int_{2 \sin w}^\infty v^{1/p-1/2}|v-y|^{1/q-1}e^{-c_1 tv} \mathrm{d}v \lesssim t^{-1/2},
			\end{align*}
 where $1/p+1/q=1$.
	The estimate of the operator $S$ over other sectors can be estimated in a similar way. Finally, we obtain  \eqref{p2ests}.
\end{proof}

This lemma implies that  the pure $\bar\partial$-problem \ref{p1dbarproblem} admits a unique solution for large   $t$.
For later use, we calculate the value of  $m^{(4)}(k)$ at $k =e^{\frac{\pi \mathrm{i}}{6}}$.

\begin{proposition}
	As $t\to\infty$,  we have  the following estimate
	\begin{equation}\label{p1estm3}
		| m^{(4)}(e^{\frac{\pi \mathrm{i}}{6}}) -\begin{pmatrix} 1&1&1\end{pmatrix}| \lesssim t^{-2/3}.
	\end{equation}
\end{proposition}

\begin{proof}
	In a similar way to  Lemma \ref{lp2ests}, we only  estimate the following  integral over $\Omega_{1}$.
Let $\varsigma = u+k_1 +v\mathrm{i} =|\varsigma|e^{\mathrm{i}w}$ with $u,v,w \in \mathbb{R}$, it follows that
	\begin{equation*}
		\iint_{\Omega_{1}} \frac{|W^{(4)}(\varsigma)|}{|\varsigma-e^{\frac{\pi \mathrm{i}}{6}}|} dA(\varsigma)  \lesssim I_4 + I_5 + I_6,
	\end{equation*}
	where
	\begin{align*}
		&I_4 = \iint_{\Omega_{1} \cap \{ |k|\le 2\} } \frac{e^{-c_1 t u^2 v}}{|\varsigma-e^{\frac{\pi \mathrm{i}}{6}}|} dA(\varsigma), \quad I_5 = \iint_{\Omega_{1} \cap \{ |k|> 2\} } \frac{e^{-c_1tv}}{|\varsigma-e^{\frac{\pi \mathrm{i}}{6}}|}dA(\varsigma),\\
		&I_6 = \iint_{\Omega_{1} \cap \{ |k|> 2\} }  \frac{|u|^{-1/2}e^{-c_1tv} }{|\varsigma-e^{\frac{\pi \mathrm{i}}{6}}|}dA(\varsigma).
	\end{align*}
	Noticing $|\varsigma - e^{\frac{\pi \mathrm{i}}{6}}|$ is bounded for $\varsigma \in \Omega_{1} \cap \{ |k|\le 2\}$,   direct calculation yields
	\begin{equation}\label{p1i4}
		I_4 \le \int_{0}^{2\sin w} \int_{v}^{2\cos w -k_1} e^{-c_1 t u^2 v} dudv \lesssim t^{-2/3}.
	\end{equation}
Moreover,  it is readily seen that
 $$I_5 \lesssim t^{-1}, \  \  I_6 \lesssim t^{-1},$$
which together with   \eqref{p1i4} gives the estimate \eqref{p1estm3}.

\end{proof}

\subsection{Proof of Theorem \ref{th1} for    $\mathcal{T}_{1}$}\label{proof1}

Inverting the sequence of transformations \eqref{defM1}, \eqref{defm2}, \eqref{p1m1m2}, \eqref{p1pdesM2RHP}, and  \eqref{p1m2m3}, the solution of RH problem  \ref{rhpvm} is given by
\begin{equation}
m(k)=m^{(4)}(k)E(k)\mathcal{R}^{(2)}(k)^{-1}T(k)^{-1}, \quad k \in \mathbb{C}\setminus \mathrm{U}.
\end{equation}
Taking $k =e^{\frac{\pi}{6}\mathrm{i}}$, and using \eqref{p1e2e} and \eqref{p1estm3}, we have
\begin{align}
m(e^{\frac{\pi}{6}\mathrm{i}})&= \begin{pmatrix} 1&1&1 \end{pmatrix} \left( I + t^{-1/3} E_1(e^{\frac{\pi}{6}\mathrm{i}}) \right) T(e^{\frac{\pi}{6}\mathrm{i}})^{-1}+\mathcal{O}(t^{-2/3+4\delta_1}),
\quad t\rightarrow\infty.\label{p1cm}
\end{align}
where $E_1(e^{\frac{\pi}{6}\mathrm{i}})$ is given by   \eqref{p1p1}.

Then, the solution of the DP equation \eqref{DP} can  be recovered from the reconstruction formula  \eqref{rescon} as follows
\begin{align}
u(y,t)
&= t^{-1/3}f_1(e^{\frac{\pi}{6}\mathrm{i}};y,t)+ \mathcal{O}(t^{-2/3+4\delta_1}), \label{p1u1}\\
x(y,t)
& = y+\log{T_{2}(e^{\frac{\pi}{6}\mathrm{i}})}- \log{T_{3}(e^{\frac{\pi}{6}\mathrm{i}})}+t^{-1/3} f_2(e^{\frac{\pi}{6}\mathrm{i}};y,t)  + \mathcal{O}(t^{-2/3+4\delta_1}), \nonumber
\end{align}
where
\begin{align}
    f_1(e^{\frac{\pi}{6}\mathrm{i}};y,t) = \frac{\partial }{\partial t}  f_2(e^{\frac{\pi}{6}\mathrm{i}};y,t),\quad
    f_2(e^{\frac{\pi}{6}\mathrm{i}};y,t) = \sum_{j=1}^3 E_1(e^{\frac{\pi}{6}\mathrm{i}})_{j2} -E_1(e^{\frac{\pi}{6}\mathrm{i}})_{j1},
\end{align}
and
$E_1(e^{\frac{\pi}{6}\mathrm{i}})_{ji}$ represents   the element in the
$j$-th row and
$i$-th column of the matrix  $E_1$.
Taking into account the boundedness of $\log{T_{2}(e^{\frac{\pi}{6}\mathrm{i}})}- \log{T_{3}(e^{\frac{\pi}{6}\mathrm{i}})}$, it is thereby inferred that $x/t-y/t=\mathcal{O}(t^{-1})$. Replacing $y/t$ by $x/t$ in \eqref{p1u1} yields  \eqref{ltp1}.

\section{Painlev\'e Asymptotics  in Transition Zone $\mathcal{T}_{2}$}\label{8}

\begin{figure}[htbp]
	\centering

	\subfigure[$0\le\hat{\xi}<3$]{\label{figured}
		\begin{minipage}[t]{0.32\linewidth}
			\centering
			\includegraphics[width=1.52in]{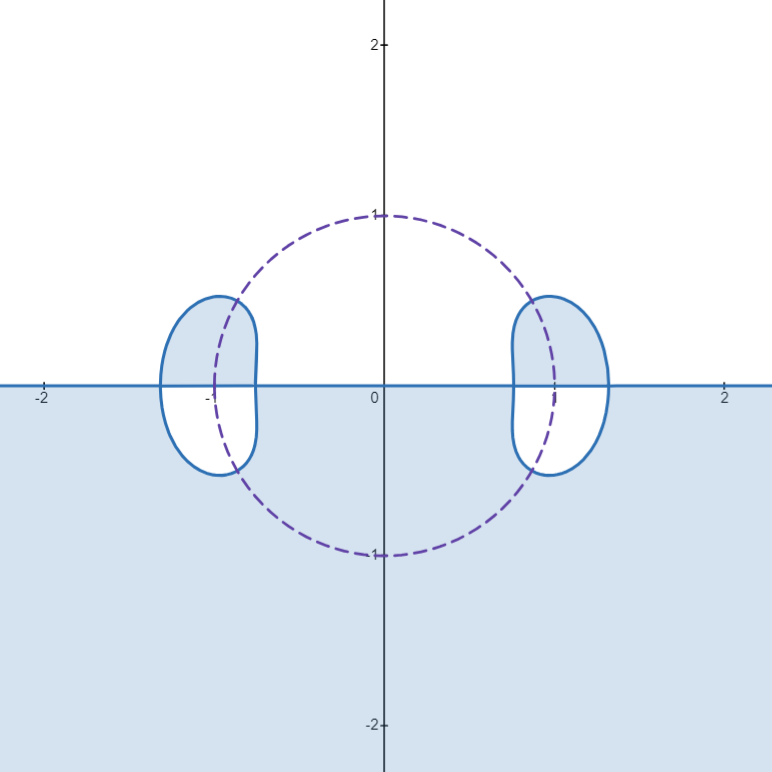}
		\end{minipage}
	}%
	\subfigure[$ \hat{\xi}=3$]{\label{figuree}
		\begin{minipage}[t]{0.32\linewidth}
			\centering
			\includegraphics[width=1.52in]{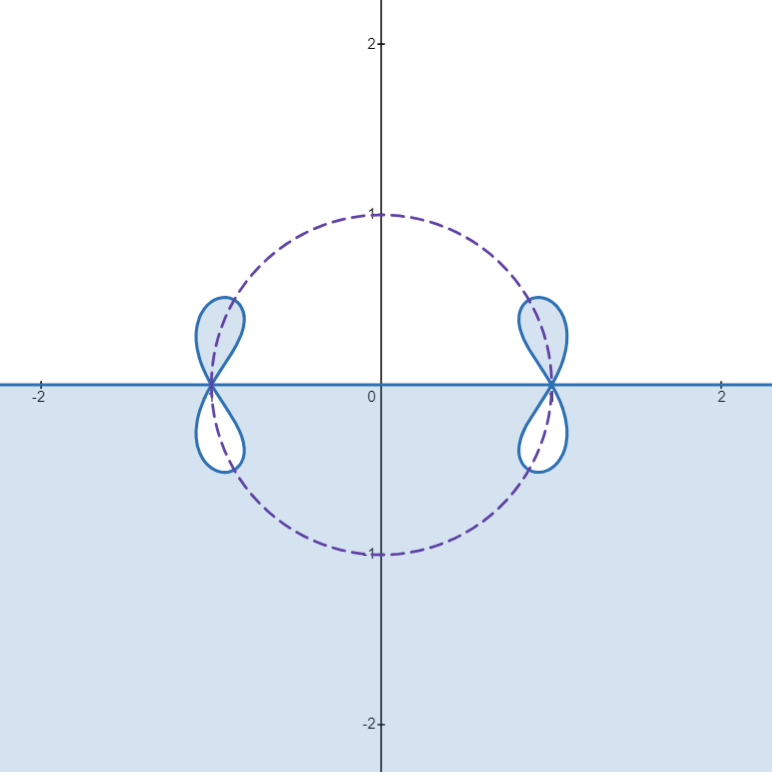}
		\end{minipage}
	}%
	\subfigure[$\hat{\xi}>3$]{\label{figuref}
		\begin{minipage}[t]{0.32\linewidth}
			\centering
			\includegraphics[width=1.52in]{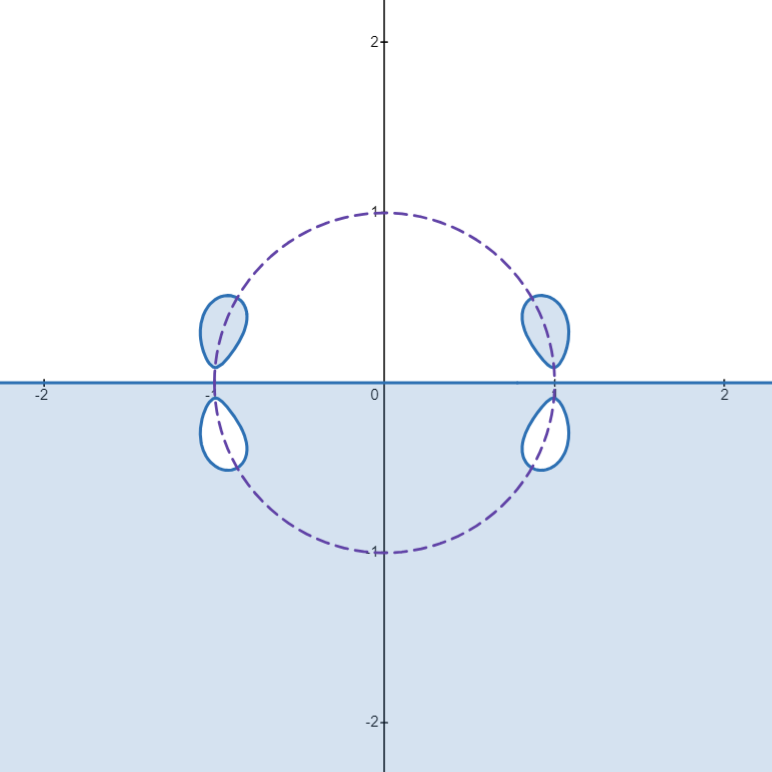}
		\end{minipage}
	}%
	\caption{\footnotesize Signature table of ${\rm Im}\theta_{12}(k)$ with different $\hat{\xi}$:
		$\textbf{(a)}$ $0\le \hat{\xi}<3$,
		$\textbf{(b)}$ $\hat{\xi}=3$,
		$\textbf{(c)}$ $\hat{\xi}>3$.
		The sectors where   ${\rm Im}\theta_{12}(k)<0$ and   ${\rm Im}\theta_{12}(k)>0$ are blue and white, respectively.
		Moreover, the  purple dotted line  stands for the unit circle.}
	\label{figthetb}
\end{figure}

In this section,  we  study the  Painlev\'e  asymptotics in  the transition  zone   $\mathcal{T}_{2}$  which  is  between  $\mathcal{Z}_2$  and
$\mathcal{S}_2$,  see Figures  \ref{figured}-\ref{figuref}.
Without loss of generality, we show that the leading asymptotics is  described by  the Painlev\'e  equation in the left-half
 zone
 $$\mathcal{T}_2^L= \mathcal{T}_2 \cap \{ (y,t): \hat{\xi}<3\}=\{ (y,t):-C<(\hat{\xi}-3)t^{2/3}<0\}.$$
  In this  case, there are 12 saddle points $\omega^l k_j, \,j=1,\cdots,4,\ l=0,1,2,$
 on three contours $\omega^l \mathbb{R}, \ l=0,1,2$,
among them  4 saddle points are  on   $\mathbb{R}$
\begin{equation}\label{sppostat}
 k_1 = - k_4 =  \frac{\sqrt{2}}{4} \left( \sqrt{s_1 +s_2 } + \sqrt{s_1+8 +s_2}\right) ,\quad
 k_2 = - k_3 =  \frac{1}{k_1},
\end{equation}
with $s_1$ and $s_2$ be given by \eqref{phas1s2}.
As $\hat{\xi} \to 3^-$, we have $s_1 \to  -3$ and $s_2 \to 3$, then
\begin{equation*}
k_1, k_2 \to  1=:k_a,\quad  k_3, k_4 \to  -1:=k_b,
\end{equation*}
from which we can have the corresponding critical points $\omega^l k_j$, $j=1,\cdots,4, \ l =1,2$.

The signature table in Figure \ref{figuree} inspires us to use the triangular factorization of the jump matrix 	$V^{(2)}(k)$  associated with RH problem $m^{(2)}(k)$ for $k \in \mathbb{R}$ in the following form
\begin{equation}\label{p2v2f}
	V^{(2)}(k) = \left(\begin{array}{ccc} 1&0&0 \\ -\bar{d}(k) e^{-\mathrm{i}t\theta_{12}(k) }&1&0 \\ 0&0&1\end{array}  \right)\left(\begin{array}{ccc} 1&d(k)e^{\mathrm{i}t\theta_{12}(k) } &0 \\ 0&1&0 \\ 0&0&1\end{array}  \right),
\end{equation}
where
\begin{equation}
	d(k):= \bar{r}(k)\frac{T_2(k)}{T_1(k)}.
\end{equation}
The factorization of $	V^{(2)}(k) $ on $\omega \mathbb{R}$ and $\omega^2 \mathbb{R}$  can be given by the symmetries.
Next, we open the $\bar\partial$ lenses using the factorization \eqref{p2v2f}  to obtain a hybrid $\bar{\partial}$-RH problem.

\subsection{Hybrid $\bar{\partial}$-RH problem }

Define
\begin{align*}
&I_1 = (k_1,\infty),\ I_2 = (k_0,k_2), \ I_3 = (k_3,k_0),\ I_4= (-\infty,k_4),\\
&I = \mathop{\cup}\limits_{i=1}^4 I_i, \ \omega I = \{\omega k:k\in I\}, \ \omega^2 I = \{\omega^2 k: k\in I \}.
\end{align*}
where we denote  $k_0:=0$.  We open the contour $I$ on $\mathbb{R}$ in a way like Subsection \ref{subsec31}, denote
 all opened sectors $\Omega_{j}$  and boundaries $\Sigma_{j}$ with  $  j=0,1,2,3,4$  depicted in Figure \ref{trans2cr}.
 Further, denote $\Sigma_{0j}, j=2,3$  as the boundaries between $\Omega_0$ and $\Omega_j, j=2,3$, respectively.
 The contours  $\omega I$ and $\omega^2 I$ can be opened  by the symmetries, and the whole contour $ \Sigma^{(3)}$  is given in  Figure \ref{pf2v3}.

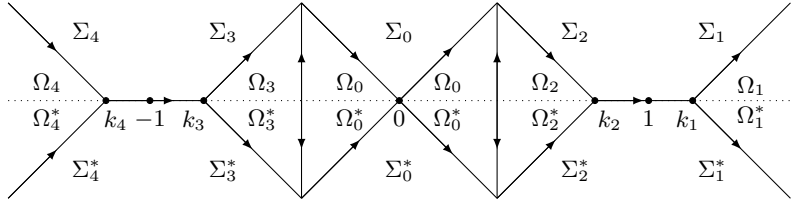
\begin{figure}[http]
	\centering
\vspace{2mm}
	\begin{tikzpicture}[scale=1.3]
	\draw[dotted](-4,0)--(4,0);
	
	\draw [  ](-2,0) -- (-1,1);
    \draw [  ]   (-1,1)--(0,0);	 	
	\draw [  ](2,0) -- (1,1);
	\draw [  ](1,1)--(0,0);
	\draw [-latex,  ](-2,0) -- (-1.5,0.5);
    \draw [-latex,  ](-1,1) -- (-0.4,0.4);	
	\draw [-latex,  ](1,1) -- (1.6,0.4);
    \draw [-latex,  ](0,0) -- (0.55,0.55);

	\draw [  ](-2,0) -- (-1,-1);
    \draw [  ](0,0) -- (-1,-1);
    \draw [  ](2,0) -- (1,-1);
    \draw [  ](0,0) -- (1,-1);
    \draw [-latex,  ](-2,0) -- (-1.5,-0.5);
    \draw [-latex,  ](-1,-1) -- (-0.4,-0.4);	
    \draw [-latex,  ](1,-1) -- (1.6,-0.4);
    \draw [-latex,  ](0,0) -- (0.55,-0.55);		

	\draw [  ](-4,1) -- (-3,0);
	\draw [-latex,  ](-4,1) -- (-3.5,0.5);
	\draw [  ](-4,-1)--(-3,0);
	\draw [-latex,  ](-4,-1) -- (-3.5,-0.5);
	\draw [  ](4,1) -- (3,0);
	\draw [-latex,  ](3,0) -- (3.5,0.5);
	\draw [  ](4,-1) --(3,0);
	\draw [-latex,  ](3,0) -- (3.5,-0.5);
	
	\draw[  ] (1,0)--(1,1);
	\draw[-latex,  ] (1,0)--(1,0.5);
    \draw[  ](-1,-0)--(-1,-1);
	\draw[-latex,  ](-1,-0)--(-1,-0.5);
		\draw[  ] (-1,0)--(-1,1);
	\draw[-latex,  ] (-1,0)--(-1,0.5);
	\draw[  ](1,-0)--(1,-1);
	\draw[-latex,  ](1,-0)--(1,-0.5);
	
 \draw[  ](-3,0)--(-2,0);
\draw[-latex,  ](-3,0)--(-2.3,0);
\draw[  ](3,0)--(2,0);
\draw[-latex,  ](2,0)--(2.5,0);

	\filldraw [ ](0,0) circle [radius=0.03];
	\node  [below]  at (0,0) {\footnotesize$0$};	
	\filldraw [ ](3,0) circle [radius=0.03];
	\filldraw [ ](2,0) circle [radius=0.03];	 	
	\filldraw [ ](-3,0) circle [radius=0.03];
	\filldraw [ ](-2,0) circle [radius=0.03];

	\node  [below]  at (2.15,0) {\footnotesize$k_{2}$};
	\node  [below]  at (2.95,0) {\footnotesize$k_{1}$};

	\filldraw [ ](2.55,0) circle [radius=0.03];
	\node  [below]  at (2.55,0) {\footnotesize$1$};		
	
	\filldraw [ ](-2.55,0) circle [radius=0.03];
	\node  [below]  at (-2.55,0) {\footnotesize$-1$};		
	
	\node  [below]  at (-2.1,0) {\footnotesize$k_{3}$};
	\node  [below]  at (-2.9,0) {\footnotesize$k_{4}$};

	\node  [above]  at (1.8,0.5) {\footnotesize$\Sigma_{2}$};
	\node  [below]  at (1.8,-0.5) {\footnotesize$\Sigma_{2}^*$};
	\node  [above]  at (3.2,0.5) {\footnotesize$\Sigma_{1}$};
	\node  [below]  at (3.2,-0.5) {\footnotesize$\Sigma_{1}^*$};

		\node  [above]  at (0,0.5) {\footnotesize$\Sigma_{0}$};
	\node  [below]  at (-0,-0.5) {\footnotesize$\Sigma_{0}^*$};
	
	\node  [above]  at (-1.8,0.5) {\footnotesize$\Sigma_{3}$};
	\node  [below]  at (-1.8,-0.5) {\footnotesize$\Sigma_{3}^*$};
	\node  [above]  at (-3.2,0.5) {\footnotesize$\Sigma_{4}$};
	\node  [below]  at (-3.2,-0.5) {\footnotesize$\Sigma_{4}^*$};
	

			
	\node  [above]  at (0.5,-0.02) {\footnotesize$\Omega_{0}$};
	\node  [below]  at (-0.5,0.02) {\footnotesize$\Omega_{0}^*$};
	\node  [above]  at (-0.5,-0.02) {\footnotesize$\Omega_{0}$};
	\node  [below]  at (0.5,0.02) {\footnotesize$\Omega_{0}^*$};
	
	\node  [above]  at (1.5,-0.02) {\footnotesize$\Omega_{2}$};
	\node  [below]  at (1.5,0.02) {\footnotesize$\Omega_{2}^*$};
	\node  [above]  at (3.6,-0.05) {\footnotesize$\Omega_{1}$};
	\node  [below]  at (3.6,0.05) {\footnotesize$\Omega_{1}^*$};

	\node  [above]  at (-1.4,-0.02) {\footnotesize$\Omega_{3}$};
	\node  [below]  at (-1.4,0.02) {\footnotesize$\Omega_{3}^*$};
	\node  [above]  at (-3.6,-0.02) {\footnotesize$\Omega_{4}$};
	\node  [below]  at (-3.6,0.02) {\footnotesize$\Omega_{4}^*$};

	\end{tikzpicture}
	\caption{\footnotesize The  contour obtained after opening a part contour  $I$  in  $\mathcal{T}_{2}$}
	\label{trans2cr}
\end{figure}

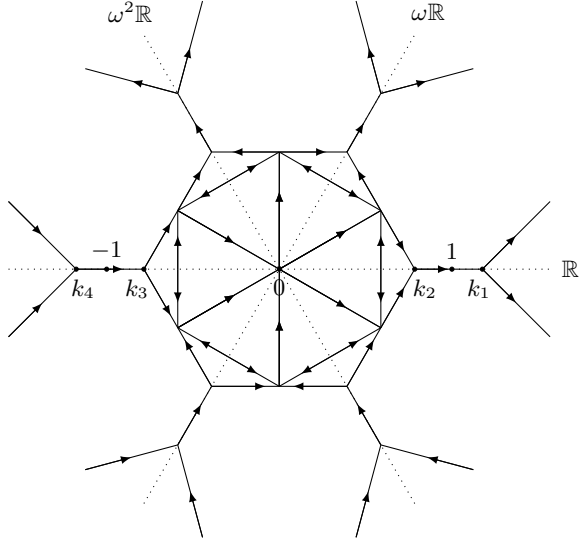
\begin{figure}[http]
	\centering
	\begin{tikzpicture}[scale=0.9]
		\draw[dotted](-4,0)--(4,0);
\node    at (4.3,0) {\footnotesize$\mathbb{R}$};
\node    at (2.2,3.8) {\footnotesize$\omega\mathbb{R}$};
\node    at (-2.2,3.8) {\footnotesize$\omega^2\mathbb{R}$};

\draw [  ] (-2,0)--(-1.5,0.866);
\draw [  ](2,0)-- (1.5,0.866);
\draw[-latex,  ] (-2,0)--(-1.6,0.6928);
\draw[-latex,  ] (1.5,0.866)--(1.85,0.2598);			
\draw [  ] (0,0)--(-1.5,0.866);
\draw [  ] (0,0)--(1.5,0.866);
\draw[-latex,  ] (-1.5,0.866)--(-0.65,0.3753);
\draw[-latex,  ] (0, 0)--(1, 0.5773);

\draw [  ](-2,0) -- (-1.5,-0.866);
\draw [  ](2,0)-- (1.5,-0.866);
\draw[-latex,  ](-2,0)--(-1.6,-0.6928);
\draw[-latex,  ](1.5,-0.866)--(1.85,-0.2598);	
\draw [  ](0,0) -- (-1.5,-0.866);
\draw [  ](0,0)-- (1.5,-0.866);
\draw[-latex,  ](-1.5,-0.866)--(-0.65,-0.3753);
\draw[-latex,  ](0,0)--(1,-0.5773);	

 \draw[  ](-1.5,0.866)--(-1.5,0);
\draw[  ](-1.5,-0.866)--(-1.5,0);
\draw[-latex,  ](-1.5,0)--(-1.5,0.5);
\draw[-latex,  ](-1.5,0)--(-1.5,-0.5);

\draw[  ](1.5,0.866)--(1.5,0);
\draw[  ](1.5,-0.866)--(1.5,0);
\draw[-latex,  ](1.5,0)--(1.5,0.5);
\draw[-latex,  ](1.5,0)--(1.5,-0.5);

\draw [  ](-4,1) -- (-3,0);
\draw [-latex,  ](-4,1) -- (-3.5,0.5);
\draw [  ](-4,-1)--(-3,0);
\draw [-latex,  ](-4,-1) -- (-3.5,-0.5);
\draw [  ](4,1) -- (3,0);
\draw [-latex,  ](3,0) -- (3.5,0.5);
\draw [  ](4,-1) --(3,0);
\draw [-latex,  ](3,0) -- (3.5,-0.5);

\draw[  ](-3,0)--(-2,0);
\draw[-latex,  ](-3,0)--(-2.3,0);
\draw[  ](3,0)--(2,0);
\draw[-latex,  ](2,0)--(2.5,0);

		\draw[dotted,rotate=60](-4,0)--(4,0);

\draw [  , rotate=60] (-2,0)--(-1.5,0.866);
\draw [  , rotate=60](2,0)-- (1.5,0.866);
\draw[-latex,  , rotate=60] (-2,0)--(-1.6,0.6928);
\draw[-latex,  , rotate=60] (1.5,0.866)--(1.85,0.2598);			
\draw [  , rotate=60] (0,0)--(-1.5,0.866);
\draw [  , rotate=60] (0,0)--(1.5,0.866);
\draw[-latex,  , rotate=60] (-1.5,0.866)--(-0.65,0.3753);
\draw[-latex,  , rotate=60] (0, 0)--(1, 0.5773);

\draw [  , rotate=60](-2,0) -- (-1.5,-0.866);
\draw [  , rotate=60](2,0)-- (1.5,-0.866);
\draw[-latex,  , rotate=60](-2,0)--(-1.6,-0.6928);
\draw[-latex,  , rotate=60](1.5,-0.866)--(1.85,-0.2598);	
\draw [  , rotate=60](0,0) -- (-1.5,-0.866);
\draw [  , rotate=60](0,0)-- (1.5,-0.866);
\draw[-latex,  , rotate=60](-1.5,-0.866)--(-0.65,-0.3753);
\draw[-latex,  , rotate=60](0,0)--(1,-0.5773);

 \draw[  , rotate=60](-1.5,0.866)--(-1.5,0);
\draw[  , rotate=60](-1.5,-0.866)--(-1.5,0);
\draw[-latex,  , rotate=60](-1.5,0)--(-1.5,0.5);
\draw[-latex,  , rotate=60](-1.5,0)--(-1.5,-0.5);

\draw[  , rotate=60](1.5,0.866)--(1.5,0);
\draw[  , rotate=60](1.5,-0.866)--(1.5,0);
\draw[-latex,  , rotate=60](1.5,0)--(1.5,0.5);
\draw[-latex,  , rotate=60](1.5,0)--(1.5,-0.5);

\draw [  , rotate=60](-4,1) -- (-3,0);
\draw [-latex,  , rotate=60](-4,1) -- (-3.5,0.5);
\draw [  , rotate=60](-4,-1)--(-3,0);
\draw [-latex,  , rotate=60](-4,-1) -- (-3.5,-0.5);
\draw [  , rotate=60](4,1) -- (3,0);
\draw [-latex,  , rotate=60](3,0) -- (3.5,0.5);
\draw [  , rotate=60](4,-1) --(3,0);
\draw [-latex,  , rotate=60](3,0) -- (3.5,-0.5);

\draw[  , rotate=60](-3,0)--(-2,0);
\draw[-latex,  , rotate=60](-3,0)--(-2.3,0);
\draw[  , rotate=60](3,0)--(2,0);
\draw[-latex,  , rotate=60](2,0)--(2.5,0);

		\draw[dotted,rotate=120](-4,0)--(4,0);

\draw [  ,rotate=120] (-2,0)--(-1.5,0.866);
\draw [  ,rotate=120](2,0)-- (1.5,0.866);
\draw[-latex,  ,rotate=120] (-2,0)--(-1.6,0.6928);
\draw[-latex,  ,rotate=120] (1.5,0.866)--(1.85,0.2598);			
\draw [  ,rotate=120] (0,0)--(-1.5,0.866);
\draw [  ,rotate=120] (0,0)--(1.5,0.866);
\draw[-latex,  ,rotate=120] (-1.5,0.866)--(-0.65,0.3753);

\draw [  ,rotate=120](-2,0) -- (-1.5,-0.866);
\draw [  ,rotate=120](2,0)-- (1.5,-0.866);
\draw[-latex,  ,rotate=120](-2,0)--(-1.6,-0.6928);
\draw[-latex,  ,rotate=120](1.5,-0.866)--(1.85,-0.2598);	
\draw [  ,rotate=120](0,0) -- (-1.5,-0.866);
\draw [  ,rotate=120](0,0)-- (1.5,-0.866);
\draw[-latex,  ,rotate=120](0,0)--(1,-0.5773);	

\draw[  ,rotate=120](-1.5,0.866)--(-1.5,0);
\draw[  ,rotate=120](-1.5,-0.866)--(-1.5,0);
\draw[-latex,  ,rotate=120](-1.5,0)--(-1.5,0.5);
\draw[-latex,  ,rotate=120](-1.5,0)--(-1.5,-0.5);

\draw[  ,rotate=120](1.5,0.866)--(1.5,0);
\draw[  ,rotate=120](1.5,-0.866)--(1.5,0);
\draw[-latex,  ,rotate=120](1.5,0)--(1.5,0.5);
\draw[-latex,  ,rotate=120](1.5,0)--(1.5,-0.5);

\draw [  ,rotate=120](-4,1) -- (-3,0);
\draw [-latex,  ,rotate=120](-4,1) -- (-3.5,0.5);
\draw [  ,rotate=120](-4,-1)--(-3,0);
\draw [-latex,  ,rotate=120](-4,-1) -- (-3.5,-0.5);
\draw [  ,rotate=120](4,1) -- (3,0);
\draw [-latex,  ,rotate=120](3,0) -- (3.5,0.5);
\draw [  ,rotate=120](4,-1) --(3,0);
\draw [-latex,  ,rotate=120](3,0) -- (3.5,-0.5);

\draw[  ,rotate=120](-3,0)--(-2,0);
\draw[-latex,  ,rotate=120](-3,0)--(-2.3,0);
\draw[  ,rotate=120](3,0)--(2,0);
\draw[-latex,  ,rotate=120](2,0)--(2.5,0);

			\filldraw [ ](0,0) circle [radius=0.03];
		\node  [below]  at (0,0) {\footnotesize$0$};	
		\filldraw [ ](3,0) circle [radius=0.03];
		\filldraw [ ](2,0) circle [radius=0.03];	 	
		\filldraw [ ](-3,0) circle [radius=0.03];
		\filldraw [ ](-2,0) circle [radius=0.03];

		\node  [below]  at (2.15,0) {\footnotesize$k_{2}$};
		\node  [below]  at (2.95,0) {\footnotesize$k_{1}$};

		\filldraw [ ](2.55,0) circle [radius=0.03];
		\node  [above]  at (2.55,0) {\footnotesize$1$};		
		
		\filldraw [ ](-2.55,0) circle [radius=0.03];
		\node  [above]  at (-2.55,0) {\footnotesize$-1$};		
		
		\node  [below]  at (-2.1,0) {\footnotesize$k_{3}$};
		\node  [below]  at (-2.9,0) {\footnotesize$k_{4}$};	
		
	\end{tikzpicture}
	\caption{\footnotesize  The whole jump contour $ \Sigma^{(3)}$  for  $ m^{(3)}(k)$,  obtained by  opening three  contours  $\omega^lI, \ l=0, 1, 2$.}
\label{pf2v3}
\end{figure}

It can be shown that there exist  matrix continuous extension functions $R_{j}(k): \ \bar{\Omega}_{j}\rightarrow\mathbb{C}, \ j=0,\cdots,4,$ continuous on $\bar{\Omega}_{j}$, with continuous first partial derivative
on $\Omega_{j}$, and boundary values
\begin{align}
	&R_{j}(k)=\Bigg\{\begin{array}{ll}
		-d(k), \ k\in I_{j},\\
	-d(k_{j}), \ k\in\Sigma_{j},
	\end{array}
\end{align}
which have the similar estimates with  \eqref{p1est0}-\eqref{p1est2}.

Using $R_{j}(k), \ j=0, \cdots,4$, we introduce
\begin{align}
	\mathcal{R}^{(2)}(k)=\left\{
	\begin{aligned}
		&\begin{pmatrix} 1 & R_{i}(k)e^{\mathrm{i}t\theta_{12}(k)} & 0 \\  0 & 1 & 0 \\ 0 & 0 & 1 \end{pmatrix}, \quad k\in\Omega_{j},\\
		&\begin{pmatrix} 1 &  0 & 0 \\ R_{i}^*(k)e^{-\mathrm{i}t\theta_{12}(k)} & 1 & 0 \\ 0 & 0 & 1 \end{pmatrix}, \quad k\in\Omega_{j}^*,\\
		&\begin{pmatrix} 1 & 0 & 0 \\ 0 & 1 & 0 \\ R_{i}(\omega^2 k)e^{-\mathrm{i}t\theta_{13}(k)} & 0 & 1 \end{pmatrix}, \quad k\in \omega\Omega_{j},\\
		&\begin{pmatrix} 1 & 0 & R_{i}^*(\omega^2 k)e^{\mathrm{i}t\theta_{13}(k)}  \\ 0 & 1 & 0 \\ 0 & 0 & 1\end{pmatrix}, \quad k\in\omega\Omega_{j}^*,\\
		&\begin{pmatrix} 1 & 0 & 0 \\ 0 & 1 & \ R_{i}(\omega k)e^{\mathrm{i}t\theta_{23}(k)} \\ 0 & 0 & 1 \end{pmatrix}, \quad k\in \omega^2\Omega_{j},\\
		&\begin{pmatrix} 1 & 0 & 0 \\ 0 & 1 & 0 \\ 0 & R_{i}^*(\omega k)e^{-\mathrm{i}t\theta_{23}(k)} & 1\end{pmatrix}, \quad k\in \omega^2\Omega_{j}^*,\\
		&I, \quad elsewhere,
	\end{aligned}
	\right.
\end{align}
and make a transformation
\begin{equation}\label{p2m1m2}
m^{(3)}(k) = m^{(2)}(k) \mathcal{R}^{(2)}(k).
\end{equation}
 Then we obtain a  hybrid $\bar\partial$-RH problem for $m^{(3)}(k)$
 which satisfies  the jump condition
 \begin{equation*}
     m^{(3)}_+(k)=m^{(3)}_-(k)V^{(3)}(k), \quad k \in \Sigma^{(3)},
 \end{equation*}
where
\begin{align*}
&\Sigma^{(3)} =   \mathop{\cup}\limits_{l=0}^2 \omega^l  \left(  (\mathbb{R}\setminus  I ) \cup (\Sigma \cup \Sigma^*) \cup (\tilde \Sigma\cup \tilde \Sigma^*) \right),
\end{align*}
with $\Sigma=\mathop{\cup}\limits_{j=0}^4\Sigma_j,
\ \ \tilde\Sigma= \Sigma_{02}\mathop{\cup}\Sigma_{03},$ and
\begin{align}
&V^{(3)}(k)=
\left\{
\begin{aligned}
    & V^{(2)}(k), \quad k \in  \mathop{\cup}\limits_{l=0}^2 \omega^l (\mathbb{R}\setminus  I),\\
    &\mathcal{R}^{(2)}(k)|_{k\in\Omega_{il}}^{-1}\mathcal{R}^{(2)}(k)|_{k\in\Omega_{jl}}, \quad k\in \mathop{\cup}\limits_{l=0}^2 \omega^l \tilde{\Sigma}, \\
    &\mathcal{R}^{(2)}(k)|_{k\in\Omega_{jl}}^{-1}\mathcal{R}^{(2)}(k)|_{k\in\Omega_{il}}, \quad k \in \mathop{\cup}\limits_{l=0}^2 \omega^l \tilde{\Sigma}^*,\\
    & \mathcal{R}^{(2)}(k')^{-1}, \quad k \in \mathop{\cup}\limits_{l=0}^2 \omega^l  {\Sigma},\\
    & \mathcal{R}^{(2)}(k'), \quad k\in \mathop{\cup}\limits_{l=0}^2 \omega^l  {\Sigma}^*.\label{p2v2}
\end{aligned}
\right.
\end{align}
Moreover, for $k \in\mathbb{C}$, we have
\begin{align*}
	\bar{\partial}m^{(3)}(k)=m^{(3)}(k)\bar{\partial}\mathcal{R}^{(2)}(k),
\end{align*}
where
\begin{align}\label{p2dbarr2}
	\bar{\partial} \mathcal{R}^{(2)}(k)=\left\{
	\begin{aligned}
		&\begin{pmatrix} 0 & 	\bar{\partial}R_{j}(k)e^{\mathrm{i}t\theta_{12}(k)} & 0 \\  0 & 0 & 0 \\ 0 & 0 & 0 \end{pmatrix}, \quad k\in\Omega_{j},\\
		&\begin{pmatrix} 0 &  0 & 0 \\ 	\bar{\partial}R_{j}^*(k)e^{-\mathrm{i}t\theta_{12}(k)} & 1 & 0 \\ 0 & 0 & 0 \end{pmatrix}, \quad k\in\Omega_{j}^*,\\
		&\begin{pmatrix} 0 & 0 & 0 \\ 0 & 0 & 0 \\ 	\bar{\partial}R_{j}(\omega^2 k)e^{-\mathrm{i}t\theta_{13}(k)} & 0 & 0 \end{pmatrix}, \quad k\in \omega\Omega_{j},\\
		&\begin{pmatrix} 0 & 0 & 	\bar{\partial}R_{j}^*(\omega^2 k)e^{\mathrm{i}t\theta_{13}(k)}  \\ 0 & 0 & 0 \\ 0 & 0 & 0\end{pmatrix}, \quad k\in\omega\Omega_{j}^*,\\
		&\begin{pmatrix} 0 & 0 & 0 \\ 0 & 0 & \ 	\bar{\partial}R_{j}(\omega k)e^{\mathrm{i}t\theta_{23}(k)} \\ 0 & 0 & 0 \end{pmatrix}, \quad k\in \omega^2\Omega_{j},\\
		&\begin{pmatrix} 0 & 0 & 0 \\ 0 & 0 & 0 \\ 0 & 	\bar{\partial}R_{j}^*(\omega k)e^{-\mathrm{i}t\theta_{23}(k)} & 0\end{pmatrix}, \quad k\in\omega^2\Omega_{j}^*,\\
		&0, \quad elsewhere.
	\end{aligned}
	\right.
\end{align}

The above hybrid $\bar\partial$-RH problem can again be decomposed into a pure RH problem  and a pure $\bar\partial$-problem. The next two subsections are then devoted to the asymptotic analysis of these two problems separately.

\subsection{Asymptotic analysis on a pure  RH problem}

 By omitting the $\bar\partial$-derivative part
of the $\bar\partial$-RH problem for $m^{(3)}(k)$,
  we obtain the following  pure RH problem.

\begin{RHP}\label{p2purerhp}
Find a  matrix-valued function  $M^{rhp}(k):= M^{rhp}(k;y,t)$ such that
\begin{itemize}
\item $M^{rhp}(k)$ is analytic in
$\mathbb{C}\setminus \Sigma^{(3)}$.
\item $M^{rhp}(k)$ has continuous boundary values $M^{rhp}_\pm(k)$ on $\Sigma^{(3)}$ and
\begin{equation}
	M^{rhp}_+(k)=M^{rhp}_-(k)V^{(3)}(k),\hspace{0.5cm}k \in \Sigma^{(3)},
\end{equation}
where $V^{(3)}(k)$ is defined by \eqref{p2v2}.
\item
	$M^{rhp}(k) =I+\mathcal{O}(k^{-1}),\hspace{0.5cm}k \rightarrow \infty.$
\end{itemize}	
\end{RHP}

Define small disks around critical points $ \omega^lk_j$
 $$\mathrm{U}_{jl} := \{k \in \mathbb{C}: |k-\omega^lk_j| \le c_0\}, \,  \, j\in \{a,b\},\, l=0, 1,2,$$
 where  the  radius $c_0$ defined by
\begin{equation}
	c_0:= \min \left\{\frac{1}{2}, 2(k_1-k_a)t^{\delta_2} \right\},
\end{equation}
where $\delta_2$ is a constant satisfying $\frac{1}{27}<\delta_2<\frac{1}{12}$.
Then, there exists a time $T$ such that the saddle points are in
$\mathrm{U} :=\mathop{\cup}\limits_{j \in \{a,b\}} (\mathrm{U}_j\cup \mathrm{U}_{j1}\cup \mathrm{U}_{j2})$ when $t>T$. Indeed, for  $ \hat \xi\in \mathcal{T}_2^L$, we have
$$|k_j -k_a| \le \sqrt{3^{-1}C} t^{-1/3}, \  j = 1,2,\quad |k_j -k_b | \le \sqrt{3^{-1}C} t^{-1/3},  \ j = 3,4,$$
which reveals that $c_0 \lesssim t^{\delta_2-1/3}$ as $t \to \infty$.

Now, we construct the solution $M^{rhp}(k)$ as follows:
\begin{align}\label{p2pdesM2RHP}
M^{rhp}(k)=\left\{
\begin{aligned}
    &E(k), \quad k\notin \mathrm{U},\\
    &E(k)M^{loc}(k), \quad k\in \mathrm{U},
\end{aligned}
\right.
\end{align}
where  $M^{loc}(k)$ is the solution of a local model, and the error function $E(k)$ is the solution of a small-norm RH problem.  Next, we construct the solution $M^{loc}(k)$.

\subsubsection{Local models} \label{p2lomod}
Similar to Subsection \ref{p1lomod}, we can construct six local models    $M_{jl}(k), j\in \{a,b\}, l =0,1,2$ with  the corresponding contours  $\Sigma_{jl} := \Sigma^{(3)} \cap \mathrm{U}_{jl}$.  Using the construction of $M_{j0}(k)$, $j=a,b$, as an example, the other cases can be given similarly.

\begin{RHP}\label{p2mlj}
Find a $3\times 3$ matrix-valued function $M_{j0}(k):=M_{j0}(k;y,t)$ such that
\begin{itemize}
\item $M_{j0}(k)$ is analytic in $\mathbb{C} \setminus \Sigma_{j0} $.
\item  For $k \in \Sigma_{j0} $, $M_{j0,+}(k)=M_{j0,-}(k)V_{j0} (k),$
where $V_{j0} (k) = V^{(3)}(k)|_{k \in \Sigma_{j0}  }$.
\item As  $k \to \infty$  in $\mathbb{C} \setminus \Sigma_j$,
$M_{j0}(k)=I+\mathcal{O}(k^{-1}).$
\end{itemize}
\end{RHP}
In order to match RH problems \ref{p2mlj}  with the model RH problem in \ref{appx},
the phase function $t\theta_{12}(k)$  is approximated with  scaled  variables as follows.
\begin{itemize}
\item For $k$ close to $k_a$,
 \begin{align}
  t \theta_{12}(k)
   = \frac{8}{3} \hat{k}^3 + 2 s \hat{k} +\mathcal{O}(\hat{k}^4 t^{-\frac{1}{3}}),
 \end{align}
where
\begin{align}\label{2phatk}
 \hat{k} = 3^{\frac{2}{3}} t^{\frac{1}{3}}(k-k_a), \quad s = 3^{-\frac{2}{3}} t^{\frac{2}{3}}(\hat{\xi}-3).
\end{align}
\item  For $k$ close to $k_b$,
 \begin{align}
  t \theta_{12}(k)
   = \frac{8}{3} \check{k}^3 + 2 s \check{k} +\mathcal{O}(\check{k}^4 t^{-\frac{1}{3}}),
 \end{align}
where $s$ is defined by \eqref{2phatk} and
\begin{align}\label{2psmb}
\check{k} = 3^{\frac{2}{3}} t^{\frac{1}{3}}(k-k_b).
\end{align}
\end{itemize}

Now we take the local model for $M_{a0}(k)$  in $\mathrm{U}_{a0}$ as an example to  match  the Painlev\'e  model, and other local models can be constructed similarly.

\subparagraph{Step I: Scaling.}
  Define the contour $\hat{\Sigma}_a$ in the $\hat{k}$-plane
\begin{equation*}
  \hat{\Sigma}_a := \mathop{\cup}\limits_{j=1}^2 (\hat{\Sigma}_{j} \cup \hat{\Sigma}_{j}^*) \cup (\hat{k}_1, \hat{k}_2),
\end{equation*}
 corresponding to the contour $\Sigma_a$ after scaling $k$ to the new variable $\hat{k}$, where
\begin{align*}
   &\hat{\Sigma}_{1} = \{\hat{k}: \hat{k}-\hat{k}_1 = l e^{\mathrm{i} (\pi-\varphi) }, 0\le l \le c_0  3^{\frac{2}{3}} t^{\frac{1}{3}}\}, \
   \hat{\Sigma}_{2} = \{\hat{k}: \hat{k}-\hat{k}_2 = l e^{\mathrm{i}\varphi }, 0\le l \le c_0 3^{\frac{2}{3}} t^{\frac{1}{3}}\},
\end{align*}
with $\hat{k}_j =3^{\frac{2}{3}} t^{\frac{1}{3}}(k_j-k_a), j=1,2$.
After scaling, we obtain the following RH problem in the $\hat{k}$-plane.

\begin{RHP}\label{p2mloc2}
Find a $3\times 3$ matrix-valued function $ {M}_{a0}(\hat{k}):=  {M}_{a0}(\hat{k}; y,t)$ such that
\begin{itemize}
\item  $ {M}_{a0}(\hat{k})$ is analytic in $\mathbb{C} \setminus \hat{\Sigma}_{a}$.
\item For $\hat{k} \in \hat{\Sigma}_{a}$, we have  $ {M}_{a0,+}(\hat{k})= {M}_{a0,-}(\hat{k})\hat{V}_{a}(\hat{k})$, where
\begin{equation}
\hat{V}_a(\hat{k}) = \begin{cases}
            \left(\begin{array}{ccc} 1&  d(k_j)e^{\mathrm{i}t\theta_{12} ( 3^{-\frac{2}{3}} t^{-\frac{1}{3}}\hat{k}+k_a) }&0 \\ 0&1&0 \\ 0&0&1\end{array}  \right), \,  \hat{k} \in \hat{\Sigma}_{j}, \, j=1,2,\\
            \left(\begin{array}{ccc} 1&0&0 \\ -d^*(k_j)  e^{-\mathrm{i}t\theta_{12} ( 3^{-\frac{2}{3}} t^{-\frac{1}{3}}\hat{k} +k_a) }&1&0 \\ 0&0&1\end{array}  \right), \,  \hat{k} \in \hat{\Sigma}_{j}^*, \, j=1,2,\\
             V^{(3)}(3^{-\frac{2}{3}}  t^{-\frac{1}{3}}\hat{k} +k_a), \,  \hat{k} \in (\hat{k}_1, \hat{k}_2).
           \end{cases}
\end{equation}
\item As $\hat{k}\rightarrow\infty$ in $\mathbb{C} \setminus \hat{\Sigma}_{a}$, $ {M}_{a0}(\hat{k})=I+\mathcal{O}(\hat{k}^{-1}).$
\end{itemize}
\end{RHP}



\subparagraph{Step \uppercase\expandafter{\romannumeral2}: Matching with the model RH problem.}

To proceed, in a similar way to  Proposition \ref{p1mathch},
it can be  shown that

\begin{proposition} \label{p2mathch}
  As $t\to\infty$,
 \begin{equation}\label{mathchp2}
   {M}_{a0}(\hat{k}) = \mathcal{A}^{-1} \Gamma_1 M^{\mathrm{L}}(\hat{k}) \Gamma_1 \mathcal{A}+ \mathcal{O}(t^{-\frac{1}{3}+4\delta_2}),
 \end{equation}
 where $M^{\mathrm{L}}(\hat{k})$ is the solution of RH problem \ref{modelp2} with  $c_1= \mathrm{i}|d(k_a)|, $
  and
 \begin{equation}
\mathcal{A} =   \left(\begin{array}{ccc}
e^{\mathrm{i}\left( \frac{\pi}{4}-\frac{\varphi_a}{2} \right) } & 0 &0 \\
0 & e^{-\mathrm{i}\left( \frac{\pi}{4}-\frac{\varphi_a}{2} \right)} & 0 \\
0 & 0 & 1
\end{array} \right), \quad \varphi_a = \arg \bar{r}( k_a) - \mathrm{i} \log T_{21}(k_a).
\end{equation}

\end{proposition}

From Proposition \ref{p2mathch}, we obtain the following result.

\begin{corollary} As $\hat{k} \to \infty$,
\begin{equation}
  {M}_{a0}(\hat{k}) =  I + \frac{  {M}_{a1}^{(1)}(s)}{\hat{k}} + \mathcal{O}(\hat{k}^{-2}),
 \end{equation}
 where
 \begin{align}
  {M}_{a1}^{(1)}(s) &= \frac{\mathrm{i}}{2}    \left(\begin{array}{ccc} \int_s^\infty v^2(\varsigma) \mathrm{d} \varsigma & -v(s)e^{\mathrm{i}\varphi_a}&0 \\ v(s)e^{-\mathrm{i}\varphi_a} &-\int_s^\infty v^2(\varsigma) \mathrm{d}\varsigma&0 \\ 0&0&0\end{array}  \right) + \mathcal{O}(t^{-\frac{1}{3}+4\delta_2}).
 \end{align}

\end{corollary}

Using a similar method, we can construct each local model $M_b(k)$ with scaled variable \eqref{2psmb} which have the following properties:
As $t\to\infty$,
\begin{equation}
 {M}_{b0} (\check{k}) = \mathcal{B}^{-1} \Gamma_1 M^{\mathrm{L}}(\check{k}) \Gamma_1 \mathcal{B}+ \mathcal{O}(t^{-\frac{1}{3}+4\delta_2}),
\end{equation}
where $M^{\mathrm{L}}(\check{k})$ is the solution of RH problem \ref{modelp2} with   $c_1=  \mathrm{i} |d(k_b)|,$
and
\begin{equation}
\mathcal{B} =   \left(\begin{array}{ccc}
e^{\mathrm{i}(\frac{\pi}{4}-\frac{\varphi_b}{2}) } & 0 &0 \\
0 & e^{-\mathrm{i}(\frac{\pi}{4}-\frac{\varphi_b}{2}) } & 0 \\
0 & 0 & 1
\end{array} \right), \quad \varphi_b = \arg \bar{r}(k_b) - \mathrm{i} \log T_{21}(k_b).
\end{equation}
Then, as $\check{k} \to \infty$,
\begin{equation}
 {M}_{b0}(\check{k}) =  I + \frac{ {M}_{b0}^{(1)}(s)}{\check{k}} + \mathcal{O}(\check{k}^{-2}),
\end{equation}
where
\begin{align}
{M}_{b0}^{(1)}(s) &= \frac{\mathrm{i}}{2}    \left(\begin{array}{ccc} \int_s^\infty v^2(\varsigma) \mathrm{d} \varsigma & -v(s)e^{\mathrm{i}\varphi_b}&0 \\ v(s)e^{-\mathrm{i}\varphi_b} &-\int_s^\infty v^2(\varsigma) \mathrm{d}\varsigma&0 \\ 0&0&0\end{array}  \right) + \mathcal{O}(t^{-\frac{1}{3}+4\delta_2}).
\end{align}

Other local models $M_{jl}(k), j\in \{a,b\},\ l = 1,2$ can be constructed by using the symmetries. Then $M^{loc}(k)$ can be constructed  as follows.

\begin{proposition} \label{p2promlo}
As $t \to \infty$,
\begin{align}\label{p2mlok}
  M^{loc}(k)&=I  + t^{-1/3}  M^{loc}_1(k,s) +\mathcal{O}(t^{-2/3+4\delta_2}),
   \end{align}
   where
 \begin{equation}
 M^{loc}_1(k,s)=	 -3^{-\frac{2}{3}}  \sum\limits_{j=a,b }  \left( \frac{M_{j }^{(1)}(s)} { k-   k_j  } +     \frac{\omega \Gamma_3  \overline{M_{j }^{(1)}(s)}  \Gamma_3 }{ k- \omega  k_j } +
 \frac{   \omega ^2 \Gamma_2  \overline{M_{j }^{(1)}(s)}  \Gamma_2}{ k- \omega^2 k_j }  \right),
\end{equation}
with
 \begin{align*}
   M_{j}^{(1)}(s) =\frac{\mathrm{i}}{2}    \left(\begin{array}{ccc} \int_s^\infty v^2(\varsigma) \mathrm{d} \varsigma & -v(s)e^{\mathrm{i}\varphi_j}&0 \\ v(s)e^{-\mathrm{i}\varphi_j} &-\int_s^\infty v^2(\varsigma) \mathrm{d}\varsigma&0 \\ 0&0&0\end{array}  \right).
 \end{align*}

\end{proposition}

\subsubsection{Small-norm RH problem} \label{p2sne}

From the decomposition  \eqref{p2pdesM2RHP},
we obtain  the following RH problem.

\begin{RHP}\label{p2E}
Find a $3\times 3$ matrix-valued function $E(k):=E(k;y,t)$ such that
\begin{itemize}
\item $E(k)$ is analytic in $\mathbb{C}\setminus \Sigma^{E}$, where $\Sigma^{E} :=\Big(\Sigma^{(3)}\setminus \mathrm{U} \Big)\cup\partial \mathrm{U}$. See Figure \ref{p2fe}.
\item For $k \in \Sigma^E$, $E_+(k)=E_-(k)V^{E}(k)$ with the jump matrix
\begin{align}\label{p2VE}
V^{E}(k)=\left\{
\begin{aligned}
   &V^{(3)}(k), \quad k\in\Sigma^{E}\setminus \mathrm{U},\\
   &M^{loc}(k), \quad k\in\partial\mathrm{U}.
\end{aligned}
\right.
\end{align}
\item As $k\rightarrow \infty$ in $\mathbb{C}\setminus \Sigma^{E}$, $E(k)=I+\mathcal{O}(k^{-1})$.

\end{itemize}
\end{RHP}

\begin{figure}[http]
	\centering
	\begin{tikzpicture}[scale=1]
\draw[dotted](-4,0)--(4,0);
\node    at (4.3,0) {\footnotesize$\mathbb{R}$};
\node    at (2.2,3.8) {\footnotesize$\omega\mathbb{R}$};
\node    at (-2.2,3.8) {\footnotesize$\omega^2\mathbb{R}$};

\draw [  ] (-2,0)--(-1.5,0.866);
\draw [  ](2,0)-- (1.5,0.866);
\draw[-latex,  ] (-2,0)--(-1.6,0.6928);
\draw[-latex,  ] (1.5,0.866)--(1.85,0.2598);			
\draw [  ] (0,0)--(-1.5,0.866);
\draw [  ] (0,0)--(1.5,0.866);
\draw[-latex,  ] (-1.5,0.866)--(-0.65,0.3753);
\draw[-latex,  ] (0, 0)--(1, 0.5773);

\draw [  ](-2,0) -- (-1.5,-0.866);
\draw [  ](2,0)-- (1.5,-0.866);
\draw[-latex,  ](-2,0)--(-1.6,-0.6928);
\draw[-latex,  ](1.5,-0.866)--(1.85,-0.2598);	
\draw [  ](0,0) -- (-1.5,-0.866);
\draw [  ](0,0)-- (1.5,-0.866);
\draw[-latex,  ](-1.5,-0.866)--(-0.65,-0.3753);
\draw[-latex,  ](0,0)--(1,-0.5773);	

\draw[  ](-1.5,0.866)--(-1.5,0);
\draw[  ](-1.5,-0.866)--(-1.5,0);
\draw[-latex,  ](-1.5,0)--(-1.5,0.5);
\draw[-latex,  ](-1.5,0)--(-1.5,-0.5);

\draw[  ](1.5,0.866)--(1.5,0);
\draw[  ](1.5,-0.866)--(1.5,0);
\draw[-latex,  ](1.5,0)--(1.5,0.5);
\draw[-latex,  ](1.5,0)--(1.5,-0.5);

\draw [  ](-4,1) -- (-3,0);
\draw [-latex,  ](-4,1) -- (-3.5,0.5);
\draw [  ](-4,-1)--(-3,0);
\draw [-latex,  ](-4,-1) -- (-3.5,-0.5);
\draw [  ](4,1) -- (3,0);
\draw [-latex,  ](3,0) -- (3.5,0.5);
\draw [  ](4,-1) --(3,0);
\draw [-latex,  ](3,0) -- (3.5,-0.5);

\draw[  ](-3,0)--(-2,0);
\draw[-latex,  ](-3,0)--(-2.3,0);
\draw[  ](3,0)--(2,0);
\draw[-latex,  ](2,0)--(2.5,0);

\draw[dotted,rotate=60](-4,0)--(4,0);

\draw [  , rotate=60] (-2,0)--(-1.5,0.866);
\draw [  , rotate=60](2,0)-- (1.5,0.866);
\draw[-latex,  , rotate=60] (-2,0)--(-1.6,0.6928);
\draw[-latex,  , rotate=60] (1.5,0.866)--(1.85,0.2598);			
\draw [  , rotate=60] (0,0)--(-1.5,0.866);
\draw [  , rotate=60] (0,0)--(1.5,0.866);
\draw[-latex,  , rotate=60] (-1.5,0.866)--(-0.65,0.3753);
\draw[-latex,  , rotate=60] (0, 0)--(1, 0.5773);

\draw [  , rotate=60](-2,0) -- (-1.5,-0.866);
\draw [  , rotate=60](2,0)-- (1.5,-0.866);
\draw[-latex,  , rotate=60](-2,0)--(-1.6,-0.6928);
\draw[-latex,  , rotate=60](1.5,-0.866)--(1.85,-0.2598);	
\draw [  , rotate=60](0,0) -- (-1.5,-0.866);
\draw [  , rotate=60](0,0)-- (1.5,-0.866);
\draw[-latex,  , rotate=60](-1.5,-0.866)--(-0.65,-0.3753);
\draw[-latex,  , rotate=60](0,0)--(1,-0.5773);

\draw[  , rotate=60](-1.5,0.866)--(-1.5,0);
\draw[  , rotate=60](-1.5,-0.866)--(-1.5,0);
\draw[-latex,  , rotate=60](-1.5,0)--(-1.5,0.5);
\draw[-latex,  , rotate=60](-1.5,0)--(-1.5,-0.5);

\draw[  , rotate=60](1.5,0.866)--(1.5,0);
\draw[  , rotate=60](1.5,-0.866)--(1.5,0);
\draw[-latex,  , rotate=60](1.5,0)--(1.5,0.5);
\draw[-latex,  , rotate=60](1.5,0)--(1.5,-0.5);

\draw [  , rotate=60](-4,1) -- (-3,0);
\draw [-latex,  , rotate=60](-4,1) -- (-3.5,0.5);
\draw [  , rotate=60](-4,-1)--(-3,0);
\draw [-latex,  , rotate=60](-4,-1) -- (-3.5,-0.5);
\draw [  , rotate=60](4,1) -- (3,0);
\draw [-latex,  , rotate=60](3,0) -- (3.5,0.5);
\draw [  , rotate=60](4,-1) --(3,0);
\draw [-latex,  , rotate=60](3,0) -- (3.5,-0.5);

\draw[  , rotate=60](-3,0)--(-2,0);
\draw[-latex,  , rotate=60](-3,0)--(-2.3,0);
\draw[  , rotate=60](3,0)--(2,0);
\draw[-latex,  , rotate=60](2,0)--(2.5,0);

\draw[dotted,rotate=120](-4,0)--(4,0);

\draw [  ,rotate=120] (-2,0)--(-1.5,0.866);
\draw [  ,rotate=120](2,0)-- (1.5,0.866);
\draw[-latex,  ,rotate=120] (-2,0)--(-1.6,0.6928);
\draw[-latex,  ,rotate=120] (1.5,0.866)--(1.85,0.2598);			
\draw [  ,rotate=120] (0,0)--(-1.5,0.866);
\draw [  ,rotate=120] (0,0)--(1.5,0.866);
\draw[-latex,  ,rotate=120] (-1.5,0.866)--(-0.65,0.3753);

\draw [  ,rotate=120](-2,0) -- (-1.5,-0.866);
\draw [  ,rotate=120](2,0)-- (1.5,-0.866);
\draw[-latex,  ,rotate=120](-2,0)--(-1.6,-0.6928);
\draw[-latex,  ,rotate=120](1.5,-0.866)--(1.85,-0.2598);	
\draw [  ,rotate=120](0,0) -- (-1.5,-0.866);
\draw [  ,rotate=120](0,0)-- (1.5,-0.866);
\draw[-latex,  ,rotate=120](0,0)--(1,-0.5773);	

\draw[  ,rotate=120](-1.5,0.866)--(-1.5,0);
\draw[  ,rotate=120](-1.5,-0.866)--(-1.5,0);
\draw[-latex,  ,rotate=120](-1.5,0)--(-1.5,0.5);
\draw[-latex,  ,rotate=120](-1.5,0)--(-1.5,-0.5);

\draw[  ,rotate=120](1.5,0.866)--(1.5,0);
\draw[  ,rotate=120](1.5,-0.866)--(1.5,0);
\draw[-latex,  ,rotate=120](1.5,0)--(1.5,0.5);
\draw[-latex,  ,rotate=120](1.5,0)--(1.5,-0.5);

\draw [  ,rotate=120](-4,1) -- (-3,0);
\draw [-latex,  ,rotate=120](-4,1) -- (-3.5,0.5);
\draw [  ,rotate=120](-4,-1)--(-3,0);
\draw [-latex,  ,rotate=120](-4,-1) -- (-3.5,-0.5);
\draw [  ,rotate=120](4,1) -- (3,0);
\draw [-latex,  ,rotate=120](3,0) -- (3.5,0.5);
\draw [  ,rotate=120](4,-1) --(3,0);
\draw [-latex,  ,rotate=120](3,0) -- (3.5,-0.5);

\draw[  ,rotate=120](-3,0)--(-2,0);
\draw[-latex,  ,rotate=120](-3,0)--(-2.3,0);
\draw[  ,rotate=120](3,0)--(2,0);
\draw[-latex,  ,rotate=120](2,0)--(2.5,0);

		\filldraw[fill=white, draw=PineGreen, thick] (2.5,0) circle (0.6cm);
		\filldraw[fill=white, draw=PineGreen, thick] (-2.5,0) circle (0.6cm);
		\draw[-latex,PineGreen ] (2.4,0.6)--(2.6,0.6);
		\draw[-latex,PineGreen ] (-2.6,0.6)--(-2.4,0.6);
		
		\filldraw[fill=white, draw=PineGreen, thick,rotate=60] (2.5,0) circle (0.6cm);
		\filldraw[fill=white, draw=PineGreen, thick,rotate=60] (-2.5,0) circle (0.6cm);
		\draw[-latex,PineGreen ,rotate=60] (2.4,0.6)--(2.6,0.6);
		\draw[-latex,PineGreen ,rotate=60] (-2.6,0.6)--(-2.4,0.6);
		
		\filldraw[fill=white, draw=PineGreen, thick,rotate=120] (2.5,0) circle (0.6cm);
		\filldraw[fill=white, draw=PineGreen, thick,rotate=120] (-2.5,0) circle (0.6cm);
		\draw[-latex,PineGreen ,rotate=120] (2.4,0.6)--(2.6,0.6);
		\draw[-latex,PineGreen ,rotate=120] (-2.6,0.6)--(-2.4,0.6);

	\end{tikzpicture}
	\caption{\footnotesize  The jump contour $ \Sigma^{E}$ of the RH problem \ref{p2E} for the error function $E(k)$.}
	\label{p2fe}
\end{figure}
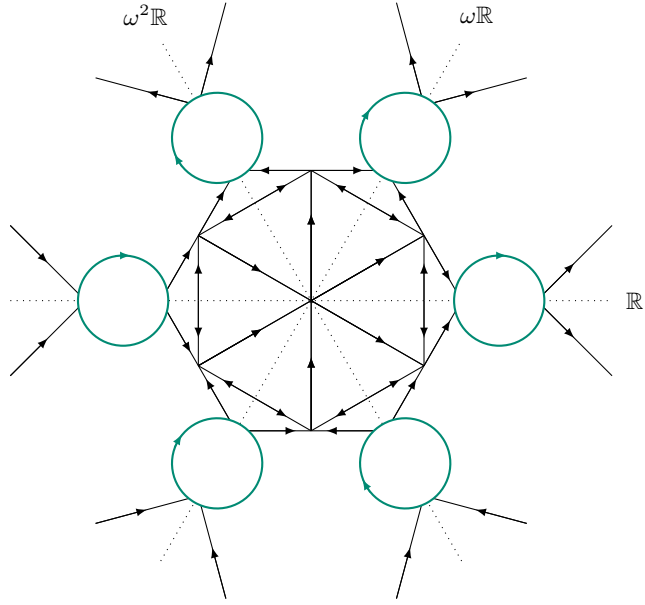
A simple calculation shows that for $1 \le p \le \infty$,
\begin{equation}
  \| V^{E} -I \|_{L^p(\Sigma^{E})} = \begin{cases}
  \mathcal{O}(e^{-ct^{3\delta_2}}), \quad k \in \Sigma^{E} \setminus \mathrm{U},\\
  \mathcal{O}(t^{-\kappa_p}), \quad k \in \partial  \mathrm{U},
  \end{cases}
\end{equation}
where $c$ is a positive constant, and $\kappa_p = \frac{p-1}{p} \delta_2 +\frac{1}{3p}$.
According to \cite{BC1984} again,  we have
\begin{equation}\label{p2e2}
E(k)=I+\frac{1}{2\pi \mathrm{i}}\int_{\Sigma^{E}}\dfrac{\left( I+\varpi(\varsigma)\right) (V^{E}(\varsigma)-I)}{\varsigma-k}d\varsigma,
\end{equation}
where $\varpi \in L^2(\Sigma^{E})$ is the unique solution of \eqref{p1en}.
Moreover,
\begin{equation}
\| C_{E} \|_{L^2(\Sigma^{E})} \lesssim t^{-\delta_2},
\end{equation}
which implies $\varpi$ exists uniquely with
\begin{align}\label{p2estce}
	\| C_{E}^j I\|_{L^2(\Sigma^{E})} \lesssim t^{-\frac{1}{6}-j\delta_2+\frac{1}{2}\delta_2}, \quad
	\| \varpi -  \sum_{j=1}^4 C_{E}^j I\|_{L^2(\Sigma^{E})} \lesssim t^{-\frac{1}{6}-\frac{9}{2}\delta_2 }.
\end{align}
Then we  can evaluate the value $E(k)$ at $k = e^{\frac{\pi}{6} \mathrm{i}}$.
\begin{proposition}\label{p2pe}
As $t\to\infty$,
\begin{equation}\label{p2ee}
E(e^{\frac{\pi}{6}\mathrm{i}}) = I + t^{-1/3}E_2(e^{\frac{\pi}{6}\mathrm{i}})  + \mathcal{O}(t^{-(2/3+4\delta_2)}),
\end{equation}
where $E_2(e^{\frac{\pi}{6}\mathrm{i}})$ is the value of the function
\begin{align}
 &E_2(k)= 3^{-\frac{2}{3}}  \sum\limits_{j=a,b }  \left( \frac{M_{j }^{(1)}(s)} { k-   k_j  } +     \frac{\omega \Gamma_3  \overline{M_{j }^{(1)}(s)}  \Gamma_3 }{ k- \omega  k_j } +
 \frac{   \omega ^2 \Gamma_2  \overline{M_{j }^{(1)}(s)}  \Gamma_2}{ k- \omega^2 k_j }  \right).\label{p2p1}
 \end{align}
\end{proposition}
\begin{proof}

Using  \eqref{p2mlok}, \eqref{p2VE}, and \eqref{p2estce},  it follows that
\begin{align*}
&E(e^{\frac{\pi}{6}\mathrm{i}}) = I + \frac{1}{2\pi \mathrm{i}} \oint_{\partial U} \frac{M^{loc}(\varsigma)-I}{\varsigma -e^{\frac{\pi}{6}\mathrm{i}} } d\varsigma + \mathcal{O}(t^{-(1/3+5\delta_2)}),
\end{align*}
which together with (\ref{p2mlok}) yields (\ref{p2ee}).

\end{proof}

\subsection{Asymptotic analysis on a  pure  $\bar\partial$-problem}

Define
\begin{equation}\label{p2m2m3mrhp}
m^{(4)}(k) : = m^{(3)}(k) M^{rhp}(k)^{-1},
\end{equation}
which satisfies the following pure $\bar{\partial}$-problem.
\begin{Dbar}\label{p2dbarproblem}
 Find a row vector-valued function  $m^{(4)}(k):=m^{(4)}(k;y,t)$ such that
\begin{itemize}
\item $m^{(4)}(k)$ is continuous  in $\mathbb{C}$.
\item  As $k \to \infty$ in $\mathbb{C}$, $m^{(4)}(k)=\begin{pmatrix} 1& 1& 1 \end{pmatrix}+\mathcal{O}(k^{-1})$.
\item $m^{(4)}(k)$ satisfies the $\bar\partial$-equation
\begin{equation}\label{p2Dbar1}
\bar{\partial}m^{(4)}(k)=m^{(4)}(k)W^{(4)}(k),\ \ k\in \mathbb{C}
\end{equation}
with
\begin{equation}\label{p2w4}
W^{(4)}(k)=M^{rhp}(k)\bar{\partial}\mathcal{R}^{(2)}(k)M^{rhp}(k)^{-1},
\end{equation}
where $\bar{\partial}\mathcal{R}^{(2)}(k)$ is defined by \eqref{p2dbarr2}.
\end{itemize}
\end{Dbar}
In a similar way to Subsection \ref{p1pdbar}, we can  obtain the following estimate.
\begin{proposition}
	There exists a large time $T>0$ such that when $t>T$,  the pure  $\bar\partial$-problem  \ref{p2dbarproblem} has a  unique   solution $m^{(4)}(k)$  with  the following estimate
	\begin{equation}\label{p2estm3}
		\left|  m^{(4)}(e^{\frac{\pi}{6}\mathrm{i}}) - \begin{pmatrix} 1& 1& 1 \end{pmatrix}        \right| \lesssim t^{-2/3}.
	\end{equation}
\end{proposition}

\subsection{Proof of Theorem \ref{th1} for $\mathcal{T}_{2}$}\label{proof2}

Inverting the sequence of transformations \eqref{defM1}, \eqref{defm2}, \eqref{p2m1m2}, \eqref{p2pdesM2RHP}, and  \eqref{p2m2m3mrhp}, the solution of RH problem  \ref{rhpvm} is given by
\begin{equation} \label{p2costm}
m(k)=m^{(4)}(k)E(k)\mathcal{R}^{(2)}(k)^{-1}T(k)^{-1}, \quad k \in \mathbb{C}\setminus \mathrm{U}.
\end{equation}
Taking $k =e^{\frac{\pi}{6}\mathrm{i}}$ in \eqref{p2costm}, and using \eqref{p2ee} and \eqref{p2estm3}, we have
\begin{align}
m(e^{\frac{\pi}{6}\mathrm{i}}) &= \begin{pmatrix} 1&1&1 \end{pmatrix} \left( I + t^{-1/3}E_2(e^{\frac{\pi}{6}\mathrm{i}})   \right)T(e^{\frac{\pi}{6}\mathrm{i}})^{-1} +\mathcal{O}(t^{-2/3+4\delta_2 }),
\ t\rightarrow\infty. \nonumber
\end{align}

In this way  the  solution (\ref{ltp2}) of  the  DP equation \eqref{DP}   can be recovered by the reconstruction formula  \eqref{rescon}.
Using  $y/t-x/t=\mathcal{O}(t^{-1})$, we obtain
the asymptotic behavior in the second transition zone.
Therefore,  the proof   of Theorem \ref{th1} for $\mathcal{T}_{2}$ is completed.

\appendix

\section{Proof of  the Non-transition  Zone $\mathcal{T}_{3}$}\label{T0}

 In this appendix, we show that there is no transition zone near  the critical line $\hat{\xi}=0$
between  $\mathcal{Z}_1$ and $\mathcal{Z}_2$  (See Figure \ref{figa}-\ref{figc}).
That is,   we    prove that  the asymptotics  in  $\mathcal{Z}_1$ match  that in $\mathcal{Z}_2$ as $t \to \infty$.

 We recall the following  asymptotics  for the solution $u(x,t)$ of the DP equation  in  two  regions  $\mathcal{Z}_1$ and $\mathcal{Z}_2$ obtained in  \cite{zx1}
\begin{align}\label{p3u}
u(x,t)=t^{-1/2} \frac{\partial}{\partial t} \left(\sum_{n=1}^3 \left( H(e^{\frac{\pi}{6}\mathrm{i}})_{n2} -H(e^{\frac{\pi}{6}\mathrm{i}})_{n1}\right) \right)+\mathcal{O}(t^{-3/4 }),
\end{align}
in which the function $H(k)$ is defined by
\begin{align}
&H(k)=-\frac{1}{2}\sum_{j=1}^{p(\hat{\xi})}F_j(k). \label{h}
\end{align}
In \eqref{h}, the  symbol  $p(\hat{\xi})$  denotes  the number of the saddle points on $\mathbb{R}$, and
$p(\hat{\xi})=8$ corresponds to  $\mathcal{Z}_1$ and   $p(\hat{\xi})=4$   to   $\mathcal{Z}_2$.  Moreover, $F_j(k)$ is  given in term of  the solution of a parabolic cylinder model
\begin{align}
	&F_j(k)=\frac{A_j(\hat{\xi})}{\sqrt{|\theta_{12}''(k_j)|}(k-k_j)}+\frac{\omega\Gamma_3\overline{A_j(\hat{\xi})}\Gamma_3}
	{\sqrt{|\theta_{12}''(\omega k_j)|}(k-\omega k_j)}
	+\frac{\omega^2\Gamma_2\overline{A_j(\hat{\xi})}\Gamma_2}{\sqrt{|\theta_{12}''(\omega^2k_j)|}(k-\omega^2k_j)},\label{fj}\\
&A_j(\hat{\xi})=\left(\begin{array}{ccc}
	0 & \tilde{\beta}^{(j)}_{12} &0\\
	\tilde{\beta}^{(j)}_{21} & 0 &0\\
	0&0&0
\end{array}\right), \ \tilde{\beta}_{12}^{(j)} =
\frac{\sqrt{2\pi}e^{\frac{\pi}{2}\nu(k_j)}e^{-\frac{\pi}{4}\mathrm{i}}}{\bar{r}_{k_j}\Gamma(\mathrm{i}\nu{k_j})}, \ \tilde{\beta}_{12}^{(j)}\tilde{\beta}_{21}^{(j)}=-\nu (k_j), \label{beta12}\\
&r_{k_j} = r(k_j)T_{12}^{(j)}(\hat{\xi})^2  e^{-2\mathrm{i} t \theta(k_j)} \zeta^{-2\mathrm{i} \eta(k_j) \nu (k_j) } e^{- \mathrm{i} \eta(k_j) \nu (k_j) \log (4 t \theta''(k_j)) \tilde{\eta}(k_j)}.
\end{align}
For more details, see Section  3.2 and Section  6.1 in \cite{zx1}.
 \begin{figure}[htbp]
	\centering
	
	\subfigure[ $-\frac{3}{8} <\hat{\xi}<0$]{\label{figa}
		\begin{minipage}[t]{0.32\linewidth}
			\centering
			\includegraphics[width=1.52in]{bf2.png}
		\end{minipage}
	}%
	\subfigure[$ \hat{\xi}=0$]{\label{figb}
		\begin{minipage}[t]{0.32\linewidth}
			\centering
			\includegraphics[width=1.52in]{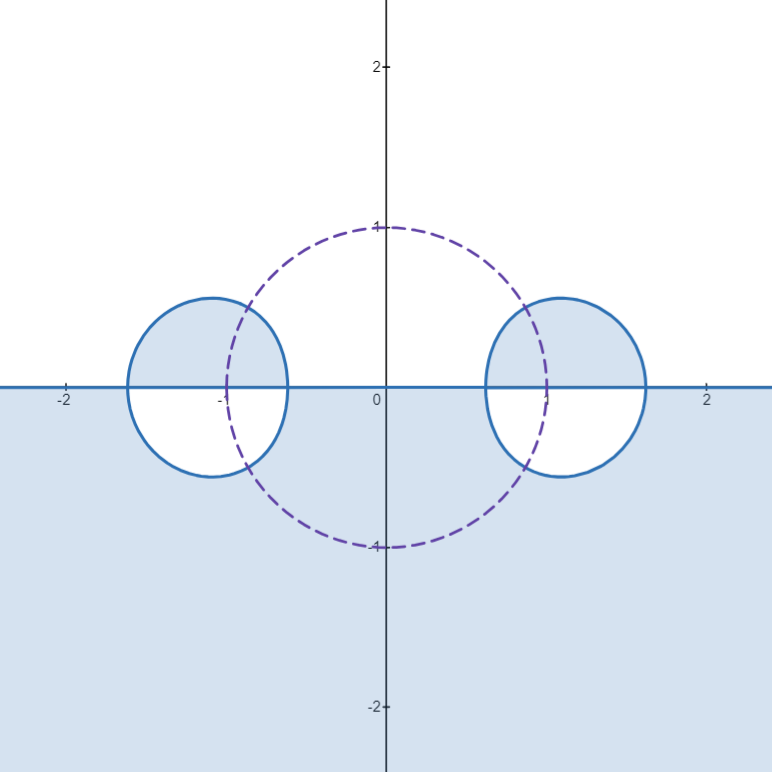}
		\end{minipage}
	}%
	\subfigure[$0\le\hat{\xi}<3$]{\label{figc}
		\begin{minipage}[t]{0.32\linewidth}
			\centering
			\includegraphics[width=1.52in]{bf4.png}
		\end{minipage}
	}%
	\caption{\footnotesize Signature table of ${\rm Im}\theta_{12}(k)$ with different $\hat{\xi}$:
		$\textbf{(a)}$ $-\frac{3}{8} <\hat{\xi}<0$,
		$\textbf{(b)}$ $\hat{\xi}=0$,
		$\textbf{(c)}$ $0\le\hat{\xi}<3$.
		The sectors where   ${\rm Im}\theta_{12}(k)<0$ and   ${\rm Im}\theta_{12}(k)>0$ are blue and white, respectively.
		Moreover, the  purple dotted line  stands for the unit circle.}
	\label{figthetc}
\end{figure}
 Further, we show that  the asymptotic result  in  $\mathcal{Z}_1$ with $p(\hat{\xi})=24$
  approaches  that in $\mathcal{Z}_2$ with $p(\hat{\xi})=12 $ as  $t \to \infty$.
 In fact, we  readily   find  that the  leading order  contribution  of   $\omega^l k_j,\ j=1,4,5,8, \ l=0,1,2,$ is  $ \mathcal{O}(t^{-1})$ which
 can be absorbed into the error term $ \mathcal{O}(t^{-3/4})$.

In $\mathcal{Z}_1$, there are $24$ saddle points $\omega^l k_j, \ j=1,\cdots,8, \ l=0,1,2$, on three contours ($\omega^l \mathbb{R}$, $l=0,1,2$); of these, $8$ saddle points
$k_j, \ j =1,\cdots,8$ are on $\mathbb{R}$.  In $ \mathcal{T}_3$, as $t \to \infty$,   we have   $ \hat{\xi}  \to 0^-$, and further
\begin{align}
&k_1 = \sqrt{3} |\hat{\xi}|^{-1/2} + \mathcal{O}(|\hat{\xi}|^{1/2})\to +\infty, \ k_4 = \frac{1}{\sqrt{3}} |\hat{\xi}|^{1/2}  + \mathcal{O}(|\hat{\xi}|^{3/2})\to 0^+, \label{p3k1}\\
&k_5 = -\sqrt{3} |\hat{\xi}|^{1/2} + \mathcal{O}(|\hat{\xi}|^{3/2})\to 0^-, \  k_8 = -\sqrt{3} |\hat{\xi}|^{-1/2} + \mathcal{O}(|\hat{\xi}|^{1/2})\to -\infty.\label{p3k8} 
\end{align}	
The corresponding limits points $\omega^l k_j, \ j=1,4,5,8,\ l=1,2$ exhibit similar properties.
\begin{align}
&k_2 = \hat k_1+\mathcal{O}(|\hat{\xi}| )\to \hat k_1, \ \
 k_3 = \hat k_2+\mathcal{O}(|\hat{\xi}| )\to \hat k_2, \\
&  k_6 = \hat k_3+\mathcal{O}(|\hat{\xi}| )\to \hat k_3, \ \
\ \  k_7 =   \hat k_4 +\mathcal{O}(|\hat{\xi}| )\to \hat k_4,
\end{align}	
where  $\hat k_1=-\hat k_4=\frac{\sqrt{5}+1}{2},  \ \hat k_2=-\hat k_3=\frac{\sqrt{5}-1}{2}$ are   four phase points as $\hat \xi=0$ in the case $\mathcal{Z}_2$.
The corresponding limits points $\omega^l  k_j, \ j=2,3,6,7,\ l=1,2$ exhibit similar properties.

Next, we demonstrate that the contributions   to the solution $u(x,t)$  from  $F_j(k)$
 near saddle points $\omega^l k_j,\ j=1,4,5,8, \ l=0,1,2,$  are  $\mathcal{O}(t^{-1})$.

Using \eqref{oritheta12} and \eqref{beta12},  we have  as  $\omega^l k_4, \omega^l k_5 \to 0, \omega^l k_1,  \omega^l k_8 \to \infty$,  for $l=0,1,2$,
	\begin{align}
	& |\theta_{12}''(\omega^lk_j)| \sim | k_j|^{-1}, \ \ \tilde{\beta}_{12}^{(j)},\  \tilde{\beta}_{21}^{(j)} \sim k_j,\ j=4,5, \label{the121}\\
	& |\theta_{12}''(\omega^lk_j)| \sim |k_j|, \ \ \tilde{\beta}_{12}^{(j)},\  \tilde{\beta}_{21}^{(j)} \sim  k_j^{-1},\ j=1,8,\label{the122}
	\end{align}
which together with (\ref{p3k1}) and (\ref{p3k8})   yields	
	\begin{align*}
	\frac{1}{\sqrt{|\theta_{12}''(\omega^lk_j)|}}  \left(\begin{array}{ccc}
	0 & \tilde{\beta}_{12}^{(j)} &0\\
	\tilde{\beta}_{21}^{(j)} & 0 &0\\
	0&0&0
	\end{array}\right) = \mathcal{O}( k_j^{3/2})  =\mathcal{O}(  t^{-1/2}), \ j=4,5,\\
	\frac{1}{\sqrt{|\theta_{12}''(\omega^lk_j)|}}  \left(\begin{array}{ccc}
	0 & \tilde{\beta}_{12}^{(j)} &0\\
	\tilde{\beta}_{21}^{(j)} & 0 &0\\
	0&0&0
	\end{array}\right) =\mathcal{O}( k_j^{-3/2} )=\mathcal{O}(  t^{-1/2}), \ j=1,8.
	\end{align*}
Meanwhile, saddle points $\omega^l k_j,\ j=2,3,6,7, \ l=0,1,2 $  approach to
saddle points $\omega^l \hat k_j,\ j=1,\cdots,4, \ l=0,1,2 $ in the case $\mathcal{Z}_2$.
Therefore, using (\ref{fj}),  the formula (\ref{h}) reduces to
	\begin{align}\label{Hp3}
H(k)= \sum_{j=1,2,3,4}\hat F_j(k) + \mathcal{O}(t^{-1/2}),
	\end{align}
where $\hat F_j(k)$ is  obtained from (\ref{fj}) by replacing $k_j$ with $\hat k_j$.

Finally, substituting  (\ref{Hp3})  into  \eqref{p3u} gives the asymptotic result in  $\mathcal{Z}_2$
\begin{align}
u(x,t)=t^{-1/2} \frac{\partial}{\partial t} \left(\sum_{n=1}^3
\left( H(e^{\frac{\pi}{6}\mathrm{i}})_{n2} -H(e^{\frac{\pi}{6}\mathrm{i}})_{n1}\right) \right)+\mathcal{O}(t^{-3/4 }), \label{p3u2}
\end{align}
with
	\begin{align}
H(k)= \sum_{j=1,2,3,4}\hat F_j(k).\label{h1}
	\end{align}
Two   results  (\ref{p3u})-(\ref{h})  and (\ref{p3u2})-(\ref{h1})   imply  that  the asymptotics in  $\mathcal{Z}_1$
goes  to that in $\mathcal{Z}_2$ near $\hat{\xi}=0$ as $t \to \infty$.
	 Therefore,  $\mathcal{T}_3$   is  not    a   transition zone.

\section{Modified Painlev\'{e} \uppercase\expandafter{\romannumeral2} RH Model} \label{appx}
The   Painlev\'{e} \uppercase\expandafter{\romannumeral2} equation takes the form
\begin{equation}\label{p23}
	 v_{ss} = 2 v^3 +s v, \quad s \in \mathbb{R},
\end{equation}
which is generally related to a $2 \times 2$ matrix-valued RH problem \cite{pa2, Charlier2020, Deiftzhoup2, FokasAblop2}.
 Here we give a modified $3 \times 3$ matrix-valued RH problem  associated with the  Painlev\'{e} \uppercase\expandafter{\romannumeral2} equation  \eqref{p23} as follows.

Denote $\Sigma^\mathrm{P} = \bigcup_{n=1}^6\left\{  \Sigma_n^\mathrm{P} = e^{\mathrm{i}\left(\frac{\pi}{6}+(n-1)\frac{\pi}{3}\right)} \mathbb{R}_+ \right\}$, see Figure \ref{Sixrays}. Let $\mathcal{C} =\{c_1,c_2,c_3\}$ be a set of complex constants such that
\begin{equation}
c_1 -c_2 +c_3 +c_1 c_2 c_3=0,
\end{equation}
and define the matrices $\{ H_n\}_{n=1}^6$ by
\begin{equation*}
H_n = 	\begin{pmatrix} 1 & 0 &0 \\ c_n e^{\mathrm{i}(\frac{8}{3}\hat{k}^3+2s\hat{k})}& 1 &0 \\ 0 & 0 & 1 \end{pmatrix}, \ n \ \text{odd}; \quad H_n = 	\begin{pmatrix} 1 & c_n e^{-\mathrm{i}(\frac{8}{3}\hat{k}^3+2s\hat{k})}&0 \\ 0 & 1 &0 \\ 0 & 0 & 1 \end{pmatrix}, \ n \ \text{even},
\end{equation*}
where $c_{n+3}=-c_n,\ n=1,2,3$.
Then there exists a countable set $\mathcal{S}_{\mathcal{C}} =\{ s_j\}_{j=1}^\infty \subset \mathbb{C}$ with $s_j \to \infty$ as $j\to\infty$, such that the following RH problem
\begin{figure}
	\begin{center}
		\begin{tikzpicture}[scale=0.9]
			\node[shape=circle,fill=black,scale=0.15] at (0,0) {0};
			\node[below] at (0.3,0.25) {\footnotesize $0$};
			\draw [] (0,-2.5 )--(0,2.5);
			\draw [-latex] (0,0)--(0,1.25);
			\draw [-latex] (0,0 )--(0,-1.25);
			\draw [] (0,0 )--(2.5,2);
			\draw [-latex] (0,0)--(1.25,1);
			\draw [] (0,0 )--(2.5,-2);
			\draw [-latex] (0,0)--(1.25,-1);
			\draw [] (0,0 )--(-2.5,2);
			\draw [-latex] (0,0)--(-1.25,1);
			\draw [] (0,0 )--(-2.5,-2);
			\draw [-latex] (0,0)--(-1.25,-1);

			\node at (0.9,1.1) {\footnotesize$\Sigma^\mathrm{P}_1$};
			\node at (1.5,-0.5) {\footnotesize$\Sigma^\mathrm{P}_6 $};
			\node at (-1.5,0.5) {\footnotesize$\Sigma^\mathrm{P}_3$};
			\node at (-0.8,-1) {\footnotesize$\Sigma^\mathrm{P}_4$};
			\node at (-0.3,1.2) {\footnotesize$\Sigma^\mathrm{P}_2$};
			\node at ( 0.4,-1.2) {\footnotesize$\Sigma^\mathrm{P}_5$};
			
			
		\end{tikzpicture}
		\caption{ \footnotesize { The jump contour $\Sigma^\mathrm{P}$.}}
		\label{Sixrays}
	\end{center}
\end{figure}

\begin{RHP}\label{1modp2}
	Find   $M^{\mathrm{P}}(\hat{k})=M^{\mathrm{P}}(\hat{k},s)$ with properties
	\begin{itemize}
		\item Analyticity: $M^{\mathrm{P}}(\hat{k})$ is analytical in $\mathbb{C}\setminus \Sigma^{\mathrm{P}}$.
		\item Jump condition:
		\begin{equation*}
			M^{\mathrm{P}}_+( \hat{k})=M^{\mathrm{P}}_-(\hat{k})H_n, \quad \hat{k} \in \Sigma_n^\mathrm{P}, \ n =1,\cdots,6.
		\end{equation*}

		\item Asymptotic behavior:
		\begin{align*}
			&M^{\mathrm{P}}(\hat{k})=I+\mathcal{O}(\hat{k}^{-1}),	\quad \hat{k} \to  \infty.
		\end{align*}
		
	\end{itemize}
\end{RHP}
\noindent has a unique solution $M^P(\hat{k})$ for each $s \in \mathbb{C} \setminus  \mathcal{S}_\mathcal{C}$. For each $n$, the restriction of $M^P(\hat{k})$ to $\arg \hat{k} \in \left(\frac{\pi(2n-3)}{6}, \frac{\pi(2n-1)}{6}\right)$ admits an analytic continuation to $\left( \mathbb{C} \setminus  \mathcal{S}_\mathcal{C} \right) \times \mathbb{C}$ and  there are smooth function $\{M_j^P(s)\}_{j=1}^\infty$ of $s \in \mathbb{C} \setminus  \mathcal{S}_\mathcal{C}$ such that, for each integer $N \ge 0$,
\begin{equation}\label{stanp}
M^P(\hat{k}) = I + \sum_{j=1}^N \frac{M_j^P(s)}{\hat{k}^j} + \mathcal{O}(\hat{k}^{-N-1}),\quad \hat{k} \to \infty,
\end{equation}
uniformly for $s$ in compact subsets of  $\mathbb{C} \setminus  \mathcal{S}_\mathcal{C}$ and for $\arg \hat{k} \in [0, 2\pi]$.
Moreover,
\begin{align}\label{up2}
  v(s)=2	\left(M_1^\mathrm{P}(s)\right)_{12} =2  \left(M_1^\mathrm{P}(s)\right)_{21}
\end{align}
solves the Painlev\'{e} \uppercase\expandafter{\romannumeral2} equation \eqref{p23}.
Further, if $\mathcal{C} = (c_1,0,-c_1)$ where $c_1 \in \mathrm{i} \mathbb{R}$ with $|c_1| <1$, then the leading coefficient $M_1^\mathrm{P}(s)$ is given by
\begin{align}\label{posee}
	M_1^\mathrm{P}(s) = \frac{1}{2} \begin{pmatrix} -i\int_{s}^\infty v(\varsigma)^2\mathrm{d}\varsigma & v(s) &0 \\ v(s) & i\int_{s}^\infty v(\varsigma)^2\mathrm{d}\varsigma &0  \\ 0& 0 & 0 \end{pmatrix},
\end{align}
and for each $c_1 > 0$,
\begin{align}\label{mPbounded}
	\sup_{\hat{k} \in \mathbb{C}\setminus \Sigma^\mathrm{P}} \sup_{s \geq -c_1} |M^\mathrm{P}(\hat{k})|  < \infty.
\end{align}
The solution of the Painlev\'{e} \uppercase\expandafter{\romannumeral2} equation \eqref{p23} is specified by
\begin{equation}
	v(s) \sim -\im c_1 \mathrm{Ai}(s) \sim - \frac{\im c_1}{2\sqrt{\pi}} s^{-\frac{1}{4}}  e^{-\frac{2}{3}s^{\frac{3}{2}}},\ s \to +\infty,
\end{equation}
where $\mathrm{Ai}(s)$ denotes the classical Airy function.

Let $\Sigma^{\mathrm{L}} = \Sigma^{\mathrm{L}}(\hat{k}_0)$ denote the contour $\Sigma^{\mathrm{L}} = \cup_{j=1}^5 \Sigma^{\mathrm{L}}_j $, as depicted in Figure \ref{fmatp2}, where
\begin{align*}
&\Sigma^{\mathrm{L}}_1 = \{\hat{k}| \hat{k}=\hat{k}_0 + r e^{\frac{\pi \mathrm{i}}{6}},\ 0\le r <\infty \}, \quad \Sigma^{\mathrm{L}}_2  = \{\hat{k}|\hat{k}=-\hat{k}_0 +r e^{\frac{5\pi \mathrm{i}}{6}},\ 0\le r <\infty  \},\\
&\Sigma^{\mathrm{L}}_3  = \{\hat{k}|\hat{k} \in \overline{\Sigma^{\mathrm{L}}_2}\}, \quad \Sigma^{\mathrm{L}}_4  =  \{\hat{k}|\hat{k} \in \overline{\Sigma^{\mathrm{L}}_1}\}, \quad
\Sigma^{\mathrm{L}}_5  = \{\hat{k}|-\hat{k}_0 \le \hat{k} \le \hat{k}_0 \}.
\end{align*}
Then, let $v(s;c_1,0,-c_1)$ denote the smooth real-valued solution of \eqref{p23} corresponding to $(c_1,0,-c_1)$ and $M^{\mathrm{P}}(\hat{k}) = M^{\mathrm{P}}(\hat{k},s;c_1,0,-c_1)$ be the corresponding solution of RH problem \ref{1modp2}.

Denote the open subsets $\{V_j\}_{j=1}^4$, as shown in Figure \ref{fmatp2} and define
\begin{equation*}
M^{\mathrm{L}}(\hat{k}) =  M^{\mathrm{P}}(\hat{k}) \times \begin{cases}
\begin{pmatrix} 1 & 0 &0 \\ c_1e^{\mathrm{i} (\frac{8\hat{k}^3}{3} + 2s \hat{k})} & 1 &0 \\ 0 & 0 & 1 \end{pmatrix}, \quad \hat{k} \in V_1 \cup V_2,\\
\begin{pmatrix} 1 & \bar{c}_1 e^{-\mathrm{i} (\frac{8\hat{k}^3}{3} + 2s \hat{k})} &0 \\ 0 & 1 &0 \\ 0 & 0 & 1 \end{pmatrix},\quad \hat{k} \in V_3 \cup V_4,\\
\end{cases}
\end{equation*}
which satisfies the following model RH problem.

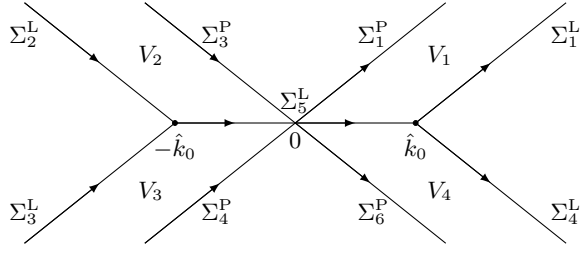
\begin{figure}[http]
	\centering
	\begin{tikzpicture}[scale=0.8]
	\filldraw (-2,0) circle [radius=0.04];
	\filldraw (2,0) circle [radius=0.04];
	\node  [below]  at (0,0) {\footnotesize$0$};
	\node  [below]  at (-2,0) {\footnotesize$-\hat{k}_0$};
	\node  [below]  at (2,0) {\footnotesize$\hat{k}_0$};
	
	\draw[] (-2,0)--(2,0);
	\draw[-latex] (-2,0)--(-1,0);
	\draw[-latex] (0,0)--(1,0);
	
	\draw[] (-4.5,2)--(-2,0);
	\draw[-latex] (-4.5,2)--(-3.25,1);
	\draw[] (0,0)--(-2.5,2);
	\draw[-latex](-2.5,2)--(-5/4,1);
	\draw[] (-4.5,-2)--(-2,0);
	\draw[-latex] (-4.5,-2)--(-3.25,-1);
	\draw[] (0,0)--(-2.5,-2);
	\draw[-latex] (-2.5,-2)--(-5/4,-1);
	
	\draw[] (4.5,2)--(2,0);
	\draw[-latex] (2,0)--(3.25,1);
	\draw[] (0,0)--(2.5,2);
	\draw[-latex] (0,0)--(5/4,1);
	\draw[] (4.5,-2)--(2,0);
	\draw[-latex] (2,0)--(3.25,-1);
	\draw[] (0,0)--(2.5,-2);
	\draw[-latex] (0,0)--(5/4,-1);
	
	\node  [above]  at (-2.4,0.8) {\footnotesize$V_2$};
	\node  [below]  at (-2.4,-0.8) {\footnotesize$V_3$};

	\node  [above]  at (2.4,0.8){\footnotesize$V_1$};
	\node  [below]  at (2.4,-0.8){\footnotesize$V_4$};
	
	    \node  [above]  at (-4.5,1.1) {\footnotesize$\Sigma^{\mathrm{L}}_2$};
	\node  [below]  at (-4.5,-1.1) {\footnotesize$\Sigma^{\mathrm{L}}_3$};

	\node  [above]  at (4.5,1.1){\footnotesize$ \Sigma^{\mathrm{L}}_1$};
	\node  [below]  at (4.5,-1.1){\footnotesize$\Sigma^{\mathrm{L}}_4$};
	  \node  [above]  at (0,0){\footnotesize$\Sigma^{\mathrm{L}}_5$};

	  	    \node  [above]  at (-1.3,1.1) {\footnotesize$\Sigma^{\mathrm{P}}_3$};
	  \node  [below]  at (-1.3,-1.1) {\footnotesize$\Sigma^{\mathrm{P}}_4$};

	  \node  [above]  at (1.3,1.1){\footnotesize$ \Sigma^{\mathrm{P}}_1$};
	  \node  [below]  at (1.3,-1.1){\footnotesize$\Sigma^{\mathrm{P}}_6$};

	\end{tikzpicture}
	\caption{\footnotesize The open subsets $\{V_j\}_{j=1}^4$.}
	\label{fmatp2}
\end{figure}

\begin{RHP}\label{modelp2}
	Find   $M^{\mathrm{L}}(\hat{k})=M^{\mathrm{L}}(\hat{k},s)$ with properties
	\begin{itemize}
		\item Analyticity: $M^{\mathrm{L}}(\hat{k})$ is analytical in $\mathbb{C}\setminus \Sigma^{\mathrm{L}}$.
		\item Jump condition:
		\begin{equation*}
			M^{\mathrm{L}}_+( \hat{k})=M^{\mathrm{L}}_-(\hat{k})V^{\mathrm{L}}(\hat{k}), \quad \hat{k} \in \Sigma^{\mathrm{L}},
		\end{equation*}
where
\begin{equation}\label{jumpl}
V^{\mathrm{L}}(\hat{k}) = \begin{cases}
\begin{pmatrix} 1 & 0 &0 \\ c_1 e^{\mathrm{i} (\frac{8\hat{k}^3}{3} +2 s \hat{k})} & 1 &0 \\ 0 & 0 & 1 \end{pmatrix}, \quad \hat{k} \in \Sigma_1^\mathrm{L} \cup \Sigma_2^\mathrm{L},\\
\begin{pmatrix} 1 & -\bar{c}_1 e^{-\mathrm{i} (\frac{8\hat{k}^3}{3} + 2s \hat{k})} &0 \\ 0 & 1 &0 \\ 0 & 0 & 1 \end{pmatrix},\quad \hat{k} \in \Sigma_3^\mathrm{L}\cup \Sigma_4^\mathrm{L},\\
\begin{pmatrix} 1 -|c_1|^2 & - \bar{c}_1 e^{-\mathrm{i} (\frac{8\hat{k}^3}{3} + 2s \hat{k})} &0 \\ c_1 e^{\mathrm{i} (\frac{8\hat{k}^3}{3} + 2s \hat{k})} & 1 &0 \\ 0 & 0 & 1 \end{pmatrix},\quad \hat{k} \in \Sigma_5^\mathrm{L},
\end{cases}
\end{equation}
with coefficients $c_1 \in \{\mathrm{i} r |-1<r<1 \}$.
		\item Asymptotic behavior: $M^{\mathrm{L}}( \hat{k})=I+\mathcal{O}(\hat{k} ^{-1}),	\quad \hat{k} \to  \infty.$

	\end{itemize}
\end{RHP}
\noindent
This model  RH problem for $M^{\mathrm{L}}(\hat{k})$ is    used to match the local RH problem for $M^{loc}(k)$   in transition zones $\mathcal{T}_{1}$ and $\mathcal{T}_{2}$.
\vspace{4mm}

    \noindent\textbf{Acknowledgements}

    This work is supported by  the National Natural Science
    Foundation of China (Grant No. 12271104, 12347141) and the
    Postdoctoral Fellowship Program of CPSF (Grant No. GZB20230167, 2023M740717).\vspace{2mm}

    \noindent\textbf{Data Availability Statements}

    The data which supports the findings of this study is available within the article.\vspace{2mm}

    \noindent{\bf Conflict of Interest}

    The authors have no conflicts to disclose.


\begin{thebibliography}{99}


\bibitem{AM} A. Degasperis, M. Procesi,
\newblock{Asymptotic integrability},
\textit{Symmetry and Perturbation Theory,  Ed. A. Degasperis and G. Gaeta, World Scientific, Singapore},  1999, 23-37.

\bibitem{ADD} A. Degasperis, D. D. Holm, A. N. W. Hone,
\newblock{A new integral equation with peakon solutions},
\textit{ Theor. Math. Phys.}, 133 (2002), 1463-1474.


\bibitem{RS}  R. S. Johnson,
\newblock{Camassa-Holm, Korteweg-de Vries and related models for water waves},
\textit{ J. Fluid Mech.}, 455 (2002), 63-82


\bibitem{RI1} R. I. Ivanov,
\newblock{Water waves and integrability},
\textit{ Phil. Trans. R. Soc. Lond. A}, 365 (2007), 2267-2280.


\bibitem{AD} A. Constantin, D. Lannes,
\newblock{The hydrodynamical relevance of the Camassa-Holm and Degasperis-Procesi equations},
\textit{ Arch. Ration. Mech. Anal.}, 192 (2009), 165-186.

\bibitem{BD} B. Alvarez-Samaniego, D. Lannes,
\newblock{Large time existence for 3D water-waves and asymptotics},
\textit{ Invent. Math.}, 171 (2008), 485-541.





\bibitem{H} H. Lundmark,
\newblock{Formation and dynamics of shock waves in the Degasperis-Procesi equation},
\textit{J. Nonl. Sci.}, 17 (3) (2007), 169-198.


\bibitem{YZ} Y. Liu, Z. Y. Yin,
\newblock{Global existence and blow-up phenomena for the Degasperis-Procesi equation},
\textit{Commun. Math. Phys.}, 267 (2006), 801-820.


\bibitem{Liu}  Z. W. Lin,  Y. Liu,
\newblock{Stability of peakons for the Degasperis-Procesi equation},
\textit{Commun. Pure Appl. Math.},  53 (2009), 0125–0146.



\bibitem{LLW}  J. Li, Y. Liu, Q. L. Wu,
\newblock{Orbital stability of the sum of smooth solitons
	in the Degasperis-Procesi equation},
\textit{J. Math. Pure   Appl.},  163 (2022),  204-230.



\bibitem{Mu} Y.  Matsuno,
\newblock{
	Multisoliton solutions of the Degasperis-Procesi
	equation and their peakon limit},
\textit{Inverse  Problems}, 21 (2005),  1553–1570.



\bibitem{HZF}
Y. Hou,  P. Zhao, E. G. Fan, Z. J. Qiao,
\newblock{ Algebro-geometric solutions for Degasperis-Procesi
	hierarchy,}
\textit{ SIAM J. Math. Anal.}, 45 (2013), 1216-1266.

\bibitem{FGP}
R. Feola, F. Giuliani, M. Procesi,
\newblock{ Reducible KAM tori for the Degasperis-Procesi equation,}
\textit{ Commun. Math. Phys.},  377 (2020), 1681–1759.


\bibitem{RI2} R. I. Ivanov,
\newblock{On the integrability of a class of nonlinear dispersive wave equations},
\textit{ J. Nonl. Math. Phys.}, 12 (2005), 462-468.





\bibitem{RD} R. Camassa, D. Holm,
\newblock{An integrable shallow water equation with peaked solitons},
\textit{ Phys. Rev. Lett.}, 71 (1993), 1661-1664.


\bibitem{3Nov1}
A. Boutet de Monvel, D.  Shepelsky, L. Zielinski,
A Riemann-Hilbert approach for the Novikov equation,
\textit{ SIGMA}, 12 (095) (2016), 22 pp.

\bibitem{3Nov2}
Y. L. Yang, E. G. Fan,
Soliton resolution and large time behavior of
solutions to the Cauchy problem for the Novikov
equation with a nonzero background,
\textit{  Adv. Math.}, 426 (2023), 109088.

\bibitem{3Bou1}
C. Charlier, J. Lenells,
On Boussinesq’s equation for water waves,
 arXiv:2204.02365.

 \bibitem{3Bou2}
 C. Charlier, J. Lenells,
 Boussinesq's equation for water waves: Asymptotics in Sector I,
\textit{ Adv. Nonl.  Anal.}, 13 (2024), 20240022.


 \bibitem{3Bou3}
  C. Charlier, J. Lenells,
 Boussinesq's equation for water waves: asymptotics in Sector V,
 \textit{ SIAM J. Math. Anal.}, 56 (2024), 4104-4142.

 \bibitem{3Bou4}
  C. Charlier, J. Lenells,
  Direct and inverse scattering for the Boussinesq equation with solitons,
  arXiv:2302.14593.

  \bibitem{3Bou5}
    C. Charlier, J. Lenells,
    The soliton resolution conjecture for the Boussinesq equation,
    arXiv:2303.10485.

    \bibitem{3SS1}
    H. Liu, X. Geng, B. Xue,
    The Deift–Zhou steepest descent method to long-time asymptotics for the Sasa–Satsuma
    equation,
    \textit{ J. Differertial Equations}, 265 (2018), 5984–6008.

    \bibitem{3SS2}
 N. Liu, B. Guo,
 Long-time asymptotics for the Sasa-Satsuma equation via nonlinear steepest descent method,
 \textit{  J. Math. Phys.}, 60 (2019), 011504.

\bibitem{pa6}
L. Huang, J. Lenells,
Asymptotics for the Sasa-Satsuma equation in terms of a modified Painlev\'{e} II transcendent,
\textit{J. Differential Equations}, 268 (2020), 7480-7504.





\bibitem{CH1}
A. Constantin, H. McKean,
\newblock{ A shallow water equation on the circle},
\textit{   Commun. Pure Appl. Math.}, 52 (1999), 949–982.

\bibitem{CH2}
R. Beals, D. H. Sattinger, J. Szmigielski,
\newblock{ Acoustic scattering and the extended Korteweg–de Vries hierarchy},
\textit{ Adv. Math.}, 40 (1998), 190–206.

\bibitem{CH3}
A. Boutet de Monvel, D. Shepelsky,
Riemann-Hilbert problem in the inverse scattering for the Camassa-Holm equation on the line, in: \newblock{Probability,
Geometry and Integrable Systems}, in: \newblock{Math. Sci. Res. Inst. Publ.}, vol. 55, Cambridge University Press, Cambridge, 2008, pp. 53–75.

\bibitem{CH4}
A. Constantin,
On the scattering problem for the Camassa-Holm equation,
\textit{Proc. R. Soc. Lond. A}, 457 (2001), 953–970.

\bibitem{CH5}
A. Constantin, J. Lenells,
\newblock{On the inverse scattering approach to the Camassa-Holm equation},
\textit{J. Nonl.  Math. Phys.}, 10 (2003), 252–255.

\bibitem{CH6} A. Boutet de Monvel, D. Shepelsky,
The Camassa–Holm equation on the half-line: a Riemann-Hilbert approach,
 \textit{J. Geom. Anal.}, 18 (2008), 285–323.




\bibitem{isth}
H. Lundmark, J. Szmigielski,
Multi-peakon solutions of the Degasperis-Procesi equation,
\textit{Inverse Problems}, 19 (2003), 1241–1245.


\bibitem{ARIJ} A. Constantin, R. I. Ivanov, J. Lenells,
\newblock{Inverse scattering transform for the Degasperis-Procesi equation},
\textit{Nonlinearity}, 23 (2010), 2559-2575.



\bibitem{Constin} A. Constantin, R. I. Ivanov,
\newblock{Dressing method for the Degasperis-Procesi equation},
\textit{Stud. Appl. Math.},   138 (2016),  205-226.





\bibitem{AD2} A. Boutet de Monvel,  D.  Shepelsky,
\newblock{A Riemann-Hilbert approach for the Degasperis-Procesi equation},
\textit{Nonlinearity}, 26 (2013), 2081-2107.



\bibitem{LDP}
J. Lenells,
\newblock{The Degasperis-Procesi equation on the half-line},
\textit{Nonl.  Anal.}, 76(2013), 122-139.

\bibitem{MLS} A. Boutet de Monvel, J. Lenells,  D.  Shepelsky,
\newblock{Long-time asymptotics for the Degasperis-Procesi equation on the half-line},
\textit{Ann. Inst. Fourier, Grenoble}, 69 (2019),  171-230.



\bibitem{zx1}
X. Zhou, Z. Y. Wang, E. G. Fan,
 Soliton resolution and asymptotic stability of $N$-solitons to the   Degasperis-Procesi equation on the line,
\textit{arXiv:2212.01765}.



\bibitem{pa1}
H. Segur, M. J. Ablowitz,
Asymptotic solutions of nonlinear evolution equations and a Painlev\'{e}
transcendent,
\textit{Phys. D}, 3 (1981), 165-184.

\bibitem{pa2}
P. Deift, X. Zhou,
A steepest descent method for oscillatory Riemann-Hilbert problems. Asymptotics
for the MKdV equation,
\textit{Ann. Math.}, 137 (1993), 295-368.


\bibitem{Charlier2020}
C. Charlier, J. Lenells, Airy and Painlev\'{e} asymptotics for the mKdV equation,
\textit{J.  Lond.  Math.  Soc.},   101  (2020),  194-225.



\bibitem{pa4}
L. Huang, L. Zhang, Higher order Airy and Painlev\'{e} asymptotics for the mKdV hierarchy,
\textit{SIAM J. Math. Anal.}, 54 (2022), 5291-5334.


\bibitem{pa5}
A. Boutet de Monvel, A. Its,  D. Shepelsky, Painlev\'{e}-type asymptotics for the Camassa-Holm
equation,
\textit{SIAM J. Math. Anal.}, 42 (2010), 1854-1873.


   \bibitem{Miller1} D. Bilman, L. M.  Ling,  P. D. Miller,  Extreme superposition: rogue waves of infinite order and the Painlev\'e-III hierarchy, Duke Math. J.,
  169 (2020), 671-760.


\bibitem{WF}   Z. Y. Wang,  E. G. Fan,  The defocusing nonlinear Schr\"odinger equation with a nonzero background: Painlev\'e asymptotics in two transition regions,
\textit{Commun. Math. Phys.},  {  402} (2023), 2879-2930.


\bibitem{xyz}
T. Y. Xu, Y. L. Yang, L. Zhang,
Transient asymptotics of the modified Camassa-Holm equation,
\textit{J. London Math. Soc.},   110 (2024), e12967.


\bibitem{sl1}
M. Bertola, A. Tovbis,
Universality in the profile of the semiclassical limit solutions to the
focusing nonlinear Schr\"odinger equation at the first breaking curve,
\textit{Int. Math. Res. Not.}, 2010 (2010), 2119-2167.


\bibitem{sll2}
M. Bertola, A. Tovbis,
Universality for the focusing nonlinear Schr\"odinger equation at the gradient catastrophe
point: rational breathers and poles of the tritronqu\'ee solution to Painlev\'e I,
\textit{Commun. Pure Appl. Math.}, 66 (2013), 678-752.


\bibitem{sll3}
T. Claeys, T. Grava,
Universality of the break-up profile for the KdV equation in the small dispersion limit using the Riemann-Hilbert approach,
\textit{Commun. Math. Phys.}, 286 (2009), 979-1009.


\bibitem{sl4}
T. Claeys, T. Grava,
Painlev\'e II asymptotics near the leading edge of the oscillatory zone for the Korteweg-de
Vries equation in the small dispersion limit,
\textit{Commun. Pure Appl. Math.}, 63 (2010), 203-232.

\bibitem{sl5}
B. Lu, P. Miller,
Universality near the gradient catastrophe point in the semiclassical sine-Gordon equation,
\textit{Commun. Pure Appl. Math.},  75 (2022), 1517-1641.







\bibitem{CJ} S. Cuccagna, R. Jenkins,
\newblock{On asymptotic stability $N$-solitons of the defocusing nonlinear Schr\"odinger equation}
\textit{Commun. Math. Phys.}, 343 (2016), 921-969.



\bibitem{RN10}
P. Deift, X. Zhou,
\newblock{Long-time asymptotics for solutions of the NLS equation with initial data in a weighted Sobolev space},
\textit{Commun. Pure Appl. Math.}, 56 (2003), 1029-1077.






\bibitem{BC1984}
R.  Beals, R. R.   Coifman,
\newblock {Scattering and inverse scattering for first-order systems}.
\textit{Commun. Pure Appl. Math.}, 37 (1984), 39-90.








\bibitem{Deiftzhoup2}
P. A. Deift, X. Zhou,
Asymptotics for the Painlev\'{e} \uppercase\expandafter{\romannumeral2}  equation,
\textit{Commun. Pure Appl. Math.}, 48 (1995), 277-337.

\bibitem{FokasAblop2}
A. S. Fokas, M. J. Ablowitz,
On the initial value problem of the second Painlev\'{e} transcendent,
\textit{Commun. Math. Phys.}, 91 (1983), 381-403.

    \end{thebibliography}
\end{document}